\documentclass[11pt]{amsart}
\usepackage{amsmath, amssymb, amsfonts, amsthm, xcolor}
\usepackage{scalerel}
\usepackage{stmaryrd}
\usepackage{mathabx}
\usepackage[numbers,comma,square,sort&compress]{natbib}
\usepackage{comment}
\definecolor{my-link}{rgb}{0.5,0.0,0.0}
\definecolor{my-blue}{rgb}{0.0,0.0,0.6}
\definecolor{my-red}{rgb}{0.5,0.0,0.0}
\definecolor{my-green}{rgb}{0.2,0.5,0.2}
\definecolor{darkgreen}{rgb}{0.0,0.5,0.0}
\definecolor{darkblue}{rgb}{0.0,0.0,0.3}
\definecolor{light-gray}{gray}{0.7}
\RequirePackage[colorlinks,urlcolor=my-blue,linkcolor=my-red,citecolor=my-green]{hyperref}

\usepackage{enumitem}

\usepackage[normalem]{ulem}

\usepackage{datetime}

\def\wh{\widehat} 
\def\wt{\widetilde}

\newcommand{\geo}{\Gamma}
\newcommand{\from}[1]{_{{#1}}}
\newcommand{\dir}[3]{^{{#1, #2#3}}}
\newcommand{\baddir}{\Theta^\w}
\newcommand{\tht}{\theta}
\newcommand{\sig}{{\raisebox{1.5pt}{\scaleobj{0.6}{\boxempty}}}}   
\newcommand{\ssig}{{\raisebox{0.8pt}{\scaleobj{0.6}{\boxempty}}}}   
\newcommand{\sigg}{{\raisebox{-0pt}{\scaleobj{1.1}{\boxempty}}}}  
\newcommand{\IG}[1]{{\mathcal S}^{#1}} 
\newcommand{\whIG}[1]{{\wh{\mathcal S}}^{#1}} 
\newcommand{\IGpns}[1]{{\mathcal S}^{#1}_{\mathrm{pns}}} 
\newcommand{\island}{\mathbb{I}^\tht}
\newcommand{\dust}{\mathcal D^\tht}
\newcommand{\Ipath}{\Psi}

\newcommand{\R}{\mathbb{R}}
\newcommand{\Z}{\mathbb{Z}}
\newcommand{\N}{\mathbb{N}}
\newcommand{\Q}{\mathbb{Q}}
\renewcommand{\P}{\mathbb{P}}

\newcommand{\ep}{\varepsilon}
\newcommand{\e}{\ep}

\newcommand{\NU}{{\mathrm{NU}}}

\newcommand{\w}{\omega}
\newcommand{\age}{{\mathfrak a}}

\newcommand\abullet{\hspace{0.6pt}{\raisebox{1pt}{\scaleobj{0.6}{\bullet}}}\hspace{0.8pt}}  

\newtheorem{thm}{Theorem}[section]
\newtheorem{lem}[thm]{Lemma}
\newtheorem{prop}[thm]{Proposition}
\newtheorem{cor}[thm]{Corollary}

\theoremstyle{definition}
\newtheorem{defn}[thm]{Definition}

\newtheorem{prob}{Open Problem}

\theoremstyle{remark}
\newtheorem{rmk}[thm]{Remark}

\numberwithin{figure}{section}
\numberwithin{equation}{section}

\usepackage[margin=1in]{geometry}

\title[Shocks, instability, and infinite geodesics in DL]{Shocks, instability, and the twenty networks\\
 of infinite geodesics in the Directed Landscape}

\author[F.~Rassoul-Agha]{Firas Rassoul-Agha}
\address{Firas Rassoul-Agha\\ University of Utah\\  Mathematics Department\\ 155S 1400E\\   Salt Lake City, UT 84112\\ USA.}
\email{firas@math.utah.edu}
\urladdr{http://www.math.utah.edu/~firas}
\thanks{F.\ Rassoul-Agha was partially supported by National Science Foundation grants DMS-2054630 and DMS-2450951}

\author[M.~Sweeney]{Mikhail Sweeney}
\address{Mikhail Sweeney\\ University of Utah\\  Mathematics Department\\ 155S 1400E\\   Salt Lake City, UT 84112\\ USA.}
\email{sweeney@math.utah.edu}
\thanks{M.\ Sweeney was partially supported by National Science Foundation grant DMS-2054630}

\date{\usdate\today}

\begin{document}

\begin{abstract}
For stochastic Hamilton-Jacobi (SHJ) equations, instability points are the space-time locations where two eternal solutions with the same asymptotic velocity differ. Another fundamental structure in such equations is shocks, which are the space-time locations where the velocity field is discontinuous. In this work, we study the KPZ fixed point, the central object of the KPZ universality class, which can be viewed as a prototype---albeit degenerate---of an inviscid SHJ equation in one spatial dimension. We describe the geometric structure of the instability region and give a detailed and precise analysis of its interplay with the shock structures of the two eternal solutions. We show that these shock structures allow one to reconstruct the instability region. Along the way, we obtain a complete classification of all possible configurations of semi-infinite geodesics emanating from arbitrary space-time points, in the directed landscape---the random environment in which the KPZ fixed point evolves.
\end{abstract}

\maketitle

\setcounter{tocdepth}{1}
\tableofcontents

\section{Introduction}

Stochastic Hamilton-Jacobi (SHJ) equations form a broad family of randomly forced partial differential equations, for which understanding the impact of noise on long-time behavior is a central problem. A phenomenon of particular interest is \emph{stochastic synchronization}, originating in random dynamical systems and closely tied to the uniqueness of \emph{eternal solutions} corresponding to a given value of the conserved asymptotic velocity. This question has been studied for the Kardar-Parisi-Zhang (KPZ) equation, as well as for several related random polymer models in the KPZ universality class, in both zero- and positive-temperature settings, where it has been shown that while eternal solutions are almost surely unique for deterministic velocities, there almost surely exist exceptional velocities for which uniqueness breaks down. This raises the problem of describing the resulting instability regions---space-time locations where distinct eternal solutions differ. 
At the same time, inviscid SHJ equations are known to exhibit \emph{shock tree} structures. This leads naturally to the question of how instability regions are related to these shock trees. In this paper, we study the KPZ fixed point, a degenerate inviscid SHJ central to the KPZ universality class, and describe the geometry of its instability regions. We show that these regions are nowhere dense, have no isolated points, and form path-connected graphs (Figure \ref{fig:IG}). We further show that a distinguished portion of these regions can be reconstructed from the shock structures of the associated eternal solutions, revealing a precise connection between shocks and instability. As a consequence, we also obtain a complete description of all possible configurations of semi-infinite geodesics in the directed landscape (Figure \ref{fig:geodesics}).

We next briefly describe the KPZ fixed point and the associated directed landscape model, place them within the framework of stochastic Hamilton-Jacobi equations, and formulate our questions in that setting before summarizing our answers for the directed landscape.

\subsection{The KPZ fixed point and the directed landscape}
The \emph{KPZ equation} is the non-linear stochastic partial differential equation (SPDE)
\begin{align}\label{KPZ}
\partial_t h = \tfrac12 (\partial_x h)^2 + \tfrac{\nu}{2}\partial_{xx}h + \beta W,
\end{align}
where $W$ denotes space-time white noise, $\beta\in\R$ modulates the strength of the noise, and $\nu>0$ is a \emph{viscosity} parameter. Introduced in 1986 by Kardar, Parisi, and Zhang \cite{Kar-Par-Zha-86} as a model for the evolution of a one-dimensional random interface $h(x,t)$, its well-posedness was established more recently in \cite{Per-Ros-19}. 

It was shown in \cite{Vir-20-} that, when $\nu=\beta=1$, and $C$ is an appropriately chosen constant, the centered and rescaled process
$\e^{1/2}\bigl(h(\e^{-1}x,\e^{-3/2}t)-C\,\e^{-3/2}t\bigr)$
converges, as $\e\to0$, to a universal limit $\mathfrak h(t,x)$, known as the \emph{KPZ fixed point}. This limit was first constructed in \cite{Mat-Qua-Rem-21} via an explicit description of its transition probabilities in terms of Fredholm determinants. By the scaling properties of the KPZ equation (see, e.g., Remark 1.1 in \cite{Jan-Ras-Sep-23-1F1S-}), the rescaled field $(x,t)\mapsto \e^{1/2}h(\e^{-1}x,\e^{-3/2}t)$ solves \eqref{KPZ} with parameters $\nu=\e^{1/2}$ and $\beta=\e^{1/4}$. The KPZ fixed point may therefore be viewed as an inviscid limit of the KPZ equation.

Analogous to Brownian motion, the KPZ fixed point is expected to be the universal scaling limit of a broad class of planar random growth models, collectively referred to as the \emph{KPZ universality class}. This class is believed to encompass both viscous and inviscid SHJ equations, interacting particle systems, percolation and growth models, random polymer measures, driven diffusive systems, and random matrices. 
Convergence to the KPZ fixed point has been established for several models in 
\cite{Wu-23-,Agg-Cor-Heg-24-a-,Agg-Cor-Heg-24-b-,Wei-Fer-Spo-17,Vir-20-,Dau-Vir-21-,Zha-25-,Vet-Vir-26}.

A geometric interpretation of the KPZ fixed point is provided by the \emph{directed landscape}, introduced in \cite{Dau-Ort-Vir-22}. This is a random continuous process $\{\mathcal L(x,s;y,t):x,s,y,t\in\R,\ s<t\}$ that serves as the environment in which the KPZ fixed point evolves: the fixed point with a given terminal condition $\varphi$ at time $t$ is given by the variational formula
\begin{align}\label{KPZfp}
\mathfrak h(x,s)=\sup_{y\in\R}\bigl\{\mathcal L(x,s;y,t)+\varphi(y)\bigr\},\qquad s<t.
\end{align}

The process $\mathcal L$ admits a pathwise variational characterization. Given a continuous space-time path $\gamma:[s,t]\to\R^2$ with $\gamma(s)=(x,s)$ and $\gamma(t)=(y,t)$, define its \emph{passage time} by
\[
\mathcal L(\gamma)=\inf_{k\in\N}\inf_{s=r_0<\cdots<r_k=t}\sum_{i=1}^k \mathcal L\bigl(\gamma(r_{i-1});\gamma(r_i)\bigr).
\]
Then
\begin{align}\label{DL-LPP}
\mathcal L(x,s;y,t)=\sup_\gamma \mathcal L(\gamma),
\end{align}
where the supremum is taken over all such paths $\gamma$. Thus, $\mathcal L(x,s;y,t)$ can be viewed as a continuum point-to-point last-passage time, while the KPZ fixed point corresponds to a point-to-line last-passage time with terminal weights $(y,t)\mapsto\varphi(y)$.

We remark that, as we find it more natural to interpret the directed landscape as a percolation model with geodesics (maximizers in \eqref{DL-LPP}) evolving forward in time, the KPZ fixed point in \eqref{KPZfp} is correspondingly formulated to evolve backward in time.
	
\subsection{Stochastic Hamilton-Jacobi equations}\label{sec:SHJ}
The KPZ equation \eqref{KPZ} is a prominent example of a broader class of randomly forced partial differential equations, namely one-dimensional stochastic Hamilton-Jacobi (SHJ) equations:
\[\partial_t \Phi + H(\nabla \Phi) = \nu \Delta \Phi - F,\]
where $\Phi:\R\times\R\to\R$, the gradient $\partial_x\Phi$ represents momentum, $H:\R\to\R$ is a convex Hamiltonian, and $F=F_\w(x,t)$ is a random potential (with forcing $f=-\partial_x F$). The parameter $\nu\ge0$ controls viscosity, with $\nu>0$ corresponding to the viscous regime and $\nu=0$ to the inviscid one.
The scaling limit described above identifies the KPZ fixed point as a degenerate inviscid SHJ equation.

In the viscous case, solutions are unique and smooth for appropriate initial data. In contrast, the inviscid equation admits multiple solutions with discontinuities in the gradient, known as \emph{shocks}, which play a central role in Hamilton-Jacobi theory and in modeling nonlinear wave phenomena. In this case, the relevant notion of solution is the \emph{viscosity solution}, obtained as the limit of viscous solutions as $\nu\to0$. This solution admits the Hopf-Lax-Oleinik variational representation: for initial data $\varphi$ at time $s$,
\begin{align}\label{LO}
\Phi(x,t)=\inf_{\gamma:\gamma(t)=x}\Bigl\{\varphi(\gamma(s))+\int_s^t L_\w(\gamma'(r),\gamma(r),r)\,dr\Bigr\},
\end{align}
where $L_\w$ is the convex dual of $H_\w(p,x,t)=H(p)+F_\w(x,t)$, and the infimum is taken over absolutely continuous paths $\gamma$. Minimizing paths are called \emph{characteristics} or \emph{geodesics}, and shock points are those from which multiple distinct such paths emanate.

We remark here that combining \eqref{KPZfp} and \eqref{DL-LPP} yields a Hopf-Lax-Oleinik representation for the KPZ fixed point, with $\mathcal L(\gamma)$ playing the role of the path action functional $\int_s^t L_\w(\gamma'(r),\gamma(r),r)\,dr$. The appearance of a supremum instead of an infimum is a matter of sign convention and can be removed by considering $-\mathfrak h$ and $-\mathcal L$.

If the potential $F$ is white in time, then the representation \eqref{LO} endows the SHJ equation with a Markov structure, allowing it to be viewed as a random dynamical system (RDS) \cite{Arn-98}. A central concept in this framework is the one force--one solution (1F1S) principle \cite{E-etal-00}, which asserts the existence of a unique \emph{eternal solution}, defined for all time and progressively measurable with respect to the noise. When this holds, solutions with different initial data converge over time, a phenomenon known as \emph{stochastic synchronization}.

In one dimension, the asymptotic velocities $\lim_{x\to\pm\infty}x^{-1}\Phi(x,t)$ are conserved, and the 1F1S principle is typically formulated for fixed values of this quantity. While it is expected to hold for deterministic velocities, it may fail for exceptional values, leading to multiple eternal solutions and hence instability. Then the instability regions are the space-time points where these solutions differ.

This phenomenon arises in both viscous and inviscid settings, whereas shocks are intrinsic to the inviscid regime. Our interest is in understanding how the geometry of shocks relates to the structure of instability, and we therefore focus on the inviscid case.

The 1F1S principle has been established for a range of models, from compact or effectively compact settings with regular forcing to noncompact settings with discrete or semidiscrete forcing and the KPZ equation on the torus; see the end of Section 1.1 in \cite{Ras-Swe-24-a-} for references.
Instability for exceptional velocities was first demonstrated in discrete SHJ models \cite{Sep-Sor-23-pmp,Jan-Ras-Sep-23}, and subsequently in the fully continuous setting for both the KPZ equation \cite{Jan-Ras-Sep-23-1F1S-} and its inviscid counterpart, the KPZ fixed point \cite{Bus-Sep-Sor-24}.
For further discussion in the settings of discrete last-passage percolation and the KPZ equation, see Section 4.1 of \cite{Jan-Ras-Sep-23} and Section 3.5 of \cite{Jan-Ras-Sep-23-1F1S-}, respectively. This instability can also be interpreted as a \emph{phase transition} in the familiar setting of Gibbs measures. See Section 2.4 in \cite{Jan-Ras-18-arxiv} and the introduction and appendix B in \cite{Gro-Jan-Ras-25-} for details in the case of random polymers.

Within this broader SHJ framework, our interest centers on the following fundamental questions:

\begin{enumerate}
\item Do instability regions exhibit any coherent geometry? For instance, are they bounded or unbounded, sparse or dense, connected or disconnected?
\item Shocks and instability both represent forms of irregularity for solutions of the inviscid SHJ equation. Can a point exhibit both simultaneously? More generally, what is the relationship between these two sets of points, and to what extent can one be used to predict the other?
\item Can one classify the geometric structure of configurations of semi-infinite characteristics of eternal solutions?
\end{enumerate}

\subsection{Our contributions}
Although the questions above are of broad interest for general inviscid stochastic Hamilton-Jacobi equations, in this paper we address them for the KPZ fixed point. This is motivated in part by the fact that the KPZ fixed point is currently the only space-time continuum model in which instability has been rigorously established. 

Other models where these questions have been answered are directed last-passage percolation on $\Z^2$ with exponential weights and the Brownian last-passage percolation (BLPP) model. The former is fully discrete and therefore does not exhibit shocks. In contrast, BLPP is a discrete-time, continuous-space model, and hence both shocks and instability appear in it. However, the shock structures we observe in the directed landscape are significantly richer. Given the central role of the directed landscape and the KPZ fixed point within the KPZ universality class, we expect the phenomena uncovered here to be representative of a broad class of inviscid SHJ equations.

In light of the relationship between shocks and instability, a pertinent question for future work is whether instability points are associated with a notion of energy dissipation, analogous to the dissipation that occurs at shocks.

Our main results are presented in Section \ref{sec:setting} and summarized in Figures \ref{fig:geodesics}--\ref{fig:island-complete}. Figure \ref{fig:IG} displays a simulation of the instability region, a closed, fractal, path-connected, nowhere dense set with no isolated points, thereby addressing the first question above. Figure \ref{fig:geodesics} answers the third question by exhibiting all possible configurations of semi-infinite geodesics in the directed landscape, and in particular classifies all possible types of shocks. 
Finally, Figure \ref{fig:island-complete} locates these configurations within the instability region, addressing one direction of the second question. It also highlights a distinguished subset, namely the boundaries of ``stability islands''. We show that the union of these (countably many) boundaries is dense in the instability set, so that its closure coincides with the instability set. We therefore refer to this subset as the \emph{skeleton} of the instability graph. In particular, the instability region can be recovered from the shock trees of the two eternal solutions, thereby addressing the other direction of the second question.

At the technical level, our methods differ from those in \cite{Jan-Ras-Sep-23, Ras-Swe-24-a-} in that the directed landscape is continuous in both space and time, and its geodesics and shock interfaces are continuous. This invalidates several arguments used in those works, where one or both variables are discrete and the path structure is correspondingly simpler. Our approach to classifying semi-infinite geodesic configurations also differs from the classification of finite geodesic networks in \cite{Dau-25}, as we work directly with infinite objects---Busemann functions and their associated geodesics. While developed independently, our methods may be viewed as a more detailed counterpart to those in \cite{Dau-Pan-26-}, extending the classification beyond deterministic time levels.

\subsection{Notation}
$\Z$ is the set of integers, $\N$ is the set of positive integers, $\R$ is the set of real numbers, and $\Q$ is the set of rationals.
Given real numbers $x,s,y,t$, we use $s\wedge t=\min(s,t)$, $s\vee t=\max(s,t)$, and write $(x,s) \le (y,t)$ to mean $x \le y$ and $s \le t$. 
Then $(x,s) < (y,s)$ means $x < y$. 
Given an interval $I\subset\R$, a \emph{space-time path} is a  function $\gamma:I\to\R^2$ such that for all $t\in I$, the second coordinate of $\gamma(t)$ is equal to $t$. 
It will at times be convenient to abuse notation and regard $\gamma$
as the path $\gamma(I)$. Given a subinterval $J\subset I$, $\gamma|_J$ is the restriction of $\gamma$ to $J$.
Given two intervals $I_i\subset\R$ and two space-time paths $\gamma_i:I_i\to\R^2$, $i\in\{1,2\}$, $\gamma_1\preceq\gamma_2$, respectively $\gamma_1\prec\gamma_2$, means $\gamma_1(t)\le\gamma_2(t)$, respectively $\gamma_1(t)<\gamma_2(t)$, for all $t\in I_1\cap I_2$.

\section{Setting and main results}\label{sec:setting}

The directed landscape, constructed in \cite{Dau-Ort-Vir-22}, is a continuous stochastic process of nonnegative passage times 
$\{\mathcal L(x,s;y,t):x,y,s,t\in\R,s<t\}$, defined on a Polish probability space $(\Omega,\mathfrak S,\P)$ and satisfying the composition law 
	\begin{align}\label{composition}
    \mathcal L(x,s;y,t)=\sup_{z\in\R}\bigl(\mathcal L(x,s;z,r)+\mathcal L(z,r;y,t)\bigr).
    \end{align}

\begin{figure}[hpt]
    \includegraphics[width=12cm]{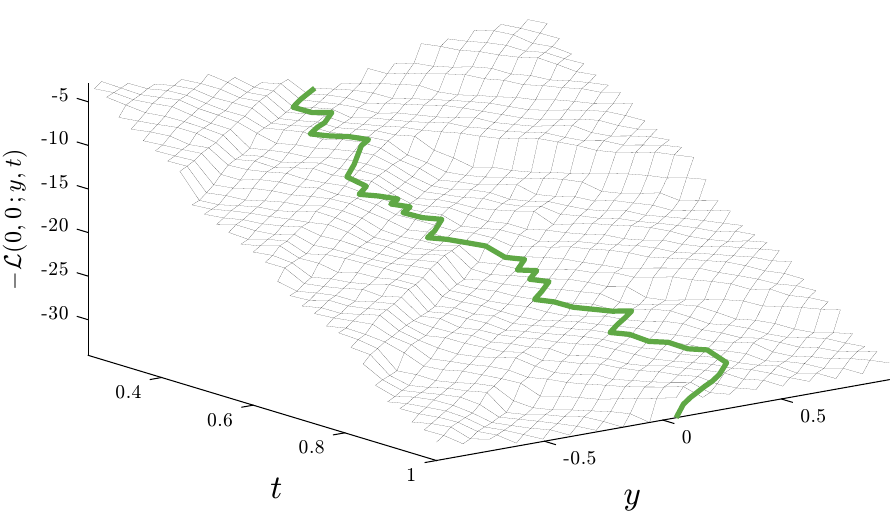}
    \caption{\small A simulation of the directed landscape map $(x,s)\mapsto -\mathcal L(0,0\,;x,s)$. The path depicts a geodesic, that is, a path that maximizes $\mathcal L$ (equivalently minimizes $-\mathcal L$) between its endpoints.} 
    \label{fig:DL}
\end{figure}
    
Given a terminal condition $\varphi$ at time $t$, the KPZ fixed point is given by 
     \begin{align}\label{fixedpt}
     \mathfrak h(x,s)=\sup_{y\in\R}\bigl(\mathcal L(x,s;y,t)+\varphi(y)\bigr),\quad x\in\R,\ s<t.
     \end{align}
The composition law \eqref{composition} ensures that this gives a Markov process (backward in time).  
    
\subsection{Infinite geodesics}
Given a continuous space-time path $\gamma:[s,t]\to\R^2$, define the passage time along $\gamma$ by
\begin{align}\label{L(gamma)}
\mathcal L(\gamma)
=\inf_{k\in\N}\;\inf_{s=r_0<r_1<\cdots<r_{k-1}<r_k=t}
\sum_{i=1}^k \mathcal L\bigl(\gamma(r_{i-1})\,;\gamma(r_i)\bigr).
\end{align}
Then
\begin{align}\label{LPP}
\mathcal L(x,s\,;y,t)= \sup_{\gamma}\mathcal L(\gamma),
\end{align}
where the supremum is taken over all continuous space-time paths $\gamma:[s,t]\to\R^2$ with $\gamma(s)=(x,s)$ and $\gamma(t)=(y,t)$. This formulation interprets the directed landscape as a continuous space-time last-passage percolation model.

Accordingly, a path that achieves the supremum in \eqref{LPP} is called a \emph{point-to-point geodesic} from $(x,s)$ to $(y,t)$. See Figure \ref{fig:DL} for an illustration. Our interest lies in the large-scale structure of the directed landscape and its geodesic paths. As the time interval between $s$ and $t$ increases, this leads to the consideration of infinite geodesics. A \emph{semi-infinite geodesic} from $(x,s)\in\R^2$ is a continuous space-time path $\gamma:[s,\infty)\to\R^2$ with $\gamma(s)=(x,s)$, such that $\gamma$ is a geodesic between any two of its points. A \emph{bi-infinite geodesic} is defined analogously as a continuous space-time path $\gamma:\R \to \R^2$ with the same property. Proposition 34 of \cite{Bha-24} shows that bi-infinite geodesics do not exist in the directed landscape. 

Lemma 13.2 in \cite{Dau-Ort-Vir-22} 
asserts that, almost surely, for any $x,s,y,t\in\R$ with $s<t$, there exists at least one geodesic from $(x,s)$ to $(y,t)$. Furthermore, by Theorem 12.1 in that paper, for any given such space-time points, there exists almost surely a unique geodesic path from $(x,s)$ to $(y,t)$. By Fubini's theorem, this means that almost surely, unique geodesics exist between Lebesgue-almost every pair of points $(x,s)$ and $(y,t)$ (with $t>s$). However, Theorem 1.10 in \cite{Bat-Gan-Ham-22} shows that there exist exceptional points between which there are multiple geodesics (see also \cite{Gan-Zha-22-}). \cite{Dau-25} describes all the possible configurations of such point-to-point geodesics. One goal of our work is to give an analogous description of all the possible configurations of semi-infinite geodesics. This is stated as Theorem \ref{main:geodesics} below.

Theorem 2.5(i) of \cite{Bus-Sep-Sor-24} asserts that, almost surely, every semi-infinite geodesic $\gamma$ is $\tht$-directed for some $\tht\in\R$, that is, $t^{-1}\gamma(t)\to(\tht,1)$ as $t\to\infty$.
Combining results from \cite{Bus-Sep-Sor-24,Bus-25-,Dau-25}, one can organize the family of all semi-infinite geodesics into a process
\begin{align}\label{geodesics.intro}
\bigl\{\geo\from{(x,s)}\dir{S}{\tht}{\sig}:(x,s)\in\R^2,\,S\in\{L,M,R\},\,\tht\in\R,\,\sigg\in\{-,+\}\bigr\}.
\end{align}
Almost surely, for any $x,s,\tht\in\R$, the geodesics 
$\geo\from{(x,s)}\dir{L}{\tht}{-}$ and $\geo\from{(x,s)}\dir{R}{\tht}{+}$ 
are, respectively, the leftmost and rightmost $\tht$-directed geodesics from $(x,s)$. Moreover, for each $\sigg\in\{-,+\}$, $\geo\from{(x,s)}\dir{L}{\tht}{\sig}
\preceq\geo\from{(x,s)}\dir{M}{\tht}{\sig}
\preceq\geo\from{(x,s)}\dir{R}{\tht}{\sig}$,
and
\begin{align}\label{geoorder}
\geo\from{(x,s)}\dir{S}{\tht}{-}\preceq\geo\from{(x,s)}\dir{S}{\tht}{+},
\qquad S\in\{L,M,R\}.
\end{align}
Depending on the configuration, the three geodesics $\geo\from{(x,s)}\dir{S}{\tht}{\sig}$, $S\in\{L,M,R\}$, may coincide, only two may be distinct, or all three may be distinct. All $\tht-$ geodesics coalesce, as do all $\tht+$ geodesics.
Furthermore, there exists an $\w$-dependent countable dense set $\baddir\subset\R$ such that if $\tht\notin\baddir$ there is no sign distinction, whereas for $\tht\in\baddir$ the $\tht-$ and $\tht+$ geodesics eventually diverge. See Section \ref{sec:Busgeo} for precise statements and additional properties of the geodesic process used in this work.

It should be emphasized that the existence of the set $\baddir$ highlights that the sign distinction is not merely a technicality but is essential to the structures analyzed in this paper. Likewise, we see in Section \ref{sec:shocks.intro} that the distinction between leftmost and rightmost semi-infinite geodesics is both necessary and central to these structures.

Our first main theorem describes all the possible configurations of semi-infinite geodesics out of a point in $\R^2$.

\begin{thm}\label{main:geodesics}
The following hold for all $\w$ in a full $\P$-probability event. 
\begin{enumerate}[label={\rm(\alph*)}, ref={\rm\alph*}] \itemsep=1pt
\item\label{geodesics.a} For each $\tht\in\R$ and $(x,s)\in\R^2$, the $\tht$-directed geodesics from $(x,s)$ realize exactly one of the twenty configurations shown in Figure \ref{fig:geodesics}.
\item\label{geodesics.b} If $\tht\notin\baddir$, each of the first four configurations in the top row of Figure \ref{fig:geodesics} occurs for a dense set of starting points in $\R^2$.
\item\label{geodesics.c} If $\tht\in\baddir$, each of the last four configurations in the top row of Figure \ref{fig:geodesics} occurs for a dense set of starting points in $\R^2$. Furthermore, each of the twelve configurations on the next three rows arises from infinitely many starting points. See Theorems \ref{main:IGgeods} and \ref{main:shocks} for further details on the location and density of their occurrence.
\end{enumerate}
\end{thm}

\begin{figure}[hpt]
    \includegraphics[height=2cm]{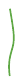}\ \ \ \ \ \ 
    \includegraphics[height=2cm]{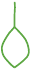}\ \ \ \ \ 
    \includegraphics[height=2cm]{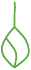}\ \ \ \ \ 
    \includegraphics[height=2cm]{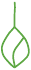}\ \ \ \ \ 
    \includegraphics[height=2cm]{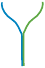}\ \ \ \ \ 
    \includegraphics[height=2cm]{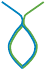}\ \ \ \ \ 
    \includegraphics[height=2cm]{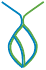}\ \ \ \ \ 
    \includegraphics[height=2cm]{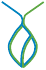}\smallskip

    \includegraphics[width=1.7cm]{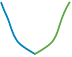}\ \ \ \ \ \ 
    \includegraphics[width=1.7cm]{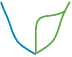}\ \ \ \ \ 
    \includegraphics[width=1.7cm]{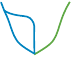}\ \ \ \ \ 
    \includegraphics[width=1.7cm]{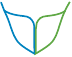}\smallskip
    
    \includegraphics[width=1.7cm]{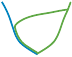}\ \ \ \ \ \ 
    \includegraphics[width=1.7cm]{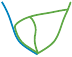}\ \ \ \ \ \ 
    \includegraphics[width=1.7cm]{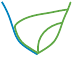}\ \ \ \ \ \ 
    \includegraphics[width=1.7cm]{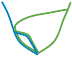} 
    \smallskip

    \includegraphics[width=1.7cm]{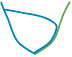}\ \ \ \ \ \ 
    \includegraphics[width=1.7cm]{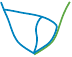}\ \ \ \ \ \ 
    \includegraphics[width=1.7cm]{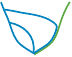}\ \ \ \ \ \ 
    \includegraphics[width=1.7cm]{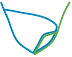}

    \caption{\small All 20 possible configurations of $\tht$-directed geodesics out of a point $p$. The first four configurations are when $\tht\notin\baddir$ and hence there is no sign distinction. For the rest, $\tht\in\baddir$ and then $W^{\tht+}$-geodesics are green and $W^{\tht-}$-geodesics are blue. The last four configurations on the first row are for the case $p\notin\IG\tht$. The four configurations on the second row are when $p\in\IG\tht$ is neither right-isolated nor left-isolated (among instability points on its time level). The first three are configurations that occur when $p$ is a \emph{dust} instability point, i.e.\  not on the boundary of any stability island. The first one is called a \emph{proper non-shock instability} (pns) point. It is also the configuration at tips of islands. The last one is the \emph{snowbird shock}, which occurs exactly at the bottom point of each island. The third row depicts configurations when $p$ is left-isolated. This is when $p$ is on the right boundary of an island and is then a $\tht+$ \emph{hugging shock}. The last row is configurations where $p$ is right-isolated and is hence on the left boundary of an island and is a $\tht-$ \emph{hugging shock}.}
    \label{fig:geodesics}
\end{figure}

The statements of Theorem \ref{main:geodesics} are formulated more precisely in Propositions \ref{prop:geodconfigstable}-\ref{prop:islandconfigs}, whose combined conclusions yield the theorem.

Lemma 8.1 of \cite{Dau-Pan-26-} establishes our theorem for time $s=0$. By the shift invariance in \cite[Lemma 10.2]{Dau-Ort-Vir-22}, this extends to all rational times $s$ simultaneously. For such times, the number of possible configurations reduces to seven. These are precisely the configurations among our twenty in which no three geodesics emanating from the starting point are initially distinct: namely, the first two and the fifth and sixth configurations in the top row, together with the first configuration in each of the second, third, and bottom rows. Our theorem shows that thirteen additional configurations can occur at exceptional times.

\subsection{Stability and instability} 
Recall that for each $\tht\in\R$ and $\sigg\in\{-,+\}$, the $\tht\sigg$ semi-infinite geodesics in \eqref{geodesics.intro} coalesce. Given two points $(x,s)$ and $(y,t)$ in $\R^2$, choose one such geodesic from each point and let $(z,r)$ denote their coalescence point. Define
\begin{align}\label{Wdef}
W^{\tht\sig}(x,s;y,t)=\mathcal L(x,s;z,r)-\mathcal L(y,t;z,r).
\end{align}
By coalescence, this definition does not depend on the choice of $S\in\{L,M,R\}$ for the two $\tht\sigg$ geodesics, nor on the particular point $(z,r)$ along their common tail. The function $W^{\tht\sig}$ is called the \emph{Busemann function}, borrowing terminology from metric geometry. See Section \ref{sec:Busproc} for precise statements and a summary of the properties of these functions used in this paper.

Using \eqref{composition} together with \eqref{Wdef}, one obtains that, almost surely, for all $x,s,x_0,s_0,\tht\in\R$, $t>s$, and $\sigg\in\{-,+\}$,
\begin{align}\label{update0}
W^{\tht\sig}(x,s;x_0,s_0)=\sup_{y\in\R}\{\mathcal L(x,s;y,t)+W^{\tht\sig}(y,t;x_0,s_0)\}
\end{align}
and, for any $S\in\{L,M,R\}$, $y=\geo\from{(x,s)}\dir{S}{\tht}{\sig}(t)$ maximizes the above supremum.
Thus $W^{\tht\sig}$ solves the KPZ fixed point update equation \eqref{fixedpt} for all times $t\in\R$. In the terminology of PDEs, such a solution is called \emph{eternal}.

As discussed in Section \ref{sec:SHJ}, the KPZ fixed point can be viewed as a degenerate inviscid stochastic Hamilton-Jacobi equation, with its Hopf-Lax-Oleinik representation obtained by combining \eqref{fixedpt} and \eqref{LPP}. In this analogy, point-to-point geodesics play the role of characteristic lines. Accordingly, the semi-infinite geodesics $\geo\from{(x,s)}\dir{S}{\tht}{\sig}$ may be interpreted as characteristic curves of $W^{\tht\sig}$ traced backward from $(x,s)$ into the distant past (recall that in this analogy time runs backward, while geodesics evolve forward).

The $\tht$-directedness of the semi-infinite geodesics translates into a spatial growth rate for $W^{\tht\sig}$. Namely, for any $x_0,s_0,s\in\R$,
$x^{-1}W^{\tht\sig}(x,s;x_0,s_0)\to 2\tht$
as $|x|\to\infty$
\cite[Lemma 5.12(iv)]{Bus-Sep-Sor-24}. Thus, this eternal solution has a conserved asymptotic velocity equal to $2\tht$.

Recall that if $\tht\notin\baddir$, the sign distinction in the semi-infinite geodesics disappears. In this case $W^{\tht-}=W^{\tht+}$, yielding a unique eternal solution with asymptotic velocity $2\tht$. As explained in Section \ref{sec:SHJ}, this corresponds to stability and stochastic synchronization for solutions initiated in the basin of attraction of the eternal solution.

In contrast, when $\tht\in\baddir$, the non-coalescence of the $\tht-$ and $\tht+$ geodesics implies that $W^{\tht-}\ne W^{\tht+}$, giving rise to two distinct eternal solutions with the same velocity $2\tht$. In this regime, the system exhibits multiple pullback attractors, leading to instability and the breakdown of stochastic synchronization. We therefore refer to directions $\tht\in\baddir$ as \emph{directions of instability}.

Motivated by the above, we analyze the region $\IG\tht$ defined via
\begin{align}\label{IGdef1}
\R^2\setminus\IG\tht=\bigl\{(x,s)\in\R^2:\exists\text{ open }O\ni(x,s)\ \text{s.t. }W^{\tht-}=W^{\tht+}\text{ on }O\times O\bigr\}.
\end{align}
We call $\IG\tht$ the \emph{instability graph} (or \emph{web of instability}) because of its graph-like structure. 
Our next main theorem, proved in Section \ref{sec:mainproofs}, summarizes the  structural properties of this set. See also Figure \ref{fig:IG} for a simulation of $\IG\tht$.

\begin{thm}\label{main:IG}
The following hold for all $\w$ in a full $\P$-probability event and for all $\tht\in\baddir$.
\begin{enumerate}[label={\rm(\alph*)}, ref={\rm\alph*}] \itemsep=1pt

\item\label{main:IG.closed}
$\IG\tht$ is closed, nowhere dense, and has no isolated points.

\item\label{main:IG.islands}
The connected components of the open set $\R^2\setminus\IG\tht$ are all bounded, and distinct components have disjoint closures. Moreover, there are countably infinitely many such components. We call these components \emph{stability islands}.

\item\label{main:IG.boundary} Each island has a unique boundary point with minimal time, called the \emph{bottom}, and a unique boundary point with maximal time, called the \emph{tip}. The boundary of the island consists of two continuous space-time paths, the \emph{left boundary} and the \emph{right boundary}, each connecting the bottom to the tip. Except at the bottom and the tip, the left boundary lies strictly to the left of the right boundary.

\item\label{main:IG.islandsdense} The islands are dense in $\IG\tht$: for any $(x,s)\in\IG\tht$ and $\rho>0$, there exists an island whose closure is inside the ball of radius $\rho$ and center $(x,s)$.\vspace{2pt}

Let $\dust$ denote the subset of $\IG\tht$ consisting of points that are not on the boundary of any stability island. Points in $\dust$ are called \emph{dust points}.\vspace{4pt}

\item\label{main:IG.tbi}
$\dust$ {\rm(}and hence also $\IG\tht${\rm)} is temporally bi-infinite: there exists an uncountable family of continuous bi-infinite space-time paths $\tau_\alpha:\R\to\R^2$ such that $\tau_\alpha(\R)\subset\dust$ for all $\alpha$, the paths are pairwise disjoint
{\rm(}$\tau_\alpha(\R)\cap\tau_\beta(\R)=\varnothing$ for $\alpha\ne\beta${\rm)}, and each path is $\tht$-directed in both time directions, i.e.\ $t^{-1}\tau_\alpha(t)\to(\tht,1)$ as $|t|\to\infty$.

\item\label{main:IG.xbi} For any $s\in\R$, the set $\dust\cap(\R\times\{s\})$ has no isolated points and is unbounded above and below. Moreover, for any $a<b$ such that $\IG\theta\cap((a,b)\times\{s\})\neq\varnothing$, the set $\dust\cap((a,b)\times\{s\})$ is uncountable.

\item\label{main:IG.directed} For each $(x,s)\in\IG\tht$, there exists at least one continuous bi-infinite space-time path through $(x,s)$ that remains on $\IG\tht$ and is $\tht$-directed in both time directions.

\item\label{main:IG.graph}
$\IG\tht$ forms a path-connected graph: for any $(x,s),(y,t)\in\IG\tht$, there exist $z,z'\in\R$ and times $r<s\wedge t\le s\vee t<r'$ such that $(z,r),(z',r')\in\IG\tht$, and there are continuous space-time paths contained in $\IG\tht$ from $(z,r)$ to each of $(x,s)$ and $(y,t)$, and from each of these two points to $(z',r')$.
\end{enumerate}
\end{thm}

\begin{figure}[hpt]
    \includegraphics[width=10cm]{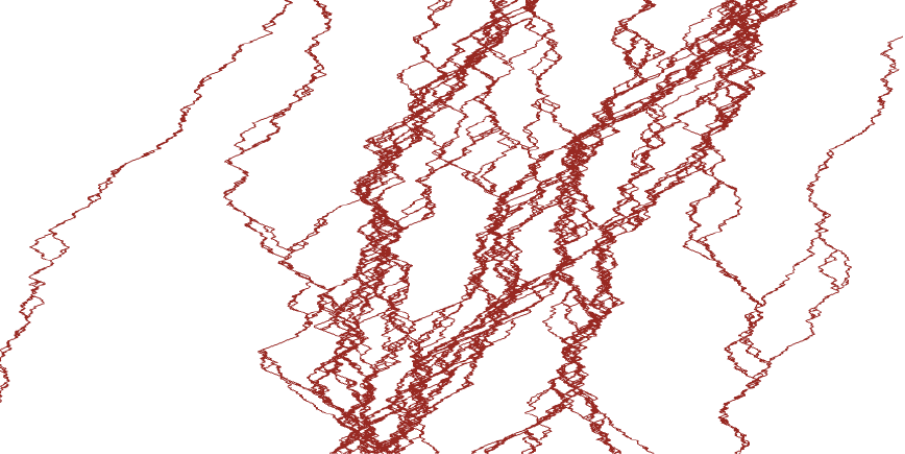}
    \caption{\small A simulation of a portion of the \emph{instability graph} $\IG\tht$. The set is a closed, nowhere dense, connected graph (Theorem \ref{main:IG}\eqref{main:IG.graph}) with no isolated points, and is bi-infinite in both space and time. It consists of the boundaries of countably many bounded \emph{stability islands} with disjoint closures, together with the remaining \emph{dust} instability points that fill the regions between islands and make up the bulk of $\IG\tht$. The set exhibits a fractal structure: islands occur at all scales and are dense in $\IG\tht$ (Theorem \ref{main:IG}\eqref{main:IG.islandsdense}).}
    \label{fig:IG}
\end{figure}

Proposition \ref{prop:IGgoup} has more details about the paths in part \eqref{main:IG.tbi}.
The next theorem complements Theorem \ref{main:geodesics}. Its proof appears in Section \ref{sec:mainproofs}.

\begin{thm}\label{main:IGgeods} The following holds for all $\w$ in a full $\P$-probability event and any $\tht\in\baddir$. 
\begin{enumerate}[label={\rm(\alph*)}, ref={\rm\alph*}] \itemsep=1pt
\item\label{IGgeods.a} The last four configurations on the top row of Figure \ref{fig:geodesics} only occur at stability points, i.e.\ points in $\R^2\setminus\IG\tht$.
\item\label{IGgeods.b} The twelve configurations appearing on the second, third, and fourth rows of Figure \ref{fig:geodesics} only occur at instability points, i.e.\ points in $\IG\tht$. Furthermore, each of them arises from a set that is dense in $\IG\tht$. For more details on the exact locations of occurence, see Theorem \ref{main:shocks}.
\end{enumerate}
\end{thm}

We leave two open problems for future work.

\begin{prob}\label{prob1}
Is it the case that, almost surely, for every $\tht\in\baddir$ and every $s\in\R$, the set $\dust\cap(\R\times\{s\})$ (and hence also $\IG\tht\cap(\R\times\{s\})$) is nowhere dense?
\end{prob}

\begin{prob}\label{prob2}
Is it the case that, almost surely, for every $\tht\in\baddir$, the set $\IG\tht$ is $\tht$-directed? That is, for any continuous space-time path $\gamma:[s,\infty)\to\R^2$ with $\gamma([s,\infty)])\subset\IG\tht$, one has
$t^{-1}\gamma(t)\to(\tht,1)$ as $t\to\infty$.
\end{prob}

Both problems admit affirmative answers in directed last-passage percolation with exponential weights and in the Brownian last-passage percolation model \cite{Jan-Ras-Sep-23,Ras-Swe-24-a-}, where the path structure is simpler. The two questions are related to the following more basic problem. 

\begin{prob}\label{prob3}
Consider the event that there exist $y>x$ and $t\ge s>0$ in $\R$, and a continuous function $\gamma:[x,y]\to\R^2$ such that $\gamma(x)=(x,s)$, $\gamma(y)=(y,t)$, and the first coordinate of $\gamma(z)$ is $z$ for all $z\in[x,y]$. Moreover, for every $z\in[x,y]$, there exist geodesics from $(-1,0)$ to $\gamma(z)$ and from $(1,0)$ to $\gamma(z)$ that intersect only at $\gamma(z)$ and lie entirely below $\gamma$. See Figure \ref{fig:open} for an illustration. Does this event have zero probability?
\end{prob}

\begin{figure}[hpt]
    \includegraphics[width=2.75cm]{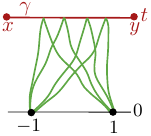}\qquad
    \includegraphics[width=3cm]{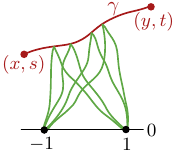}\qquad
    \includegraphics[width=3cm]{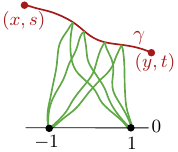}
    \caption{\small An illustration of Open Problem \ref{prob3}. Left: the case of a horizontal line segment. Middle and right: the case of a space-time path.}
    \label{fig:open}
\end{figure}

Theorem 1 in \cite{Bat-Gan-Ham-22} shows that for all rational $t=s>0$, the Hausdorff dimension of the set of points $z\in[x,y]$ as in Problem \ref{prob3} is $1/2$, thereby yielding a positive answer to Problem \ref{prob1} for rational times (see also \eqref{Hausdorff} and Remark \ref{rmk:Haus0}). A corresponding result for arbitrary $t=s>0$ would imply a positive answer to Problem \ref{prob1}; see Remark \ref{rmk:Hausdorff}. A positive answer to Problem \ref{prob3} for arbitrary $t>s>0$ implies a positive answer to Problem \ref{prob2};
see Remark \ref{rk:ancestry}. 

\subsection{Shocks}\label{sec:shocks.intro} Viewing semi-infinite geodesics as characteristic lines leads to the following notion of a shock point for the eternal solution $W^{\tht\sig}$. A point $(x,s)\in\R^2$ is called a $\tht\sigg$ \emph{shock} if the two geodesics $\geo\from{(x,s)}\dir{L}{\tht}{\sig}$ and $\geo\from{(x,s)}\dir{R}{\tht}{\sig}$ immediately separate. More precisely, there exists $t>s$ such that
$\geo\from{(x,s)}\dir{L}{\tht}{\sig}(r)<\geo\from{(x,s)}\dir{R}{\tht}{\sig}(r)$ for all $r\in(s,t)$.
Note that the geodesic $\geo\from{(x,s)}\dir{M}{\tht}{\sig}$ plays no role in this definition. 

As in the classical Burgers' equation, $\tht\sigg$ shocks form continuous interfaces that organize into a tree structure, which is dual to the tree formed by the relative interiors of the semi-infinite $W^{\tht\sig}$-geodesics $\geo\from{(x,s)}\dir{S}{\tht}{\sig}((s,\infty))$, $S\in\{L,M,R\}$. 
The precise definitions and properties of shock points and interfaces needed in this work are summarized in Section \ref{sec:cif+shock}. 

Shock points mark locations where multiple semi-infinite geodesics correspond to the same Busemann function (i.e.\ the same eternal solution). By Theorem 6.1(ii) of \cite{Bus-Sep-Sor-24}, such points exist for all $\tht\in\R$ and $\sigg\in\{-,+\}$ (in fact, our Lemma \ref{lem:shocksdense} below shows that they are dense in $\R^2$). However, since the geodesics $\geo\from{(x,s)}\dir{S}{\tht}{\sig}$, $S\in\{L,R\}$ and $\sigg\in\{-,+\}$, are all $\tht$-directed, there are also points that are not shock points but still emit multiple $\tht$-directed geodesics, and are hence  instability points. This naturally raises the question of how shock points relate to instability points. Understanding this relationship is the main objective of the present paper and is summarized in the next main theorem. 

To state the next theorem, we first introduce the various types of shock points, classified according to the configuration of geodesics emanating from the point.

When $(x,s)\notin\IG\tht$, there is a single type of shock, which we call a \emph{proper double shock}, since the absence of a sign distinction makes the point simultaneously a $\tht-$ and a $\tht+$ shock. These correspond to the configurations in the top row of Figure \ref{fig:geodesics}, excluding the first and fifth, which are not shocks.

When $(x,s)\in\IG\tht$, there are four types of shock configurations (up to $+/-$ symmetry), all depicted in Figure \ref{fig:geodesics}. The first configuration in the second row is not a shock; we refer to such a point as a \emph{proper non-shock instability} point (\emph{pns} point). The second and third configurations in that row are, respectively, $\tht+$ and $\tht-$ \emph{single shocks}, while the fourth configuration is a \emph{snowbird\footnote{The term is a reference to the Salt Lake City ski resort logo: \url{https://www.snowbird.com/}} shock}. 

All configurations in the third and fourth rows correspond to shocks. Those in the third row are called $\tht+$ \emph{hugging shocks}, while those in the fourth row are $\tht-$ \emph{hugging shocks}. We prove, in Lemma \ref{LRisolgeo}\eqref{LRisolgeo.b} below, that no point can be simultaneously a $\tht-$ and a $\tht+$ hugging shock. Among these configurations, the first three in each of the third and fourth rows are single shocks, whereas the last configuration in each row is an \emph{improper double shock}.

Before stating our main theorem, we introduce one final piece of terminology. The duality between shock interfaces and geodesics, together with the ordering \eqref{geoorder}, induces an ordering on shock interfaces: $\tht+$ interfaces eventually lie to the left of $\tht-$ interfaces. Accordingly, if a $\tht-$ interface emanating from $(x,s)\in\R^2$ lies strictly to the left of the corresponding $\tht+$ interface over some interval $(r,s)$, we say that the pair of interfaces is \emph{misordered}.

We are now ready to state our main theorem; see Figure \ref{fig:island-complete} for an illustration. The proof is given in Section \ref{sec:mainproofs}.

\begin{figure}[hpt]
    \includegraphics[width=7cm]{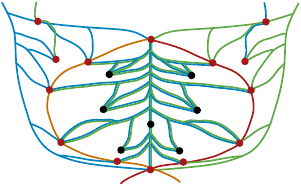}
    \caption{\small All 20 possible configurations of $\tht$-directed geodesics, in relation to the instability graph.  
    Compare with Figure \ref{fig:geodesics}. 
The stability island is the region between the misordered $\tht+$ shock interface (in red) and $\tht-$ shock interface (in orange).
Points in the interior of the island are stable and configurations arising from them correspond to the last four on the top row of Figure \ref{fig:geodesics}. 
Points on the right boundary are left-isolated and configurations from them correspond to those in the third row, while those on the left boundary are right-isolated and their configurations correspond to the last row. 
The top point of the island is a pns point and corresponds to the first configuration on the second row, and the bottom point is a snowbird shock and corresponds to the last configuration in the second row. Both are neither left- nor right-isolated. Dust instability points, not contained in the closure of any island, are also neither left- nor right-isolated and correspond to the first three configurations on the second row.}
    \label{fig:island-complete}
\end{figure}

\begin{thm}\label{main:shocks}
The following hold for all $\w$ in a full $\P$-probability event and all $\tht\in\baddir$. 
\begin{enumerate}[label={\rm(\alph*)}, ref={\rm\alph*}] \itemsep=1pt
\item\label{shocks.a} If $(x,s)\notin\IG\tht$, then it is either not a shock or is a proper double shock. Proper double shocks can only occur at stability points.
\item\label{shocks.b} A point $(x,s)\in\R^2$ is the tip of a stability island if and only if there exists a misordered pair of $\tht-$ and $\tht+$ shock interfaces $\tau^-$ and $\tau^+$ emanating from $(x,s)$ such that, for every $r<s$,
$\tau^-([r,s])\setminus\NU_1^{\tht+}\neq\varnothing$
and $\tau^+([r,s])\setminus\NU_1^{\tht-}\neq\varnothing$.
Moreover, every such tip is a \emph{pns} point. 
\item\label{shocks.c} A point $(x,s)$ is a \emph{snowbird shock} if and only if it is the bottom of a stability island. In that case, the tip of the island is the point at which the geodesics $\geo\from{(x,s)}\dir{R}{\tht}{-}$ and $\geo\from{(x,s)}\dir{L}{\tht}{+}$ first separate. Furthermore, $(x,s)$ is the point where the misordered pair of $\tht\pm$ interfaces from the tip reintersect.
\item\label{shocks.d} A point $(x,s)$ is a $\tht-$ {\rm(}respectively $\tht+${\rm)} hugging shock if and only if it lies on the left {\rm(}respectively right{\rm)} boundary of a stability island and is neither the tip nor the bottom of that island.
\item\label{shocks.e} If $(x,s)$ is a dust point {\rm(}i.e.\ $(x,s)\in\dust${\rm)}, then it is either a \emph{pns} point or a single shock. Conversely, single shocks occur only at dust points. If $(x,s)$ is a pns point, then it is either the tip of a stability island or a dust point. 
\item\label{shocks.f} Fix a stability island. Then each of the first, third, and fourth configurations on the third {\rm(}respectively fourth{\rm)} row in Figure \ref{fig:geodesics} occurs on a set of starting points that is dense in the right {\rm(}respectively left{\rm)} boundary of the island.  The second configuration in that row occurs on a set of isolated points along the same boundary, which accumulates at the tip of the island.
\end{enumerate}
\end{thm}


One of the key messages of the theorem is the following. Shock points correspond to locations where the uniqueness of semi-infinite geodesics fails only locally, whereas instability points mark regions where this non-uniqueness persists globally. It is not {\it a priori} clear that the $\tht$-instability regions can be determined solely from the $\tht\pm$ shock interfaces. The theorem shows that this is indeed the case: the stability islands form a kind of \emph{skeleton} of the instability graph, which can be completely reconstructed from the $\tht\pm$ shock interfaces, and whose closure yields the full instability graph.
The prominence of stability islands, as opposed to dust points, is further underscored by Lemma \ref{shocksintersect}, which shows that $\tht\sigg$ shock interfaces never coalesce at dust points. In PDE terms, coalescence of shock interfaces corresponds to a redistribution of entropy (or energy) dissipation. This suggests that such dissipative interactions are confined to the boundaries of stability islands, highlighting a possibly distinguished structural role.

\section{Properties of Busemann functions, geodesics, and shock interfaces}\label{sec:inputs}

This section introduces the Busemann process, its semi-infinite geodesics, and the associated shock interfaces. We also collect the structural inputs from the literature that we rely on, and conclude the section with a number of new properties whose proofs are deferred to Appendix \ref{sec:auxlems}. All statements in this section hold on a single full-probability event, specified at the end of the section. From that point onward, we fix $\w$ in this event, and the remainder of the paper proceeds deterministically, relying solely on geometric and analytic arguments.

There are three reasons for presenting the inputs in this consolidated form. First, once this section is complete, all subsequent arguments refer only to results stated here, without further citation. Second, we isolate and make explicit the structural properties needed for our analysis---namely continuity, monotonicity, coalescence, and duality properties of the directed landscape, the Busemann process, the geodesics, and the interfaces. Many of these properties are already known in broader contexts, and the remaining ones are expected to hold in related settings, such as one-dimensional stochastic Burgers' equations. Finally, organizing the material in this way allows us to work on a single full-probability event on which all required properties hold, so that the remainder of the paper is entirely deterministic.

It is worth noting that, although the high-level overview in Section \ref{sec:setting} begins with semi-infinite geodesics and defines the Busemann functions from them (see \eqref{Wdef}), the exposition here proceeds in the opposite direction, starting with the Busemann process.

\subsection{The Busemann process}\label{sec:Busproc}
We begin with the directed landscape $\{\mathcal L(x,s;y,t):x,s,y,t\in\R^2,\ t>s\}$. By Lemma 10.3 in \cite{Dau-Ort-Vir-22}, there exists an event $\Omega_1$ such that $\P(\Omega_1)=1$ and for all $\w\in\Omega_1$,
\begin{align}\label{L:cont}
\mathcal L\text{ is continuous on }\{(x,s,y,t)\in\R^4:s<t\}.
\end{align}

The paper \cite{Dau-25} analyzes the possible configurations of point-to-point geodesics in the directed landscape. For our purposes, we require the fact that two specific configurations cannot occur.
More precisely, by Theorems 1.2 (requirement~4) and 1.7 of that paper, there exists an event $\Omega_2 \subset \Omega_1$ with $\P(\Omega_2)=1$ such that the following statements hold for all $\w\in \Omega_2$ and $(x,s) \in \R^2$:
\begin{enumerate}[label={\rm(\alph*)}, ref={\rm\alph*}] \itemsep=1pt
\item For all $y,r,t \in \R$ with $s<r<t$, 
\begin{align}\label{Duncan-deg3}
\begin{split}
&\text{there do not exist three geodesics from $(x,s)$ to $(y,t)$ that are}\\
&\text{disjoint on the time interval $(s,r)$ and coincide on $[r,t]$.}
\end{split}
\end{align}
\item For any $t_i > s$ and $y_i \in \R$, $i \in \{1,2,3,4\}$, and any choice of geodesics $\gamma_i$ from $(x,s)$ to $(y_i,t_i)$, $i \in \{1,2,3,4\}$, 
\begin{align}\label{no4stars}
\text{there exist $t \in \bigl(s,\min(t_1,t_2,t_3,t_4)\bigr)$ and $i\ne j$ in $\{1,2,3,4\}$}:\,\gamma_i|_{[s,t]}=\gamma_j|_{[s,t]}.
\end{align}
In words, there do not exist four geodesics emanating from $(x,s)$ that separate immediately. 
\end{enumerate}

The next object we require is the associated Busemann process 
\begin{align}
\{ W^{\tht\sig}(x,s;y,t) : \tht,x,s,y,t \in \R,\ \sigg \in \{-,+\} \}
\end{align}
of eternal solutions. This process was constructed in \cite[Theorem 5.1]{Bus-Sep-Sor-24}. The exceptional directions of instability are those for which there exist two distinct eternal solutions.

\begin{defn}\label{def:baddir}
    $\tht\in\R$ is called an \emph{exceptional} direction if $W^{\tht-}\neq W^{\tht+}$. The random set of exceptional directions is denoted by $\baddir$. When $\tht\notin\baddir$, we may drop the sign distinction and write $W^\tht$ for both $W^{\tht-}=W^{\tht+}$.
\end{defn}

For $s, \theta \in \R$, let
\begin{align}\label{def:kSs}
\IG\theta_s=\bigl\{\text{points of increase of the function }y \mapsto W^{\theta-}(0,s; y,s) - W^{\theta+}(0,s; y,s)\bigr\}.
\end{align}
We refer to Definition \ref{inc-pts} below for the precise meaning of a point of increase. By Lemma \ref{ptinc}, the set $\IG\tht_s$ is closed.

We now summarize the properties of the Busemann process that are used in this work. There exists an event $\Omega_3\subset\Omega_2$ with $\P(\Omega_3)=1$ and such that the following statements hold for all $\w\in\Omega_3$. 

\begin{enumerate} [label={\rm(\alph*)}, ref={\rm\alph*}] \itemsep=1pt 
\item{\rm(Update)}\cite[Theorem 5.1(iv)]{Bus-Sep-Sor-24} For any $x,y,\tht$ and $t>s$ in $\R$ and any $\sigg\in\{-,+\}$,
    \begin{align}\label{update}
        \sup_{z\in\R} \{\mathcal{L}(x,s;z,t)+W^{\tht\sig}(z,t;y,t)\} = W^{\tht\sig}(x,s;y,t).
    \end{align}
\item{\rm(Continuity)}\cite[Theorem 5.1(i)]{Bus-Sep-Sor-24} For all $\tht\in\R$ and $\sigg\in\{-,+\}$, 
\begin{align}\label{W:cont}
\text{$W^{\tht\sig}(x,s;y;t)$ is continuous in $x$, $s$, $y$, and $t$.}
\end{align}
\item{\rm(Cocycle)}\cite[Theorem 5.1(ii)]{Bus-Sep-Sor-24} For all $x$, $s$, $y$, $t$, $z$, $r$, and $\tht$ in $\R$, and all $\sigg\in\{-,+\}$,
\begin{align}\label{cocycle}
W^{\tht\sig}(x,s;y,t)+W^{\tht\sig}(y,t;z,r)=W^{\tht\sig}(x,s;z,r).
\end{align}
\item{\rm(Montonicity)}\cite[Theorem 5.1(iii)]{Bus-Sep-Sor-24} For all $s$, $x<y$, and $\tht$ in $\R$,
\begin{align}\label{W:mono}
W^{\tht-}(x,s;y,s)\ge W^{\tht+}(x,s;y,s).
\end{align}
\item{\rm(Growth)}\cite[Lemma 5.12(iv)]{Bus-Sep-Sor-24} For all $x,s,t,\tht\in\R$ and $\sigg\in\{-,+\}$,
\begin{align}\label{Wgrowth}
\lim_{|y|\to\infty} y^{-1}W^{\tht\sig}(y,t;x,s)=2\tht.
\end{align}
\item{\rm(Instability)}\cite[(5.7) and Theorem 5.5(iii)]{Bus-Sep-Sor-24}  For any $x,s\in\R$, we have the equivalence
\begin{align}\label{exceptional}
\tht\in\baddir\ \Longleftrightarrow\ \exists y\in\R:W^{\tht-}(x,s;y,s)\ne W^{\tht+}(x,s;y,s).
\end{align}
Furthermore,
\begin{align}\label{countable}
\baddir\text{ is countable and dense}.
\end{align}
\item{\rm(Unboundedness)}\cite[(5.6)]{Bus-Sep-Sor-24} For all $\tht\in\baddir$, $s\in\R$, and $\sigg\in\{-,+\}$,
\begin{align}\label{Wunbounded}
\lim_{x\to\ssig\infty}\bigl(W^{\tht-}(0,s;x,s)-W^{\tht+}(0,s;x,s)\bigr)=\sigg\infty.
\end{align}
\item{\rm(Dimension)}\cite{Bus-Sep-Sor-24} For all $\tht\in\baddir$ and any $a<b$ and $s$ in $\Q$,
\begin{align}\label{Hausdorff}
\text{$\IG\tht_s\cap(a,b)$ is either empty or has Hausdorff dimension $1/2$.}
\end{align}
\end{enumerate}

\begin{rmk}\label{rmk:Haus0}
Theorem 2.10(iii) in \cite{Bus-Sep-Sor-24} implies that, $\P$-almost surely, for every $\tht \in \baddir$ and every rational $s \in \Q$, the set $\IG\tht_s$ has Hausdorff dimension $1/2$. In fact, the same argument yields the stronger statement \eqref{Hausdorff}. More precisely, combining Theorems 5.3(iii), 5.5(ii), and 8.9 of that paper reduces the claim to the corresponding classical fact for the zero set of Brownian motion.
\end{rmk}

\begin{rmk}\label{rmk:Hausdorff}
It is believed that, on an event of full probability, \eqref{Hausdorff} holds for all times $s \in \R$. In any case, we invoke \eqref{Hausdorff} (with $s \in \Q$) only to show that the set of points in $\IG\tht_s$ that do not lie on the boundary of any connected component of the open set $\R \setminus \IG\tht_s$ is nonempty and has no isolated points. See Section \ref{sec:sketch} for more on this.
\end{rmk}

The update equation \eqref{update} and the cocycle property \eqref{cocycle} give that for any  $\w\in\Omega_3$ and any $x_0,s_0,y,\tht\in\R$, $s<t$ in $\R$, and $\sigg\in\{-,+\}$, $(y,t)\mapsto W^{\tht\sig}(y,t;x_0,s_0)$ is an eternal solution for the KPZ fixed point:
\begin{align}\label{eternal}
\begin{split}
\sup_{z\in\R}\{\mathcal{L}(x,s;z,t)+W^{\tht\sig}(z,t;x_0,s_0)\}
=W^{\tht\sig}(x,s;x_0,s_0).
\end{split}
\end{align}

\subsection{Busemann semi-infinite geodesics}\label{sec:Busgeo}
Next, we define the semi-infinite geodesics corresponding to the Busemann eternal solutions. From \eqref{eternal} and the cocycle property \eqref{cocycle}, we get that for $\w\in\Omega_3$, $\sigg\in\{-,+\}$, and any $\tht,x,s,z,t\in\R$ with $s<t$,
\begin{align}\label{L<W}
\mathcal L(x,s;z,t)\le W^{\tht\sig}(x,s;z,t).
\end{align}
A $W^{\tht\sig}$-geodesic emanating from $(u,r)\in\R^2$ is a continuous space-time path $\gamma:[r,\infty)\to\R^2$ such that $\gamma(r)=(u,r)$ and, for all $t>s\ge r$,
\begin{align}\label{rec}
\mathcal L(\gamma(s);\gamma(t)) = W^{\tht\sig}(\gamma(s);\gamma(t)).
\end{align}

By \cite[Theorem 5.9]{Bus-Sep-Sor-24}, there exists an event $\Omega_4\subset\Omega_3$ with $\P(\Omega_4)=1$ such that, for every $\w\in\Omega_4$, every $u,r,\tht\in\R$, and every $\sigg\in\{-,+\}$, there exists at least one $W^{\tht\sig}$-geodesic emanating from $(u,r)$, and any semi-infinite geodesic from $(u,r)$ that coalesces with it is itself a $W^{\tht\sig}$-geodesic. 
\eqref{L<W} and \eqref{rec} imply that  
\begin{align}\label{geomaximizes}
\text{$(z,t)=\gamma(t)$ maximizes the supremum in the update rule \eqref{eternal} with $(x,s)=\gamma(s)$.} 
\end{align}
(Note that $(x_0,s_0)$ is an arbitrary base point that can be changed to any other point in $\R^2$ using the cocycle property \eqref{cocycle}.)
Consequently, for every $t>s\ge r$, the segment $\gamma|_{[s,t]}$ is a characteristic line for the KPZ fixed point, transporting the initial data $W^{\tht\sig}(\,\abullet\,,t;x_0,s_0)$ at time $t$ to the point $(x,s)$, and $\gamma$ may be interpreted as a characteristic line of the eternal solution $W^{\tht\sig}$, propagating information from the remote past $t\to\infty$ (recall that time runs backward for the PDE analogy) to the point $(x,s)$. In percolation terminology, $\gamma$ is a geodesic between any two of its points and is hence a semi-infinite geodesic.

Both Theorem 1 in \cite{Bha-24} and Lemma 3.3 in \cite{Dau-25} independently imply that there exists an event $\Omega_5 \subset \Omega_4$ with $\P(\Omega_5)=1$ such that for all $\w\in\Omega_5$, all $x,s,y_1,y_2,t\in\R$ with $t>s$, and any geodesics $\gamma_i$ from $(x,s)$ to $(y_i,t)$, $i\in\{1,2\}$,
\begin{align}\label{no bubble}
\forall r'<r''\text{ in }(s,t):\bigl[\gamma_1(r')=\gamma_2(r')\text{ and }
\gamma_1(r'')=\gamma_2(r'')\bigr]\Longrightarrow \gamma_1(r)=\gamma_2(r)\ \ \forall r\in[r',r''].
\end{align}
This and Theorem 7.1(ii) in \cite{Bus-Sep-Sor-24} (see their Remark 7.2) imply the existence of an event $\Omega_6 \subset \Omega_5$ with $\P(\Omega_6)=1$ such that for all $\w\in\Omega_6$, $x,s,\tht\in\R$, $\sigg\in\{-,+\}$, any two distinct $W^{\tht\sig}$-geodesics $\gamma_1,\gamma_2:[s,\infty)\to\R^2$ out of $(x,s)$ must immediately split and then later reunite and coalesce. Precisely, there exists an $r>s$ such that $\gamma_1(t)=\gamma_1(t)$ for all $t\ge r$ and $\gamma_1(t)\ne\gamma_2(t)$ for all $t\in(s,r)$.
With \eqref{no4stars}, we get that for any $\w\in\Omega_6$, $x,s,\tht\in\R$, and $\sigg\in\{-,+\}$, there are at most three distinct $W^{\tht\sig}$-geodesics out of $(x,s)$. These geodesics are called the \emph{Busemann geodesics} and we denote them by 
\begin{align}\label{geodesics}
\bigl\{\geo\from{(x,s)}\dir{S}{\tht}{\sig}:(x,s)\in\R^2,\,S\in\{L,M,R\},\,\tht\in\R,\,\sigg\in\{-,+\}\bigr\}.
\end{align}

\begin{rmk}\label{rk:Mgeo}
When there is a unique $W^{\tht\sig}$-geodesic out of $(x,s)$, we have $\geo\from{(x,s)}\dir{L}{\tht}{\sig}=\geo\from{(x,s)}\dir{M}{\tht}{\sig}=\geo\from{(x,s)}\dir{R}{\tht}{\sig}$. When there are exactly two distinct $W^{\tht+}$-geodesics out of $(x,s)$, we use the convention that $\geo\from{(x,s)}\dir{M}{\tht}{+}=\geo\from{(x,s)}\dir{R}{\tht}{+}$. When there are exactly two distinct $W^{\tht-}$-geodesics out of $(x,s)$, we use the convention that $\geo\from{(x,s)}\dir{M}{\tht}{-}=\geo\from{(x,s)}\dir{L}{\tht}{-}$. 
\end{rmk}

\begin{rmk}
In what follows, we adopt the convention of using the same notation for a geodesic and its image: for a space-time path $\gamma:I\to\R^2$, we may write $\gamma$ to denote both the map and its range $\gamma(I)$.
\end{rmk}

Now, we summarize the properties of the Busemann geodesics \eqref{geodesics}.
There exists an event $\Omega_7\subset\Omega_6$ such that $\P(\Omega_7)=1$ and the following hold for all $\w\in\Omega_7$.
\begin{enumerate} [label={\rm(\alph*)}, ref={\rm\alph*}] \itemsep=1pt 
\item{\rm(Continuity)} 
For all $x$, $s$, and $\tht$ in $\R$, and for all $\sigg\in\{-,+\}$ and $S\in\{L,M,R\}$,
\begin{align}\label{geo:cont}
\text{$\geo\from{(x,s)}\dir{S}{\tht}{\sig}:[s,\infty)\longrightarrow\R^2$ is  a continuous space-time path.}
\end{align}
\item{\rm(Directedness)}\cite[Theorem 5.9(ii)(d)]{Bus-Sep-Sor-24} For all $x$, $s$, and $\tht$ in $\R$, and for all $\sigg\in\{-,+\}$ and $S\in\{L,M,R\}$,
\begin{align}\label{geo:directed}
\lim_{t\to\infty}t^{-1}\geo\from{(x,s)}\dir{S}{\tht}{\sig}(t)=(\tht,1).
\end{align}
\item{\rm(Point-to-point extremality)}\cite[Theorem 5.9(iv)]{Bus-Sep-Sor-24} For all $x$, $t>r\ge s$,  $\tht$ in $\R$, and all $\sigg\in\{-,+\}$,
\begin{align}\label{p2pleftmost}
\text{$\geo\from{(x,s)}\dir{L}{\tht}{\sig}([r,t])$ is the leftmost geodesic between $\geo\from{(x,s)}\dir{L}{\tht}{\sig}(r)$ and $\geo\from{(x,s)}\dir{L}{\tht}{\sig}(t)$.}
\end{align}
Similarly, 
\begin{align}\label{p2prightmost}
\text{$\geo\from{(x,s)}\dir{R}{\tht}{\sig}([r,t])$ is the rightmost geodesic between $\geo\from{(x,s)}\dir{R}{\tht}{\sig}(r)$ and $\geo\from{(x,s)}\dir{R}{\tht}{\sig}(t)$.}
\end{align}
\item{\rm(Point-to-line extremality)}\cite[Theorem 5.9(iii)]{Bus-Sep-Sor-24} For all $x$, $t>r\ge s$,  $\tht$ in $\R$, and all $\sigg\in\{-,+\}$,
\begin{align}\label{p2lleftmost}
\text{$y=\geo\from{(x,s)}\dir{L}{\tht}{\sig}(t)$ is the leftmost maximizer of $\mathcal L\bigl(\geo\from{(x,s)}\dir{L}{\tht}{\sig}(r);y,t\bigr)+W^{\tht\sig}(y,t;x_0,s_0)$ over $y\in\R$.}
\end{align}
Similarly, 
\begin{align}\label{p2lrightmost}
\text{$y=\geo\from{(x,s)}\dir{R}{\tht}{\sig}(t)$ is the rightmost maximizer of $\mathcal L\bigl(\geo\from{(x,s)}\dir{R}{\tht}{\sig}(r);y,t\bigr)+W^{\tht\sig}(y,t;x_0,s_0)$ over $y\in\R$.}
\end{align}
(Here, $(x_0,s_0)$ is an arbitrary point that can be changed to any other point in $\Z^2$ by the cocycle property \eqref{cocycle}.)
\item{\rm(Global extremality)}\cite[Theorem 6.5(i)]{Bus-Sep-Sor-24} For all $x,s,\tht\in\R$ and any semi-infinite geodesic $\gamma:[s,\infty]\to\R^2$ such that $\gamma(s)=(x,s)$ and $t^{-1}\gamma(t)\to(\tht,1)$ as $t\to\infty$,
\begin{align}\label{extreme}
\geo\from{(x,s)}\dir{L}{\tht}{-}\preceq\gamma\preceq\geo\from{(x,s)}\dir{R}{\tht}{+}.
\end{align}
\item{\rm(Ordering)}[Lemma \ref{lm:geo-ordering}] For all $s$, $x<y<z$, and $\tht$ in $\R$, 
and all $\sigg\in\{-,+\}$ and $S\in\{L,M,R\}$,
\begin{align}
\geo\from{(x,s)}\dir{R}{\tht}{\sig}\preceq\geo\from{(y,s)}\dir{L}{\tht}{\sig}
\preceq\geo\from{(y,s)}\dir{M}{\tht}{\sig}
\preceq\geo\from{(y,s)}\dir{R}{\tht}{\sig}
\preceq\geo\from{(z,s)}\dir{L}{\tht}{\sig}
\quad\text{and}\quad
\geo\from{(x,s)}\dir{S}{\tht}{-}\preceq\geo\from{(x,s)}\dir{S}{\tht}{+}.\label{geo:mono}
\end{align}
\item{\rm(Sign distinction)}[Remark \ref{rk:sign}] For all $x,s\in\R$ and $S\in\{L,M,R\}$,
\begin{align}\label{sign}
\begin{split}
&\tht\notin\baddir\Longrightarrow\geo\from{(x,s)}\dir{S}{\tht}{-}=\geo\from{(x,s)}\dir{S}{\tht}{+}\text{ (in which case we may drop the sign and write $\geo\from{(x,s)}\dir{S}{\tht}{}$)},\\
&\tht\in\baddir\Longrightarrow\geo\from{(x,s)}\dir{S}{\tht}{-}(t)<\geo\from{(x,s)}\dir{S}{\tht}{+}(t)\quad\text{for all large enough }t.
\end{split}
\end{align}
\item{\rm(Limits)}\cite[Theorem 6.3(v)]{Bus-Sep-Sor-24} For all $x$, $s$, and $\tht$ in $\R$, and for all $\sigg\in\{-,+\}$ and $S\in\{L,M,R\}$,
\begin{align}\label{geo:lim}
  \forall t\ge s,\ \lim_{z\nearrow x}\geo\from{(z,s)}\dir{S}{\tht}{\sig}(t)=\geo\from{(x,s)}\dir{L}{\tht}{\sig}(t)\quad\text{and}\quad
   \lim_{z\searrow x}\geo\from{(z,s)}\dir{S}{\tht}{\sig}(t)=\geo\from{(x,s)}\dir{R}{\tht}{\sig}.
\end{align}
\item{\rm(Jumps)}\cite[Theorem 7.9]{Bus-Sep-Sor-24} For all $\tht\in\baddir$ and $s$, and $x<y$ in $\R$,
\begin{align}\label{Wjump}
W^{\tht-}(x,s;y,s)>W^{\tht+}(x,s;y,s)\ \Longleftrightarrow\ 
\geo\from{(x,s)}\dir{R}{\tht}{-}\cap\geo\from{(y,s)}\dir{L}{\tht}{+}=\varnothing.
\end{align}
\item{\rm(Disjointness)}\cite[Theorem 2.10(ii)]{Bus-Sep-Sor-24} For any $(x,s)\in\Q^2$ and $\tht\in\R$, 
\begin{align}\label{Q notin IG}
\geo\from{(x,s)}\dir{L}{\tht}{-}\cap\geo\from{(x,s)}\dir{R}{\tht}{+}\ne\{(x,s)\}.
\end{align}
\item{\rm(Coalescence)}\cite[Theorem 7.1 and Remark 7.2]{Bus-Sep-Sor-24}  For all $x, s, y, t, \tht \in \R$,  $\sigg \in \{-,+\}$, and $S, S' \in \{L, M, R\}$, we have that
\begin{align}
\begin{split}
&\text{if }\geo\from{(x,s)}\dir{S}{\tht}{\sig}(r)
=
\geo\from{(y,t)}\dir{S'}{\tht}{\sig}(r)
\text{ for }r \ge s \vee t \text{ with }r> s\wedge t\text{, then }
\geo\from{(x,s)}\dir{S}{\tht}{\sig}\big|_{[r,\infty)}
=\geo\from{(y,t)}\dir{S'}{\tht}{\sig}\big|_{[r,\infty)}.
\end{split}\label{geo:coal1}\\[4pt]
&\text{Furthermore, there exists an $r$ as in \eqref{geo:coal1}.}
\label{geo:coal2}
\end{align}
When $r$ is the minimal time for which \eqref{geo:coal1} holds, we say the geodesics \emph{coalesce at time $r$}.\end{enumerate}

\begin{rmk}\label{rk:sign}
The first line in \eqref{sign} follows from the fact that $\tht\notin\baddir$ implies $W^{\tht-}=W^{\tht+}$, by Definition \ref{def:baddir}. When $S\in\{L,R\}$, the second line in \eqref{sign} follows from \cite[Theorem 7.3]{Bus-Sep-Sor-24}. 
This and the coalescence \eqref{geo:coal1}-\eqref{geo:coal2} prove the second implication in \eqref{sign} also for $S=M$. 
\end{rmk}

\begin{rmk}\label{rk:Mgeostable}
When $\tht \notin \baddir$, there is no sign distinction in the Busemann geodesics.  However, our convention from Remark \ref{rk:Mgeo} for defining $\geo\from{(x,s)}\dir{M}{\tht}{}$ becomes ambiguous when exactly two $W^{\tht}$-geodesics emanate from $(x,s)$. In this case, we adopt the convention $\geo\from{(x,s)}\dir{M}{\tht}{} = \geo\from{(x,s)}\dir{L}{\tht}{}$.
This choice is purely notational and has no effect on the arguments or results.
\end{rmk}

The coalescence property \eqref{geo:coal1} has the following consequence: for any $\w\in\Omega_7$, any $x$, $t>s$, and $\tht$ in $\R$, and any $\sigg\in\{-,+\}$ and $S\in\{L,M,R\}$,  
\begin{align}
    \geo\from{\geo\from{(x,s)}\dir{S}{\tht}{\sig}(t)}\dir{L}{\tht}{\sig}(r)=\geo\from{\geo\from{(x,s)}\dir{S}{\tht}{\sig}(t)}\dir{M}{\tht}{\sig}(r)=\geo\from{\geo\from{(x,s)}\dir{S}{\tht}{\sig}(t)}\dir{R}{\tht}{\sig}(r)=\geo\from{(x,s)}\dir{S}{\tht}{\sig}(r),\quad\text{for all }r\ge t.\label{geo:restart}
\end{align}

Theorem 2.5(i) in \cite{Bus-Sep-Sor-24} says that there exists an event $\Omega_8\subset\Omega_7$ with $\P(\Omega_8)=1$ and such that for all $\w\in\Omega_8$ and $x,s\in\R$, every semi-infinite geodesic $\gamma$ out of $(x,s)$ is directed: there exists a $\tht\in\R$ such that $t^{-1}\gamma(t)\to(\tht,1)$ as $t\to\infty$.
Next, Theorem 1.5 in \cite{Bus-25-} states that there exists an event $\Omega_9\subset\Omega_8$ with $\P(\Omega_9)=1$ and such that for all $\w\in\Omega_9$, all $\tht\in\R$, any $x_i,s_i\in\R$ and geodesics $\gamma_i:[s_i,\infty)\to\R^2$, $i\in\{1,2,3\}$, that are $\tht$-directed,
\begin{align}\label{no3geo}
\exists i\ne j\text{ in }\{1,2,3\},\ \exists r>\max\{s_1,s_2,s_3\}\,:\,\gamma_i(t)=\gamma_j(t)\ \ \forall t\ge r.
\end{align}
Together with \eqref{extreme}, this implies that any $\tht$-directed geodesic $\gamma$ emanating from $(x,s)$ must coalesce with either $\geo\from{(x,s)}\dir{L}{\tht}{-}$ or $\geo\from{(x,s)}\dir{R}{\tht}{+}$. Consequently, for some $\sigg\in\{-,+\}$, $\gamma$ is a $W^{\tht\sig}$-geodesic and therefore coincides with one of the three $W^{\tht\sig}$-geodesics.
In short, for every $\w \in \Omega_9$, 
\begin{align}\label{allgeo}
\text{the process in \eqref{geodesics} exhausts all semi-infinite geodesics.}
\end{align}

By Proposition 34 in \cite{Bha-24}, there exists an event $\Omega_{10} \subset \Omega_9$ with $\P(\Omega_{10})=1$ such that, for all $\w \in \Omega_{10}$,
\begin{align}\label{nobiinfinite}
\text{there exist no bi-infinite geodesics.}
\end{align}
Consequently, on $\Omega_{10}$, the process \eqref{geodesics} contains all infinite geodesics of the model.

\subsection{Competition interfaces and the shocks tree}\label{sec:cif+shock}
Finally, we turn to the set of shocks.

\begin{defn}\label{def:shock}
For $\tht\in\R$ and $\sigg\in\{-,+\}$, a point $(x,s)$ such that $\geo\from{(x,s)}\dir{L}{\tht}{\sig}$ and $\geo\from{(x,s)}\dir{R}{\tht}{\sig}$ immediately separate is called a $\tht\sigg$ \emph{shock}. Following (6.2) in \cite{Bus-Sep-Sor-24}, we denote the set of $\tht\sigg$ shocks by $\NU_1^{\tht\sig}$. When $\tht\notin\baddir$, we drop the sign distinction and write $\NU_1^\tht$.
\end{defn}

\begin{figure}[hpt]
    \includegraphics[width=3.5cm]{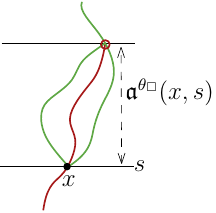}
    \caption{\small A shock point $(x,s)$.}
    \label{fig:shock}
\end{figure}

By \eqref{geo:coal2}, we have that for any $\w\in\Omega_{10}$, $x,s,\tht\in\R$, and any $\sigg\in\{-,+\}$, 
\begin{align}\label{age}
\begin{split}
   (x,s)\in\NU_1^{\tht\sig}\ \Longrightarrow\ 
       &\exists \age^{\tht\sig}(x,s)>0:\geo\from{(x,s)}\dir{L}{\tht}{\sig}(r)<\geo\from{(x,s)}\dir{R}{\tht}{\sig}(r)\ \forall r\in(s,s+\age^{\tht\sig}(x,s))\\
       &\qquad\qquad\qquad\qquad\text{ and }
    \geo\from{(x,s)}\dir{L}{\tht}{\sig}(r)=\geo\from{(x,s)}\dir{R}{\tht}{\sig}(r)\ \forall r\ge s+\age^{\tht\sig}(x,s).
\end{split}
\end{align}
Then we say that the shock at $(x,s)$ \emph{resolved} at time $s+\age^{\tht\sig}(x,s)$ and $\age^{\tht\sig}(x,s)$ is called the \emph{age} of the shock. See Figure \ref{fig:shock}. See also Remark \ref{rk:age} for a justification of the terminology. 

The coalescence \eqref{geo:coal1} implies that for any $\w\in\Omega_{10}$, $x,s,\tht\in\R$, and any $\sigg\in\{-,+\}$,
\begin{align}
    \geo\from{(x,s)}\dir{L}{\tht}{\sig}\ne\geo\from{(x,s)}\dir{R}{\tht}{\sig}\Longleftrightarrow (x,s)\in\NU_1^{\tht\sig}.\label{geo:no-lollipop}
\end{align}

By Theorem 6.1(ii) in \cite{Bus-Sep-Sor-24}, there exists an event $\Omega_{11}\subset\Omega_{10}$ with $\P(\Omega_{11})=1$ and such that for all $\w\in\Omega_{11}$, $s,\tht\in\R$, and $\sigg\in\{-,+\}$, 
\begin{align}\label{NUcountable}
  \text{$\NU_1^{\tht\sig}\cap\bigl(\R\times\{s\}\bigr)$ is countably infinite.}  
\end{align}
Our Lemma \ref{lem:shocksdense} below says that the above set is also dense in $\R\times\{s\}$. 

By Theorem 6.1(i) in \cite{Bus-Sep-Sor-24}, for any $(x,s)\in\R^2$, 
\[\P\bigl\{\exists\tht\in\R,\sigg\in\{-,+\}:(x,s)\in\NU_1^{\tht\sig}\bigr\}=0.\]
Therefore, there exists an event $\Omega_{12}\subset\Omega_{11}$ with $\P(\Omega_{12})=1$ such that for each $\w\in\Omega_{12}$, 
\begin{align}\label{ratnoshock}
\bigcup_{\substack{\tht\in\R\\ \sig\in\{-,+\}}}\NU_1^{\tht\sig}\subset\R^2\setminus\Q^2.
\end{align}

We need the following result, eliminating two specific configurations of semi-infinite geodesics. 

\begin{lem}\label{lm:Duncan}{\rm[\citenum{Dau-Ort-Vir-22}, Lemma 10.2(1-2) and \citenum{Dau-Pan-26-}, Lemma 8.1]}
There exists an event $\Omega_{13}\subset\Omega_{12}$ with $\P(\Omega_{13})=1$ and such that for any $\w\in\Omega_{13}$ and any $x,\tht\in\R$ and $s\in\Q$, we have the following.
\begin{enumerate} [label={\rm(\alph*)}, ref={\rm\alph*}] \itemsep=1pt 
\item\label{Duncan.a} If $\gamma^L$ and $\gamma^R$ are two geodesics out of $(x,s)$ that are $\tht$-directed {\rm(}i.e.\ satisfy \eqref{geo:directed}{\rm)} and the two geodesics separate immediately at $(x,s)$, then reunite at $(y,t)$, for the first time after time $s$, and then separate again at some later time $t'>t$, then there exists a $t''\in(t,t')$ such that $\gamma^L(r)=\gamma^R(r)$ for $r\in[t,t'']$. See Figure \ref{fig:tail}.
\item\label{Duncan.b} 
For any $\tht$-directed geodesic $\gamma$ out of $(x,s)$, there exists an $r>s$  such that either $\gamma|_{[s,r]}=\geo\from{(x,s)}\dir{L}{\tht}{-}\big|_{[s,r]}$ or $\gamma|_{[s,r]}=\geo\from{(x,s)}\dir{R}{\tht}{+}\big|_{[s,r]}$. In particular, 
there are no distinct middle $W^{\tht-}$ or $W^{\tht+}$ geodesics from $(x,s)$. 
In the notation of Remarks \ref{rk:Mgeo} and \ref{rk:Mgeostable}, if $\tht \notin \baddir$, then $\geo\from{(x,s)}\dir{M}{\tht}{} = \geo\from{(x,s)}\dir{L}{\tht}{}$, while if $\tht \in \baddir$, then $\geo\from{(x,s)}\dir{M}{\tht}{-} = \geo\from{(x,s)}\dir{L}{\tht}{-}$ and $\geo\from{(x,s)}\dir{M}{\tht}{+} = \geo\from{(x,s)}\dir{R}{\tht}{+}$.
\end{enumerate}
\end{lem}



\begin{figure}[hpt]
    \includegraphics[width=2.5cm]{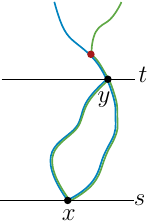}
    \qquad\qquad
    \includegraphics[width=2.5cm]{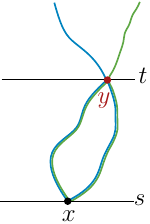}
    \caption{\small Illustration of Lemma \ref{lm:Duncan}\eqref{Duncan.a}: the configuration on the right never happens for geodesics directed in the same direction and starting at a rational (or deterministic) time.}
    \label{fig:tail}
\end{figure}

\begin{rmk}
An immediate consequence of our Theorem \ref{main:geodesics} is that part \eqref{Duncan.a} of Lemma \ref{lm:Duncan} holds almost surely for every time $s$. In contrast, the same theorem shows that, unlike the statement of Lemma \ref{lm:Duncan}\eqref{Duncan.b} for deterministic times, there almost surely exists a dense set of exceptional times at which distinct middle geodesics do occur.
\end{rmk}

Similarly to one-dimensional Burgers' equations, the shocks $\NU_1^{\tht\sig}$ form a tree. To see this, we first recall the notion of competition interfaces, first introduced in \cite{Rah-Vir-25} for the directed landscape. 

\begin{defn}\label{def:init}
A \emph{boundary condition} is an upper semi-continuous function $f:\R\to\R$ that is linearly bounded above:
\[\exists C:\forall x\in\R\ \ f(x)\le C(1+|x|).\]
\end{defn}

Given a boundary condition $f$, a reference space point $a$, and a time level $t$, define the \emph{competition function}
\begin{align}\label{ddef}
d_{(a,t)}(x,r)=\sup_{y\ge a}\{\mathcal L(x,r;y,t)+f(y)\}-\sup_{y\le a}\{\mathcal L(x,r;y,t)+f(y)\},\ r<t,\ x\in\R.
\end{align}
Corollary 10.7 in \cite{Dau-Ort-Vir-22} says that there exists an event $\Omega_{14}\subset\Omega_{13}$ with $\P(\Omega_{14})=1$ and such that for all $\w\in\Omega_{14}$ and all $\e>0$, there exist positive finite constants $C_\e(\w)$ and $C'_\e(\w)$ such that for all $r\le t-\e$ and $x,y$ in $\R$, 
\[\mathcal L(x,r;y,t)\le C_\e(\w)(|x|+|r|+|y|+|t|)-C'_\e(\w)(y-x)^2.\]
This and the continuity of $\mathcal L$ imply that $d_{(a,t)}$ is continuous on $\R\times(-\infty,t)$.
A now-standard paths-crossing argument shows that  $x\mapsto d_{(a,t)}(x,r)$ is nondecreasing; see  \cite[Proposition 4.1]{Rah-Vir-25}. 

Next, define the left and right \emph{competition interfaces}: for $r<t$,
\begin{align}\label{CIs}
I^-_{(a,t)}(r)=\inf\{x\in\R:d_{(a,t)}(x,r)\ge0\}\quad\text{and}\quad
I^+_{(a,t)}(r)=\sup\{x\in\R:d_{(a,t)}(x,r)\le0\}.
\end{align}
Set $I^-_{(a,t)}(t)=I^+_{(a,t)}(t)=a$.

By the continuity and monotonicity of the competition function, we get that 
\begin{align}\label{Iorder}
-\infty\le I^-_{(a,t)}(r)\le I^+_{(a,t)}(r)\le\infty\quad\forall r<t.
\end{align}
The same holds for $r=t$ by the definition of interfaces at time $t$.

There exists an event $\Omega_{15}\subset\Omega_{14}$\label{Om15} with $\P(\Omega_{15})=1$ and such that for any $\w\in\Omega_{15}$, any boundary condition $f$ (as in Definition \ref{def:init}), and any reference point $a$ and time level $t$, the interfaces $I^{-}_{(a,t)}$ and $I^{+}_{(a,t)}$ have the following properties. 

\begin{enumerate} [label={\rm(\alph*)}, ref={\rm\alph*}] \itemsep=1pt 
\item{\rm(Continuity)}\cite[Propositions 4.4, 4.5, and 4.6]{Rah-Vir-25} For each $\sigg\in\{-,+\}$, 
\begin{align}\label{I-cont}
\text{the interface $I^{\sig}_{(a,t)}$ is continuous on $(-\infty,t]$.} 
\end{align}
In particular, $I^{\sig}_{(a,t)}(r)\to a$ as $r\nearrow t$.
\item{\rm(Limits)}\cite[Lemma 5.2]{Rah-Vir-25} There exists a sequence $a_n$ that strictly increases to $a$ and another sequence $a'_n$ that strictly decreases to $a$ such that
\begin{align}\label{Icv}
\forall r\le t\,:\,I^-_{(a,t)}(r)=\lim_{n\to\infty}I^+_{(a_n,t)}(r) \quad\text{and}\quad I^+_{(a,t)}(r)=\lim_{n\to\infty}I^-_{(a'_n,t)}(r).
\end{align}
\item{\rm(Directedness)}\cite[Corollary 4.14]{Rah-Vir-25} If for some $\alpha\in\R$, $x^{-1} f(x)\to\alpha$ as $|x|\to\infty$, then for both $\sigg\in\{-,+\}$,
\begin{align}\label{Idir}
    \lim_{t\to-\infty} t^{-1} I^\sig_{(a,t)}=\tfrac{\alpha}2.
\end{align}
\end{enumerate}

\begin{rmk}
In \cite{Rah-Vir-25}, competition interfaces were constructed in the forward (positive) time direction, that is, their interfaces $I^{\pm}_{(a,t)}(r)$ were defined for $r>t$. The interfaces considered here are obtained by reflecting both space and time. This transformation is justified by the invariance in distribution of the directed landscape under such reflections; see \cite[Lemma 10.2(3)]{Dau-Ort-Vir-22}. More importantly, the construction in \cite{Rah-Vir-25} was done for fixed given boundary conditions $f$, reference points $a$, and time level $t=0$. By the temporal shift-invariance of the directed landscape \cite[Lemma 10.2(1-2)]{Dau-Ort-Vir-22}, the same holds for any fixed time level $t$. What is more is that the construction and results in \cite{Rah-Vir-25} are in fact deterministic, once one restricts to a full measure event for which certain almost sure properties of the directed landscape are satisfied. Thus, the results hold in fact simultaneously for all boundary conditions $f$, reference points $a$, and time levels $t$, on a single full probability event.
\end{rmk}

We now describe the construction of the shock interfaces and the associated tree structure developed in \cite{Bha-24}. 

\begin{rmk}
Below we invoke results from \cite{Bha-24} that were originally stated under the assumption $\tht\notin\baddir$. However, their proofs extend verbatim to the case $\tht\in\baddir$, once $W^\tht$ and its geodesics are replaced by $W^{\tht\sig}$ and its geodesics. The arguments in \cite{Bha-24} rely only on the  properties of the Busemann process, its semi-infinite geodesics, and the competition interfaces from \cite{Rah-Vir-25}, all of which hold for $\w\in\Omega_{15}$ as described above. Accordingly, throughout what follows we apply the results of \cite{Bha-24} for all $\tht\in\R$, provided $\w\in\Omega_{15}$.
\end{rmk}

\begin{defn}\label{def:shockint}
A $\tht\sigg$ \emph{shock interface} is a continuous semi-infinite space-time path $\tau:(-\infty,t]\to\R^2$ such that $\tau(s)\in\NU_1^{\tht\sig}$ for all $s<t$.
\end{defn}

From the coalescence \eqref{geo:coal1}-\eqref{geo:coal2}), we see that, for $\w\in\Omega_{15}$, $\tht\in\R$ and $\sigg\in\{-,+\}$, the collection 
\begin{equation}\label{Ttree}
    \mathcal T^{\tht\sig}=\bigcup_{\substack{x,s\in\R\\ S\in\{L,M,R\}}}\geo\from{(x,s)}\dir{S}{\tht}{\sig}((s,\infty))
\end{equation}
of $W^{\tht\sig}$-geodesics, with their initial points removed, forms a tree. The next lemma says that the collection of $\tht\sigg$ shock interfaces, with their initial points removed, is exactly the dual of $\mathcal T^{\tht\sig}$. 

\begin{lem}\label{lm:shockint}
For any $\w\in\Omega_{15}$, $t,\tht\in\R$, $\sigg\in\{-,+\}$, 
$\tau:(-\infty,t]\to\R^2$ is a $\tht\sigg$ shock interface if, and only if, it is a continuous semi-infinite space-time path such that $\tau(s)\notin\mathcal T^{\tht\sig}$ for all $s<t$.
\end{lem}

\begin{proof}
   \eqref{geo:restart} implies that the relative interior of any $W^{\tht\sig}$-geodesic has no $\tht\sigg$ shock points:
\begin{align}\label{noshcoksongeo}
\NU_1^{\tht\sig}\subset\R^2\setminus\mathcal T^{\tht\sig}.
\end{align}
This gives one direction.  The other direction is in  \cite[Lemma 33]{Bha-24}.
\end{proof}

A consequence of this lemma is that given any $W^{\tht\sig}$-geodesic $\gamma:[s,\infty)\to\R^2$ and any $\tht\sigg$ shock interface $\tau:(-\infty,t]\to\R^2$, 
\begin{align}\label{no-intersection}
\gamma((s,\infty))\cap\tau((-\infty,t))=\varnothing.
\end{align}

By the continuity \eqref{W:cont} and the growth rate \eqref{Wgrowth} of $W^{\tht\sig}$, $f(y)=W^{\tht\sig}(y,s;x,s)$ satisfies Definition \ref{def:init}.
Lemma 31 in \cite{Bha-24} shows that for each point $(x,s)\in\R^2$, taking $f$ as a boundary condition at time $s$ and defining $I^-_{(x,s)}$ and $I^+_{(x,s)}$ as in \eqref{CIs} makes the space-time paths
\begin{align}\label{Upsilondef}
\Upsilon_{(x,s)}\dir{L}{\tht}{\sig}(r)=\bigl(I^-_{(x,s)}(r),r\bigr)\quad\text{and}\quad
\Upsilon_{(x,s)}\dir{R}{\tht}{\sig}(r)=\bigl(I^+_{(x,s)}(r),r\bigr),\quad r\le s,
\end{align}
avoid $\mathcal T^{\tht\sig}$. 
By \eqref{I-cont}, they are also continuous. Thus,
by Lemma \ref{lm:shockint}, these are $\tht\sigg$ shock interfaces.
This shows that $\tht\sigg$ shock interfaces do exist from every point in $\R^2$. Furthermore, by \eqref{Wgrowth} and \eqref{Idir}, we have for each $S\in\{L,R\}$,
\begin{align}\label{Upsilondir}
\lim_{r\to-\infty} r^{-1}\Upsilon_{(x,s)}\dir{S}{\tht}{\sig}(r)=(\tht,1). 
\end{align}

The next lemma says that every $\tht\sigg$ shock is in the relative interior of some $\tht\sigg$ shock interface. The proof is left to Appendix \ref{sec:auxlems}. 

\begin{lem}\label{lm:shockintallshocks}
For all $\w\in\Omega_{15}$, $\tht\in\R$, and $\sigg\in\{-,+\}$,
\[\mathcal I^{\tht\sig}=\bigcup\bigl\{\tau((-\infty,s)):\tau:(-\infty,s]\to\R^2\text{ is a $\tht\sigg$ shock interface}\bigr\}=\NU_1^{\tht\sig}.\]
\end{lem}

Since $\mathcal T^{\tht\sig}$ is a tree, it contains no cycles; consequently, $\mathcal I^{\tht\sig}$ also contains no cycles and therefore forms a forest (that is, a disjoint union of trees).
 For the details, see Lemma 30 in \cite{Bha-24}. Since the existence of two trees in this forest would imply the existence of a bi-infinite geodesic, \eqref{nobiinfinite} implies that 
\begin{align}\label{Itree}
\mathcal I^{\tht\sig}\text{ is a tree. Hence, any two $\tht\sigg$ shock interfaces must coalesce.}
\end{align}
By coalescence we mean that if two interfaces start from distinct points, they eventually meet and merge at their first intersection, after which they coincide. This includes the case where one interface starts at a relative interior point of the other, in which case that point is the first intersection and hence the coalescence point, so the former simply follows the latter from that time onward. If the interfaces start from the same point, then either they coincide identically, or they initially separate and later reintersect, merging from that point onward.

We now state a number of new preliminary results about shock interfaces, the proofs of which are deferred to  Appendix \ref{sec:auxlems}.

Lemma 33 in \cite{Bha-24} states that the left and right $W^{\tht\sig}$-geodesics emanating from any point in the relative interior of a $\tht\sigg$ shock interface $\tau:(-\infty,t]\to\R^2$ are split by $\tau$: 
\begin{align}\label{shocksplitsgeo}
\forall r<s<t,\ \geo\from{\tau(r)}\dir{L}{\tht}{\sig}(s)<\tau(r)<\geo\from{\tau(r)}\dir{R}{\tht}{\sig}(s).
\end{align}

Conversely, Lemma 32 in \cite{Bha-24} says that if $(x,s)\in\mathcal T^{\tht\sig}$, then the interfaces $\Upsilon_{(x,s)}\dir{L}{\tht}{\sig}$ and $\Upsilon_{(x,s)}\dir{R}{\tht}{\sig}$ are distinct. The following strengthens this result.

\begin{lem}\label{lm:geosplitsshocks}
  For each $\w\in\Omega_{15}$, $x,s,\tht\in\R$, and $\sigg\in\{-,+\}$, for any $W^{\tht\sig}$-geodesic $\gamma$ out of $(x,s)$, for all $t>r>s$,
  $\Upsilon_{\gamma(t)}\dir{L}{\tht}{\sig}(r)<\gamma(r)<\Upsilon_{\gamma(t)}\dir{R}{\tht}{\sig}(r)$.
\end{lem}

Our next theorem uses \eqref{shocksplitsgeo} to 
strengthen the convergences in \eqref{geo:lim}. 

\begin{thm}\label{th:convanydir}
   There exists an event $\Omega_0\subset\Omega_{15}$ such that $\P(\Omega_0)=1$ and for all $\w\in\Omega_0$, $x,s,\tht\in\R$, and $\sigg\in\{+,-\}$, the following holds. For any sequences $(x_n,s_n)\to(x,s)$ and $S_n\in\{L,M,R\}$, there exists a subsequence $n_k$ and an $S\in\{L,M,R\}$ such that for any $t>s$, there exists a $k_0$ such that 
   \[\geo\from{(x_{n_k},s_{n_k})}\dir{S_{n_k}}{\tht}{\sig}(r) = \geo\from{(x,s)}\dir{S}{\tht}{\sig}(r),\quad\text{for all }k\ge k_0\text{ and all }r\ge t.
   \]
\end{thm}

As an immediate corollary, we get that the convergences in \eqref{geo:lim} hold in the overlap topology.

\begin{cor}\label{cor:geo:lim}
Let $\w\in\Omega_0$, $x,s,\tht\in\R$, $\sigg\in\{-,+\}$, and $S\in\{L,M,R\}$. Then for each $r>s$ there exists $\e>0$ such that 
\begin{align}\label{geo:lim2}
  x-\e<z<x\Rightarrow\geo\from{(z,s)}\dir{S}{\tht}{\sig}\big|_{[r,\infty)}=\geo\from{(x,s)}\dir{L}{\tht}{\sig}\big|_{[r,\infty)}\quad\text{and}\quad
  x<y<x+\e\Rightarrow\geo\from{(y,s)}\dir{S}{\tht}{\sig}\big|_{[r,\infty)}=\geo\from{(x,s)}\dir{R}{\tht}{\sig}\big|_{[r,\infty)}.
\end{align}
\end{cor}

Returning to shock interfaces, we prove the following ordering and directedness result.

\begin{lem}\label{lm:shocksandwich}
  For all $\w\in\Omega_0$, $x,s,\tht\in\R$, $\sigg\in\{-,+\}$, and any $\tht\sigg$ shock interface $\tau$ from $(x,s)$, 
  \begin{align}\label{shocksmono}
  \Upsilon_{(x,s)}\dir{L}{\tht}{\sig}\preceq\tau\preceq\Upsilon_{(x,s)}\dir{R}{\tht}{\sig}.
  \end{align}
  Consequently, 
  \begin{align}\label{shocksdir}
  \lim_{r\to-\infty}r^{-1}\tau(r)=(\tht,1).
  \end{align}
\end{lem}

By the directedness property \eqref{shocksdir}, if $\tht<\tht'$, then for any $\sigg,\sigg'\in\{-,+\}$, and any $\tht\sigg$ shock interface $\tau$ and $\tht'\sigg'$ shock interface $\tau'$, we have
$\tau(r)>\tau'(r)$ for all sufficiently large (negative) $r$.
The next lemma gives a similar asymptotic ordering for $\tht-$ and $\tht+$ shock interfaces.

\begin{lem}\label{lm:shockasymorder}
  For all $\w\in\Omega_0$, $s$, $\tht$, and $x<y$ in $\R$, and any $\tht+$ shock interface $\tau^+$ from $(x,s)$ and $\tht-$ shock interface $\tau^-$ from $(y,s)$, 
  \begin{align}\label{+<-}
  \tau^+(r)\le\tau^-(r)\quad\forall r\le s.
  \end{align}
  If $x=y$, then there exists an $r_0\le s$ such that
  \begin{align}\label{+<=-}
  \tau^+(r)\le\tau^-(r)\quad\forall r\le r_0.
  \end{align}
\end{lem}


Next, we use \eqref{Icv} to deduce the following convergence result.

\begin{lem}\label{lm:shocklims}
  For all $\w\in\Omega_0$, $x,s,\tht\in\R$, $S\in\{L,R\}$, and $\sigg\in\{-,+\}$, 
  \begin{align}\label{shocklims}
  \forall r\le t\,:\,\Upsilon_{(x,s)}\dir{L}{\tht}{\sig}(r)=\lim_{z\nearrow x}\Upsilon_{(z,s)}\dir{S}{\tht}{\sig}(r)
  \quad\text{and}\quad
  \Upsilon_{(x,s)}\dir{R}{\tht}{\sig}(r)=\lim_{z\searrow x}\Upsilon_{(z,s)}\dir{S}{\tht}{\sig}(r).
  \end{align}
\end{lem}

The next two lemmas assert that coalescence points of geodesics, respectively of interfaces, are precisely the points from which three distinct interfaces, respectively geodesics, emanate.

\begin{lem}\label{lm:intcoal}
   Let $\w\in\Omega_0$, $\tht\in\R$, and $\sigg\in\{-,+\}$. Let $t,t'\in\R$ and suppose $\tau$ and $\tau'$ are two $\tht\sigg$ shock interfaces that start at levels $t$ and $t'$, respectively, and coalesce at $(x,s)$ with $s<t\wedge t'$. Then $\geo\from{(x,s)}\dir{M}{\tht}{\sig}$ is distinct from both $\geo\from{(x,s)}\dir{S}{\tht}{\sig}$, $S\in\{L,R\}$, and is the unique $W^{\tht\sig}$-geodesic  out of $(x,s)$ that remains strictly between $\tau$ and $\tau'$ on the time interval $(s,t\wedge t')$. Consequently, there does not exist a point $(x,s)$ at which three $\tht\sigg$ shock interfaces coalesce. 
\end{lem}

\begin{lem}\label{lm:geocoal}
   Let $\w\in\Omega_0$, $\tht\in\R$, and $\sigg\in\{-,+\}$. Let $r,r'\in\R$ and suppose $\gamma$ and $\gamma'$ are two $W^{\tht\sig}$-geodesics that start at levels $r$ and $r'$, respectively, and coalesce at $(x,s)$ with $s>r\vee r'$. Then there exists a $\tht\sigg$ shock interface $\tau$ out of $(x,s)$ that remains strictly between $\gamma$ and $\gamma'$ on the time interval $(r\vee r',s)$. 
\end{lem}

The last result of the section says that shocks are dense at every time level. This stands in sharp contrast to Lemma 8.1 in \cite{Ras-Swe-24-a-}, which establishes that in the Brownian last-passage percolation model, shocks are nowhere dense at any time level.

\begin{lem}\label{lem:shocksdense}
Let $\w\in\Omega_0$, $s,\tht\in\R$, and $\sigg\in\{+,-\}$. Then $\NU_1^{\tht\sig}\cap\bigl(\R\times\{s\}\bigr)$ is dense in $\R\times\{s\}$.
\end{lem}

Henceforth, we fix $\w \in \Omega_0$. All subsequent statements and proofs are therefore deterministic in nature, relying on a combination of analytical and geometric arguments.

\section{Characterization of instability points}\label{sec:inst}

We begin the development by giving two equivalent characterizations of instability points: an analytic one, which we adopt as the definition, and a geometric one. 

\begin{defn}\label{inc-pts}
For a non-decreasing function $f:\R\to\R$, we call $x\in\R$ a \emph{point of increase} of $f$ if, for all $z<x$ and $y>x$, $f(z)<f(y)$.  
\end{defn}

By the cocycle property \eqref{cocycle} and the monotonicity \eqref{W:mono}, for $s,\tht$ and $y>x$ in $\R$,  
\begin{align*}
&\bigl(W^{\tht-}(0,s;y,s)-W^{\tht+}(0,s;y,s))-
\bigl(W^{\tht-}(0,s;x,s)-W^{\tht+}(0,s;x,s))\\
&\qquad\qquad\qquad
=W^{\tht-}(x,s;y,s)-W^{\tht+}(x,s;y,s)\ge0.
\end{align*}
Thus, the difference profile $y\to W^{\tht-}(0,s;y,s)-W^{\tht+}(0,s;y,s)$ is nondecreasing. By \eqref{W:cont}, it is also continuous in $y$.

\begin{defn}\label{def:buseinstab}
	For $\w\in\Omega_0$, $\tht,s,x\in\R$, we call $(x,s)$ a \emph{$\tht$-instability} point if $x$ is a point of increase of the function
	$y\mapsto W^{\tht-}(0,s;y,s)-W^{\tht+}(0,s;y,s)$. 
\end{defn}

Recalling \eqref{def:kSs}, the set of points $x$ such that $(x,s)$ is an instability point is denoted by $\IG\tht_s$. By Lemma \ref{ptinc}, it is closed.
    
\begin{defn}\label{def:stht}
    For $\tht \in \R$, we define the \emph{instability graph} $\IG{\tht}$ to be the union over all $\tht$-instability points: $\IG{\tht} = \bigcup_{s\in\R}\IG\tht_s$. See Figure \ref{fig:IG} for a simulation.
\end{defn}

The term ``graph'' is justified by the connectivity property in Theorem \ref{main:IG}\eqref{main:IG.graph}.  

\begin{rmk}
Definition \ref{def:stht} differs from \eqref{IGdef1} and is the formulation we adopt for $\IG\tht$. In the proof of Theorem \ref{main:IG}, given in Appendix \ref{sec:mainproofs}, we show that the two definitions are equivalent.
\end{rmk}

Recall Definition \ref{def:baddir} of the set $\baddir$ of exceptional directions. Recall from \eqref{countable} that $\baddir$ is a countable dense subset of $\R$.
The following is immediate from \eqref{exceptional} and the cocycle property \eqref{cocycle}.

\begin{lem}
    For all $\w\in\Omega_0$, $\IG\tht\ne\varnothing$ if and only if $\tht\in\baddir$.
\end{lem}

The next lemma implies that when $\tht\in\baddir$, $\tht$-instability points not only exist in $\R^2$ but occur at every time level. It follows from \eqref{Wunbounded}.

\begin{lem}\label{kS inf}
Let $\w\in\Omega_0$ and $\tht\in\baddir$. Then for each $s\in\R$, $\inf\IG\tht_s=-\infty$ and $\sup\IG\tht_s=\infty$.
\end{lem}

\begin{defn}\label{def:isolstht}
    For $\tht \in \baddir$ and $(x,s)\in \IG{\tht}$, we call $(x,s)$ \emph{left-isolated} if there exists an $\ep>0$ such that for all $y\in (x-\ep,x)$, $(y,s)\notin\IG{\tht}$. Similarly, we call $(x,s)$ \emph{right-isolated} if there exists an $\ep>0$ such that for all $y\in (x,x+\ep)$, $(y,s)\notin\IG{\tht}$.
\end{defn}

The next lemma follows from Lemma \ref{ptinc}\eqref{ptinc.c} and the continuity \eqref{W:cont}.

\begin{lem}\label{no-isolated}
Let $\w\in\Omega_0$ and $\tht\in\baddir$. Then there are no points $(x,s)\in\IG\tht$ that are both left- and right-isolated. 
\end{lem}

We next describe the geometric characterization of instability points.
The following lemma is immediate from Definition \ref{def:buseinstab} and \eqref{Wjump}. See the left panel in Figure \ref{fig:instint}. 

\begin{figure}[hpt]
    \includegraphics[width=3.5cm]{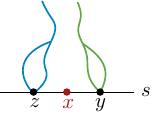}
    \qquad\qquad
    \includegraphics[width=3.5cm]{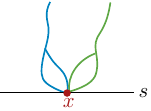}
    \caption{\small Left: $\geo\from{(z,t)}\dir{R}{\tht}{-}\cap \geo\from{(y,t)}\dir{L}{\tht}{+}=\varnothing$. Right: $\geo\from{(x,s)}\dir{L}{\tht}{-}\cap \geo\from{(x,s)}\dir{R}{\tht}{+}=\{(x,s)\}$. The right panel also depicts a snowbird shock (Definition \ref{def:snowbird}).}
    \label{fig:instint}
\end{figure}

\begin{lem}\label{def:geoinstabint}
	For all $\w\in\Omega_0$, $(x,s)\in\R^2$, and $\tht\in\baddir$, 
    $(x,s)\in\IG{\tht}$ if and only if for any $z<x$ and $y>x$, $\geo\from{(z,s)}\dir{R}{\tht}{-}\cap \geo\from{(y,s)}\dir{L}{\tht}{+}=\varnothing$.
\end{lem}

The next lemma improves the above criterion to one about geodesics from the same point. See the left panel in Figure \ref{fig:instint}.

\begin{lem}\label{def:geoinstabpt}
	For $\w\in\Omega_0$, $(x,s)\in\R^2$, and $\tht\in\baddir$, $(x,s)\in\IG{\tht}$ if and only if $\geo\from{(x,s)}\dir{L}{\tht}{-}\cap \geo\from{(x,s)}\dir{R}{\tht}{+} = \{(x,s)\}$.
\end{lem}

\begin{proof}
Suppose $(u,t)\in \geo\from{(x,s)}\dir{L}{\tht}{-}\cap \geo\from{(x,s)}\dir{R}{\tht}{+}$.  
Take $z<x$ and $y>x$, and define $(u',r')$ and $(u'',r'')$ to be the coalescence points of 
$\geo\from{(z,s)}\dir{R}{\tht}{-}$ with $\geo\from{(x,s)}\dir{L}{\tht}{-}$ 
and of 
$\geo\from{(y,s)}\dir{L}{\tht}{+}$ with $\geo\from{(x,s)}\dir{R}{\tht}{+}$, respectively.  
Corollary \ref{cor:geo:lim}  
implies that  
$r',r''<t$ for $z$ and $y$ sufficiently close to $x$.  
This implies $(u,t)\in \geo\from{(z,s)}\dir{R}{\tht}{-}\cap \geo\from{(y,s)}\dir{L}{\tht}{+}$.
By Lemma \ref{def:geoinstabint}, $(x,s)\notin \IG{\tht}$.  

For the converse direction, suppose that 
$\geo\from{(x,s)}\dir{L}{\tht}{-}\cap \geo\from{(x,s)}\dir{R}{\tht}{+} = \{(x,s)\}$.
Take $z<x$ and $y>x$.  
By the geodesic ordering \eqref{geo:mono}, we have
$\geo\from{(z,s)}\dir{R}{\tht}{-}\preceq \geo\from{(x,s)}\dir{L}{\tht}{-}$ and $\geo\from{(x,s)}\dir{R}{\tht}{+}\preceq \geo\from{(y,s)}\dir{L}{\tht}{+}$ and hence $\geo\from{(z,s)}\dir{R}{\tht}{-}\cap \geo\from{(y,s)}\dir{L}{\tht}{+} = \varnothing$.
Then Lemma \ref{def:geoinstabint} gives $(x,s)\in\IG{\tht}$.
\end{proof}

\begin{rmk}\label{Rk8.3}
The above lemma resolves the question posed in Remark 8.3 of \cite{Bus-Sep-Sor-24}.  
In their notation, it asserts that, for $\P$-almost every $\w$, and for all $s\in\R$ and $\tht\in\baddir$,  
the set of instability points $\mathcal D_{s,\tht}$ coincides with the set $\mathfrak S_{s,\tht}$ of points $x$ for which there exist two disjoint geodesics starting from $(x,s)$ and going in direction $\tht$. 
\end{rmk}

\section{Shock interfaces}\label{sec:shocks}

In this section, we present the first results concerning the interaction between shock interfaces and instability regions. Lemma \ref{lem:stableshocks} shows that, within stability regions, shocks exhibit no sign distinction. Then, Lemma \ref{lem:IGgodown} demonstrates that once a shock interface passes through an instability point, it subsequently propagates along instability points. Finally, Lemma \ref{LRisolgeo} establishes a connection between left- and right-isolated instability points and a specific class of shock points.

Recall that the age $\age^{\tht\sig}(x,s)$ of a shock $(x,s)\in\NU_1^{\tht\sig}$ is defined in \eqref{age}.

\begin{lem}\label{lem:stableshocks}
	Let $\w\in\Omega_0$ and $\tht\in\baddir$. 
    If $(x,s) \in(\NU_1^{\tht-}\cup\,\NU_1^{\tht+})\setminus\IG\tht$, then
    $(x,s)\in \NU_1^{\tht-}\cap \NU_1^{\tht+}$, $\age^{\tht-}(x,s)=\age^{\tht+}(x,s)=a$,  
    $\geo\from{(x,s)}\dir{S}{\tht}{-}\big|_{[s,s+a]}=\geo\from{(x,s)}\dir{S}{\tht}{+}\big|_{[s,s+a]}$ for both $S\in\{L,R\}$, and the point $\geo\from{(x,s)}\dir{S}{\tht}{\sig}(s+a)$ at which this double shock resolves is not in $\IG\tht$. We call this shock a \emph{proper double shock}. See Figure \ref{fig:proper} and the last three configurations on the top row of Figure \ref{fig:geodesics}.
\end{lem}

\begin{figure}[hpt]
    \includegraphics[width=3.5cm]{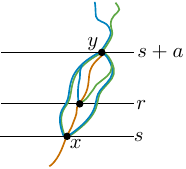}
    \caption{\small An illustration of the proof of Lemma \ref{lem:stableshocks}. $(x,s)$ is a stable point that is a double shock. Taking a point on the shock interface from the coalescence point creates two geodesics that separate immediately, then rejoin at $(y,s+a)$.}    
    \label{fig:proper}
\end{figure}

\begin{proof}
    By Lemma \ref{def:geoinstabpt}, there exists a $t>s$ such that $\geo\from{(x,s)}\dir{L}{\tht}{-}(t)=\geo\from{(x,s)}\dir{R}{\tht}{+}(t)$. The geodesic ordering \eqref{geo:mono} implies then that, for both $S\in\{L,R\}$, $\geo\from{(x,s)}\dir{S}{\tht}{-}(t)=\geo\from{(x,s)}\dir{S}{\tht}{+}(t)$ and, by the extremality  \eqref{p2pleftmost}-\eqref{p2prightmost}, $\geo\from{(x,s)}\dir{S}{\tht}{-}$ and $\geo\from{(x,s)}\dir{S}{\tht}{+}$ match on the time interval $[s,t]$. See Figure \ref{fig:proper}. The claims of the lemma follow, except for the last one. 

    Let $(y,s+a)=\geo\from{(x,s)}\dir{S}{\tht}{\sig}(s+a)$. As we have just shown, this is the same point for all $S\in\{L,R\}$ and $\sigg\in\{-,+\}$. We prove that $(y,s+a)\notin\IG\tht$.

    By Lemma \ref{lm:geocoal}, there exists a $\tht-$ shock interface $\tau$ out of $(y,s+a)$ that remains strictly between $\geo\from{(x,s)}\dir{L}{\tht}{-}$ and $\geo\from{(x,s)}\dir{R}{\tht}{-}$ on the time interval $(s,s+a)$.  Take a rational time $r\in(s,s+a)$. 
    By \eqref{shocksplitsgeo}, on the interval $(r,s+a)$, $\geo\from{\tau(r)}\dir{L}{\tht}{-}$ must remain strictly left of $\tau$ and  $\geo\from{\tau(r)}\dir{R}{\tht}{-}$ goes strictly right of it. By the ordering \eqref{geo:mono}, $\geo\from{\tau(r)}\dir{R}{\tht}{+}\succeq\geo\from{\tau(r)}\dir{R}{\tht}{-}$. Consequently, 
    $\geo\from{\tau(r)}\dir{L}{\tht}{-}$ and $\geo\from{\tau(r)}\dir{R}{\tht}{+}$ cannot reintersect on $(r,s+a)$.
    On the other hand, the coalescence \eqref{geo:coal1} keeps the two geodesics weakly between $\geo\from{(x,s)}\dir{L}{\tht}{-}$ and $\geo\from{(x,s)}\dir{R}{\tht}{+}$, and hence forces them to intersect at $(y,s+a)$.  Now Lemma \ref{lm:Duncan}\eqref{Duncan.a} prohibits the two geodesics from immediately separating at $(y,s+a)$. By \eqref{geo:restart}, we get that $\geo\from{(y,s+a)}\dir{L}{\tht}{-}$ and $\geo\from{(y,s+a)}\dir{R}{\tht}{+}$ separate at a time $>s+a$ and, by Lemma \ref{def:geoinstabint}, $(y,s+a)\notin\IG\tht$.
\end{proof}

Next, we show that shock interfaces emanating from instability points continue running along instability points.

\begin{lem}\label{lem:IGgodown}
    Let $\w\in\Omega_0$, $\tht\in\baddir$, $(x,s)\in\IG\tht$, and $\sigg \in \{+,-\}$. Then for any $\tht\sigg$ shock interface $\gamma$ out of $(x,s)$, $\gamma\subset\IG\tht$. 
\end{lem}

\begin{proof}
The claim is clear for $r=s$, since $\gamma(s)=(x,s)$. Fix $r<s$.  
Suppose that $\gamma(r)\notin \IG\tht$.  
By the geodesic ordering \eqref{geo:mono}, we have 
$\geo\from{\gamma(r)}\dir{L}{\tht}{-} \preceq \geo\from{\gamma(r)}\dir{L}{\tht}{\sig}\preceq\geo\from{\gamma(r)}\dir{R}{\tht}{\sig} \preceq \geo\from{\gamma(r)}\dir{R}{\tht}{+}$.
By \eqref{shocksplitsgeo}, the interface $\gamma$ separates 
$\geo\from{\gamma(r)}\dir{L}{\tht}{\sig}$ and 
$\geo\from{\gamma(r)}\dir{R}{\tht}{\sig}$ 
and prevents them from touching over the time interval $(r,s)$.  
It follows that 
$\geo\from{\gamma(r)}\dir{L}{\tht}{-}$ and 
$\geo\from{\gamma(r)}\dir{R}{\tht}{+}$ 
cannot intersect on $(r,s)$, and moreover,
$\geo\from{\gamma(r)}\dir{L}{\tht}{-}(s) \le (x,s) \le 
\geo\from{\gamma(r)}\dir{R}{\tht}{+}(s)$.

Now, by \eqref{geo:restart}, the geodesic ordering \eqref{geo:mono}, the fact that $(x,s)\in\IG\tht$, and the characterization in Lemma \ref{def:geoinstabpt}, we obtain
\[
\geo\from{\gamma(r)}\dir{L}{\tht}{-}(t)
\le \geo\from{(x,s)}\dir{L}{\tht}{-}(t)
< \geo\from{(x,s)}\dir{R}{\tht}{+}(t)
\le \geo\from{\gamma(r)}\dir{R}{\tht}{+}(t)
\quad \text{for all } t > s.
\]
But then since $\gamma(r)\notin\IG\tht$, Lemma \ref{def:geoinstabpt} implies that we must have
$\geo\from{\gamma(r)}\dir{L}{\tht}{-}(s)
= (x,s) 
= \geo\from{\gamma(r)}\dir{R}{\tht}{+}(s)$.
Together with the geodesic ordering \eqref{geo:mono}, this forces 
$\geo\from{\gamma(r)}\dir{L}{\tht}{-}(s)=\geo\from{\gamma(r)}\dir{R}{\tht}{-}(s)=\geo\from{\gamma(r)}\dir{R}{\tht}{+}(s)=(x,s)$. Then \eqref{p2prightmost} implies that $\geo\from{\gamma(r)}\dir{R}{\tht}{-}$ and $\geo\from{\gamma(r)}\dir{R}{\tht}{+}$ match on the time interval $[r,s]$. In particular,  $\geo\from{\gamma(r)}\dir{L}{\tht}{-}$ and $\geo\from{\gamma(r)}\dir{R}{\tht}{-}$ do not intersect on $(r,s)$ and meet at $(x,s)$.

We have thus shown that if $\gamma(r)\notin\IG\tht$, then $\gamma(r)\in\NU_1^{\tht-}\setminus\IG\tht$ and $\age^{\tht-}(\gamma(r))=s-r$. By Lemma \ref{lem:stableshocks}, we get that $(x,s)\notin\IG\tht$, which contradicts the assumption that $(x,s)$ is unstable. Thus, $\gamma(r)\in\IG\tht$. 
\end{proof}

Now, we characterize left- and right-isolated instability points in terms of hugging shocks. See the third row of Figure \ref{fig:geodesics} for an illustration of $\tht+$ hugging shocks and the fourth row of that figure for $\tht-$ hugging shocks. 

\begin{defn}\label{def:hugging}
For $\tht\in\baddir$, a point $(x,s)\in\R^2$ is a \emph{$\tht+$ hugging shock} if $(x,s)\in\IG\tht$ and $\geo\from{(x,s)}\dir{L}{\tht}{-}\cap\geo\from{(x,s)}\dir{L}{\tht}{+}\ne\{(x,s)\}$. Similarly, $(x,s)$ is a \emph{$\tht-$ hugging shock} if $(x,s)\in\IG\tht$ and $\geo\from{(x,s)}\dir{R}{\tht}{-}\cap\geo\from{(x,s)}\dir{R}{\tht}{+}\ne\{(x,s)\}$. 
\end{defn}

The next lemma explains the term ``hugging''. 


\begin{lem}\label{lm:hugging}
   Let $\w\in\Omega_0$, $\tht\in\baddir$, and $(x,s)\in\IG\tht$. 
   \begin{enumerate} [label={\rm(\alph*)}, ref={\rm\alph*}] \itemsep=1pt 
   \item\label{lm:hugging.a} For each $\sigg\in\{-,+\}$, a $\tht\sigg$ hugging shock is in $\NU_1^{\tht\sig}$.
   \item\label{lm:hugging.b} $(x,s)$ is a $\tht+$ hugging shock if and only if there exists a $\delta>0$ such that  $\geo\from{(x,s)}\dir{L}{\tht}{-}(t)=\geo\from{(x,s)}\dir{L}{\tht}{+}(t)$ for all $t\in[s,s+\delta]$.  
    \item\label{lm:hugging.c} $(x,s)$ is a $\tht-$ hugging shock if and only if there exists a $\delta>0$ such that  $\geo\from{(x,s)}\dir{R}{\tht}{-}(t)=\geo\from{(x,s)}\dir{R}{\tht}{+}(t)$ for all $t\in[s,s+\delta]$.
    \end{enumerate}
\end{lem}

\begin{proof}
 Part \eqref{lm:hugging.a}. Since $\geo\from{(x,s)}\dir{L}{\tht}{-}\cap\geo\from{(x,s)}\dir{R}{\tht}{+}=\{(x,s)\}$ while $\geo\from{(x,s)}\dir{L}{\tht}{-}\cap\geo\from{(x,s)}\dir{L}{\tht}{+}\neq\{(x,s)\}$, we have $\geo\from{(x,s)}\dir{L}{\tht}{+}\ne\geo\from{(x,s)}\dir{R}{\tht}{+}$ and, by \eqref{geo:no-lollipop}, $(x,s)\in\NU_1^{\tht+}$.
 The case $\sigg=-$ is similar.\smallskip

 Parts \eqref{lm:hugging.b} and \eqref{lm:hugging.c} follow from the extremality in, respectively, \eqref{p2pleftmost} and \eqref{p2prightmost}.
\end{proof}

Recall Definition \ref{def:isolstht} of left- and right-isolated instability points.

\begin{lem}\label{LRisolgeo}
    Let $\w\in\Omega_0$, $\tht\in\baddir$, and $(s,x)\in\IG\tht$. 
    \begin{enumerate} [label={\rm(\alph*)}, ref={\rm\alph*}] \itemsep=1pt \item\label{LRisolgeo.a} $(x,s)$ is left-isolated, respectively right-isolated, if and only if $(x,s)$ is a $\tht+$, respectively $\tht-$, hugging shock. 
    \item\label{LRisolgeo.b} The point $(x,s)$ cannot simultaneously be a $\tht-$ hugging shock and a $\tht+$ hugging shock.
    \end{enumerate}
\end{lem}

\begin{proof}
Part \eqref{LRisolgeo.a}. We prove the claim regarding left-isolated points, the case of right-isolated ones being similar.
Suppose $(x,s)$ is not a $\tht+$ hugging shock. Then, by Lemma \ref{lm:hugging}\eqref{lm:hugging.b}, $\geo\from{(x,s)}\dir{L}{\tht}{-}\cap\geo\from{(x,s)}\dir{L}{\tht}{+}=\{(x,s)\}$. By the geodesic ordering \eqref{geo:mono}, for any $z<x$, $\geo\from{(z,s)}\dir{R}{\tht}{-}\preceq\geo\from{(x,s)}\dir{L}{\tht}{-}$ and hence
 $\geo\from{(z,s)}\dir{R}{\tht}{-}\cap\geo\from{(x,s)}\dir{L}{\tht}{+}=\varnothing$.  
    By \eqref{Wjump}, this implies $W^{\tht-}(z,s;x,s) > W^{\tht+}(z,s;x,s)$, for all $z<x$.
    By \eqref{W:cont}, $W^{\tht-}$ and $W^{\tht+}$ are continuous, and  Lemma \ref{ptinc}\eqref{ptinc.b} and Definition \ref{def:buseinstab} tell us $(x,s)$ is not left-isolated.
        
    For the other direction, suppose that $(x,s)$ is a $\tht+$ hugging shock. Then there exists $t>s$ such that $\geo\from{(x,s)}\dir{L}{\tht}{-}(t)=\geo\from{(x,s)}\dir{L}{\tht}{+}(t)$.
    By the limits \eqref{geo:lim2}, 
    we have that for $z<x$ sufficiently close to $x$, $\geo\from{(z,s)}\dir{L}{\tht}{-}(t)
=\geo\from{(x,s)}\dir{L}{\tht}{-}(t)
=\geo\from{(x,s)}\dir{L}{\tht}{+}(t)
=\geo\from{(z,s)}\dir{R}{\tht}{+}(t)$.
By Lemma \ref{def:geoinstabpt}, this implies $(z,s)\notin\IG{\tht}$. Therefore, $(x,s)$ is left-isolated.\smallskip

 Part \eqref{LRisolgeo.b} follows from part \eqref{LRisolgeo.a} and Lemma \ref{no-isolated}.
\end{proof}

\section{The instability graph}

This section develops part of the material needed for the proof of Theorem \ref{main:IG}.
 The main development is the construction, in Proposition \ref{prop:IGgoup}, of an uncountable family of bi-infinite interfaces that partition the instability graph. Such interfaces were also constructed in the recent work \cite{Bha-Bus-Sor-25-}; however, our method of construction is different (see Remark \ref{rk-I} for more).

\begin{lem}\label{lem:IGisclosed}
    For any $\w\in\Omega_0$ and $\tht\in\baddir$, $\IG\tht$ is closed and nowhere dense, i.e.\ has empty interior. 
\end{lem}

\begin{proof}
 To prove that $\IG\tht$ is closed, we argue by contradiction. 
 Suppose there exists a sequence $(x_n,s_n)\in\IG\tht$ that converges to a point $(x,s)\not\in\IG\tht$. By Lemma \ref{def:geoinstabpt}, $\geo\from{(x,s)}\dir{L}{\tht}{-}$ and $\geo\from{(x,s)}\dir{L}{\tht}{+}$ must intersect at some point $(y,t)$ with $t>s$. By the geodesic ordering \eqref{geo:mono}, $\geo\from{(x,s)}\dir{S}{\tht}{\sig}$ must pass through $(y,t)$, for all $S\in\{L,M,R\}$ and $\sigg\in\{-,+\}$. By Theorem \ref{th:convanydir}, once $(x_n,s_n)$ is close enough to $(x,s)$, $\geo\from{(x_n,s_n)}\dir{L}{\tht}{-}(t)=\geo\from{(x,s)}\dir{S}{\tht}{-}(t)=(y,t)$, for some $S\in\{L,M,R\}$, and $\geo\from{(x_n,s_n)}\dir{R}{\tht}{+}(t)=\geo\from{(x,s)}\dir{S'}{\tht}{+}(t)=(y,t)$, for some $S'\in\{L,M,R\}$. Lemma \ref{def:geoinstabpt} implies then that $(x_n,s_n)\not\in\IG\tht$, a contradiction.

 By \eqref{Q notin IG} and Lemma \ref{def:geoinstabpt}, 
 \begin{align}\label{QnotIG}
 \IG\tht\subset\R^2\setminus\Q^2.
 \end{align}
This and the just proven fact that $\IG\tht$ is closed imply that $\IG\tht$ is nowhere dense. 
\end{proof}

The following result is a consequence of the closedness of $\IG\tht$ and is useful in the sequel. 

\begin{lem}\label{lm:isodense}
Take any $\w\in\Omega_0$, $\tht\in\baddir$, and any $x<y$ and $s$ in $\R$. Suppose $(x,s)$ and $(y,s)$ are both in $\IG\tht$ and that either $(x,s)$ is not right-isolated or $(y,s)$ is not left-isolated. Then there exist infinitely many $z\in(x,y)$ such that $(z,s)$ is not left-isolated and there exist infinitely many $z\in(x,y)$ such that $(z,s)$ is not right-isolated. 
\end{lem}

\begin{proof}
    We present the proof in the case where $(x,s)$ is not right-isolated; the complementary case is analogous. It suffices to show the existence of points $z$ and $z'$ in $(x,y)$ such that $(z,s)$ is not right-isolated and $(z',s)$ is not left-isolated. The argument can then be iterated inductively to produce infinitely many such points.
    
    We begin with finding a point $z$ as above. Take $z_1<z_2$ in $(x,y)$ such that $(z_i,s)\in\IG\tht$, $i\in\{1,2\}$. Such points exist, since $(x,s)$ is not right-isolated. If $(z_i,s)$ is not right-isolated, for some $i\in\{1,2\}$, then set $z=z_i$.  If, on the other hand, both points are right-isolated, then define $z=\inf\{a>z_1:(a,s)\in\IG\tht\}$. Since $\IG\tht$ is closed, $(z,s)\in\IG\tht$. Since $(z_1,s)$ is right-isolated, $z>z_1>x$. Since $(z_2,s)\in\IG\tht$, $z\le z_2<y$. The minimality of $z$ and the fact that $z>z_1$ imply that $(z,s)$ is left-isolated. By Lemma \ref{no-isolated}, $(z,s)$ is not right-isolated.

   The point $z'$ can be constructed in an analogous manner. If either $(z_1,s)$ or $(z_2,s)$ is not left-isolated, then we are done. Otherwise, if both points are left-isolated, define $z'=\sup\{a<z_2 : (a,s)\in\IG\tht\}$. As in the previous paragraph, it follows that $(z',s)\in\IG\tht$ and is not left-isolated.
\end{proof}

The following is the main construction in this section. 

\begin{prop}\label{prop:IGgoup}
    Let $\w\in\Omega_0$. There exists a process $\bigl\{\Ipath_{(x,s)}^{\tht\sig}:\tht\in\baddir,(x,s)\in\IG\tht,\sigg\in\{-,+\}\bigr\}$ of bi-infinite space-time paths such that the following hold for all $\tht\in\baddir$ and $(x,s)\in\IG\tht$.
    \begin{enumerate} [label={\rm(\alph*)}, ref={\rm\alph*}] \itemsep=1pt 
    \item\label{IGgoup.a} For each $\sigg\in\{-,+\}$, $\Ipath_{(x,s)}^{\tht\sig}:\R\to\R^2$ is a continuous space-time path.
    \item\label{IGgoup.b} For all $t\in\R$, $\Ipath_{(x,s)}^{\tht-}(t)\le \Ipath_{(x,s)}^{\tht+}(t)$.
    \item\label{IGgoup.c} For all $t\in\R$ and $\sigg\in\{-,+\}$, $\Ipath_{(x,s)}^{\tht\sig}(t)\in\IG\tht$.
    \item\label{IGgoup.d} For all $t\in\R$, if $\Ipath_{(x,s)}^{\tht-}(t)<\Ipath_{(x,s)}^{\tht+}(t)$, then $\Ipath_{(x,s)}^{\tht-}(t)$ is right-isolated and $\Ipath_{(x,s)}^{\tht+}(t)$ is left-isolated. Conversely, if $\Ipath_{(x,s)}^{\tht-}(t)=\Ipath_{(x,s)}^{\tht+}(t)$, then this point is neither left- nor right-isolated.
    \item\label{IGgoup.e} $(x,s)\in\{\Ipath_{(x,s)}^{\tht-}(s),\,\Ipath_{(x,s)}^{\tht+}(s)\}$.
Moreover, $\Ipath_{(x,s)}^{\tht-}(s)=(x,s)$ if and only if $(x,s)$ is not left-isolated, and
$\Ipath_{(x,s)}^{\tht+}(s)=(x,s)$ if and only if $(x,s)$ is not right-isolated. 
    \item\label{IGgoup.f} For any $r\in\R$ and $\sigg\in\{-,+\}$, $\Ipath_{\Ipath_{(x,s)}^{\tht\sig}(r)}^{\tht-}=\Ipath_{(x,s)}^{\tht-}$ and $\Ipath_{\Ipath_{(x,s)}^{\tht\sig}(r)}^{\tht+}=\Ipath_{(x,s)}^{\tht+}$.
    \item\label{IGgoup.g} For any $(x',s')\in\IG\tht$ and $\sigg\in\{-,+\}$, either $\Ipath_{(x',s')}^{\tht\sig}=\Ipath_{(x,s)}^{\tht\sig}$ or $\Ipath_{(x',s')}^{\tht\sig}\cap\Ipath_{(x,s)}^{\tht\sig}=\varnothing$.
    \item\label{IGgoup.h} For any $s'\in\R$, 
    \begin{align}\label{Isplit-strict}
    \begin{split}
    &\geo\from{\Ipath_{(x,s)}^{\tht-}(s')}\dir{L}{\tht}{-}\big|_{(s',\infty)}\prec\Ipath_{(x,s)}^{\tht-}\big|_{(s',\infty)}\preceq\geo\from{\Ipath_{(x,s)}^{\tht-}(s')}\dir{R}{\tht}{+}\big|_{(s',\infty)}\quad\text{and}\\
    &\geo\from{\Ipath_{(x,s)}^{\tht+}(s')}\dir{L}{\tht}{-}\big|_{(s',\infty)}\preceq\Ipath_{(x,s)}^{\tht+}\big|_{(s',\infty)}\prec\geo\from{\Ipath_{(x,s)}^{\tht+}(s')}\dir{R}{\tht}{+}\big|_{(s',\infty)}.
    \end{split}
    \end{align}
    \item\label{IGgoup.i} For any $s'\in\R$ and $\sigg\in\{-,+\}$, $\Upsilon\from{\Ipath_{(x,s)}^{\tht\sig}(s')}\dir{L}{\tht}{+}\preceq\Ipath_{(x,s)}^{\tht\sig}\preceq\Upsilon\from{\Ipath_{(x,s)}^{\tht\sig}(s')}\dir{R}{\tht}{-}$.
    \item\label{IGgoup.j} For any $\sigg\in\{-,+\}$,
$\displaystyle\lim_{|t|\to\infty}t^{-1}\Ipath_{(x,s)}^{\tht\sig}(t)=(\tht,1)$.
\end{enumerate}
\end{prop}

\begin{rmk}
Lemma \ref{lem:IGgodown} says that if $(x,s)\in\IG\tht$, then the shock interfaces $\Upsilon_{(x,s)}^{S,\tht\sig}$, $\sigg\in\{-,+\}$, give two continuous space-time paths that go backward in time and remain in $\IG\tht$. These are different from the paths $\Ipath_{(x,s)}^{\tht\sig}$, since  $\Upsilon_{(x,s)}^{S,\tht\sig}$ and $\Upsilon_{(x',s')}^{S,\tht\sig}$ always coalesce, while $\Ipath_{(x,s)}^{S,\tht\sig}$ and $\Ipath_{(x',s')}^{S,\tht\sig}$  are either identical or disjoint.  
\end{rmk}

\begin{rmk}
Lemma \ref{I-prop}\eqref{I-prop.d}, proved in the course of establishing Proposition \ref{prop:IGgoup}, identifies the region strictly between the interfaces $\Ipath_{(x,s)}^{\tht\pm}$ as a stability region. Combining this with Lemma \ref{aux:RbndryLgeo}, we deduce that the weak inequalities in Proposition \ref{prop:IGgoup}\eqref{IGgoup.h} can become equalities at most at a single time. More precisely, equality can occur only if $\Ipath_{(x,s)}^{\tht-}(s')<\Ipath_{(x,s)}^{\tht+}(s')$, and then only at
\[
t=\inf\bigl\{r>s':\Ipath_{(x,s)}^{\tht-}(r)=\Ipath_{(x,s)}^{\tht+}(r)\bigr\}\in(s',\infty),
\]
which corresponds to the tip of the associated stability island.
\end{rmk}

\begin{rmk}\label{rk-I}
The interfaces $\Ipath^{\tht\sig}_{(x,s)}$ appearing in Proposition \ref{prop:IGgoup} were also constructed in Section 3 of \cite{Bha-Bus-Sor-25-} (see Proposition 3.9 there). In that work, semi-infinite interfaces are first built, using the competition interfaces of \cite{Rah-Vir-25}, and then shown to be consistent and to extend to bi-infinite paths; the properties listed in our proposition are subsequently established in Sections 4, 5, and 8.
By contrast, we define the bi-infinite interfaces directly (see \eqref{I}) and verify that this definition yields the desired properties, leading to a more concise and streamlined derivation. In addition to these properties, \cite{Bha-Bus-Sor-25-} shows that the region strictly between the interfaces $\Ipath^{\tht,\pm}_{(x,s)}$ consists of finite connected components with disjoint closures. As noted in the previous remark, this corresponds to a stability region. Accordingly, the finiteness and disjointness properties arise naturally from the analysis of stability regions, and we defer their proof to Lemma \ref{lm:disjointI}, where they follow directly from the stability structure. A semi-infinite version of these interfaces also appears in \cite[Equation (2.9)]{Dun-Sor-26}, where they are used for a different purpose.
\end{rmk}

The proof of Proposition \ref{prop:IGgoup} requires a number of intermediate lemmas. For $\w\in\Omega_0$, $\tht\in\baddir$, and $x_0,s_0\in\R$, define the function
    \begin{align}\label{h}
        h^\tht_{(x_0,s_0)}(y,t) = W^{\tht-}(y,t;x_0,s_0) \vee W^{\tht+}(y,t;x_0,s_0),\quad y,t\in\R.
    \end{align}

\begin{lem}
For any $\w\in\Omega_0$, $\tht\in\baddir$, and $x_0,s_0\in\R$, $h^\tht_{(x_0,s_0)}$ is an eternal solution: for any $x$ and $t>s$ in $\R$,
    \begin{align}\label{h-eternal}
        \sup_{y\in\R} \{\mathcal{L}(x,s;y,t)+h^\tht_{(x_0,s_0)}(y,t)\} = h^\tht_{(x_0,s_0)}(x,s).
    \end{align}
\end{lem}

\begin{proof} 
Apply \eqref{eternal} to write
\begin{align*}
h^\tht_{(x_0,s_0)}(x,s) 
&=\sup_{y\in\R} \{\mathcal{L}(x,s;y,t)+W^{\tht-}(y,t;x_0,s_0)\}
\vee\sup_{y\in\R} \{\mathcal{L}(x,s;y,t)+W^{\tht+}(y,t;x_0,s_0)\}\\
&=\sup_{y\in\R} \{\mathcal{L}(x,s;y,t)+W^{\tht-}(y,t;x_0,s_0)\vee W^{\tht+}(y,t;x_0,s_0)\}\\
&=\sup_{y\in\R} \{\mathcal{L}(x,s;y,t)+h^\tht_{(x_0,s_0)}(y,t)\}.\qedhere
\end{align*}
\end{proof}

For $\w\in\Omega_0$, $\tht\in\baddir$, and $x_0,s_0\in\R$, define 
\begin{align}\label{I}
\begin{split}
  I_{(x_0,s_0)}^{\tht+}(t)&=\inf\{y\in\R:W^{\tht-}(y,t;x_0,s_0)<W^{\tht+}(y,t;x_0,s_0)\}\quad\text{and}\\ 
  I_{(x_0,s_0)}^{\tht-}(t)&=\sup\{y\in\R:W^{\tht-}(y,t;x_0,s_0)>W^{\tht+}(y,t;x_0,s_0)\}.
\end{split}
\end{align}

\begin{lem}\label{I-prop}
The following hold for any $\w\in\Omega_0$, $\tht\in\baddir$, and $x_0,s_0,t\in\R$.

\begin{enumerate} [label={\rm(\alph*)}, ref={\rm\alph*}] \itemsep=1pt 
\item\label{I-prop.a} $-\infty<I_{(x_0,s_0)}^{\tht-}(t)\le I_{(x_0,s_0)}^{\tht+}(t)<\infty$.
\item\label{I-prop.b} $W^{\tht-}(y,t;x_0,s_0)<W^{\tht+}(y,t;x_0,s_0)$ for all $y>I_{(x_0,s_0)}^{\tht+}(t)$.
\item\label{I-prop.c} $W^{\tht-}(y,t;x_0,s_0)>W^{\tht+}(y,t;x_0,s_0)$ for all $y<I_{(x_0,s_0)}^{\tht-}(t)$.
\item\label{I-prop.d} $W^{\tht-}(y,t;x_0,s_0)=W^{\tht+}(y,t;x_0,s_0)$ for all $y\in[I_{(x_0,s_0)}^{\tht-}(t),I_{(x_0,s_0)}^{\tht+}(t)]$.
\item\label{I-prop.e} $I_{(x_0,s_0)}^{\tht-}(s_0)\le x_0\le I_{(x_0,s_0)}^{\tht+}(s_0)$.
\item\label{I-prop.f} $W^{\tht-}(I_{(x_0,s_0)}^{\tht+}(t),t;y,t)>W^{\tht+}(I_{(x_0,s_0)}^{\tht+}(t),t;y,t)$ for all $y>I_{(x_0,s_0)}^{\tht+}(t)$.
\item\label{I-prop.g} $W^{\tht-}(I_{(x_0,s_0)}^{\tht-}(t),t;y,t)<W^{\tht+}(I_{(x_0,s_0)}^{\tht-}(t),t;y,t)$ for all $y<I_{(x_0,s_0)}^{\tht-}(t)$.
\item\label{I-prop.h} $W^{\tht-}(I_{(x_0,s_0)}^{\tht-}(t),t;y,t)=W^{\tht+}(I_{(x_0,s_0)}^{\tht-}(t),t;y,t)$ for all $y\in[I_{(x_0,s_0)}^{\tht-}(t),I_{(x_0,s_0)}^{\tht+}(t)]$.
\end{enumerate}
\end{lem}

\begin{proof}
By the cocycle property \eqref{cocycle}, 
\begin{align*}
&W^{\tht-}(y,t;x_0,s_0)-W^{\tht+}(y,t;x_0,s_0)\\
&\qquad=\bigl(W^{\tht-}(y,t;0,t)-W^{\tht+}(y,t;0,t)\bigr)
+\bigl(W^{\tht-}(0,t;x_0,s_0)-W^{\tht+}(0,t;x_0,s_0)\bigr).
\end{align*}
This and \eqref{Wunbounded} give  that for $\sigg\in\{-,+\}$,
\[\lim_{y\to\sigg\infty}\bigl(W^{\tht-}(y,t;x_0,s_0)-W^{\tht+}(y,t;x_0,s_0)\bigr)=-\sigg\infty.\]
This and the continuity of $W^{\tht-}$ and $W^{\tht+}$ imply that $I_{(x_0,s_0)}^{\tht-}(t)>-\infty$ and $I_{(x_0,s_0)}^{\tht+}(t)<\infty$.\smallskip

By the cocycle property \eqref{cocycle} and the monotonicity \eqref{W:mono}, if $z>y$, then
\begin{align*}
&W^{\tht-}(z,t;x_0,s_0)-W^{\tht+}(z,t;x_0,s_0)\\
&\qquad=\bigl(W^{\tht-}(y,t;x_0,s_0)-W^{\tht+}(y,t;x_0,s_0)\bigr)
-\bigl(W^{\tht-}(y,t;z,t)-W^{\tht+}(y,t;z,t)\bigr)\\
&\qquad\le
W^{\tht-}(y,t;x_0,s_0)-W^{\tht+}(y,t;x_0,s_0).
\end{align*}
Hence, the function $y \mapsto W^{\tht-}(y,t;x_0,s_0)-W^{\tht+}(y,t;x_0,s_0)$ is nonincreasing. This implies that $I_{(x_0,s_0)}^{\tht-}(t)\le I_{(x_0,s_0)}^{\tht+}(t)$, establishing part \eqref{I-prop.a}. The  nonincreasing behavior also yields parts \eqref{I-prop.b} and \eqref{I-prop.c}. Moreover, it follows that $W^{\tht-}(y,t;x_0,s_0)=W^{\tht+}(y,t;x_0,s_0)$ for all $y$ strictly between $I_{(x_0,s_0)}^{\tht-}(t)$ and $I_{(x_0,s_0)}^{\tht+}(t)$. Then, by the continuity \eqref{W:cont} of $W^{\tht-}$ and $W^{\tht+}$, we get that for both $\sigg\in\{-,+\}$,
\begin{align}\label{aux912}
W^{\tht-}(I_{(x_0,s_0)}^{\tht\sig}(t),t;x_0,s_0)=W^{\tht+}(I_{(x_0,s_0)}^{\tht\sig}(t),t;x_0,s_0).
\end{align}
We have thus shown part \eqref{I-prop.d} to hold.\smallskip

By the monotonicity \eqref{W:mono}, we have $W^{\tht-}(y,s_0;x_0,s_0)\ge W^{\tht+}(y,s_0;x_0,s_0)$, for all $y\le x_0$. This and part \eqref{I-prop.b} imply that $y\le I_{(x_0,s_0)}^{\tht+}(s_0)$ for all such $y$. Thus, $x_0\le I_{(x_0,s_0)}^{\tht+}(s_0)$. A similar argument shows that $x_0\ge I_{(x_0,s_0)}^{\tht-}(s_0)$ and part \eqref{I-prop.e} is proved.\smallskip

By \eqref{aux912} and the cocycle property \eqref{cocycle}, for all $y\in\R$,
\begin{align*}
&W^{\tht-}(I_{(x_0,s_0)}^{\tht\sig}(t),t;y,t)-W^{\tht+}(I_{(x_0,s_0)}^{\tht\sig}(t),t;y,t)\\
&\quad=\bigl(W^{\tht-}(I_{(x_0,s_0)}^{\tht\sig}(t),t;x_0,s_0)-W^{\tht+}(I_{(x_0,s_0)}^{\tht\sig}(t),t;x_0,s_0)\bigr)
-\bigl(W^{\tht-}(y,t;x_0,s_0)-W^{\tht+}(y,t;x_0,s_0)\bigr)\\
&\quad=-\bigl(W^{\tht-}(y,t;x_0,s_0)-W^{\tht+}(y,t;x_0,s_0)\bigr).
\end{align*}
Thus, \eqref{I-prop.f} follows from \eqref{I-prop.b}, \eqref{I-prop.g} from \eqref{I-prop.c}, and \eqref{I-prop.h} from \eqref{I-prop.d}. 
\end{proof}

\begin{lem}\label{hmixed}
Let $\w\in\Omega_0$, $\tht\in\baddir$, and $x_0,s_0\in\R$. Then for any $\sigg\in\{-,+\}$ and any $t>s$ and $x$ in $\R$,
\[h^\tht_{(x_0,s_0)}(x,s)=\sup_{y\ge I_{(x_0,s_0)}^{\tht\sig}(t)}\!\!\{\mathcal L(x,s;y,t)+W^{\tht+}(y,t;x_0,s_0)\}\,\vee\!\!\!\sup_{y\le I_{(x_0,s_0)}^{\tht\sig}(t)}\!\!\{\mathcal L(x,s;y,t)+W^{\tht-}(y,t;x_0,s_0)\}.\]
\end{lem}

\begin{proof}
By definition \eqref{h} and parts (\ref{I-prop.b}-\ref{I-prop.d}) of Lemma \ref{I-prop}, we have $h^\tht_{(x_0,s_0)}(y,t)=W^{\tht+}(y,t;x_0,s_0)$ when $y\ge I_{(x_0,s_0)}^{\tht\sig}(t)$ and $h^\tht_{(x_0,s_0)}(y,t)=W^{\tht-}(y,t;x_0,s_0)$ when $y\le I_{(x_0,s_0)}^{\tht\sig}(t)$. The claim follows from this and \eqref{h-eternal}.
\end{proof}

\begin{lem}\label{lm:x0onI}
Let $\w\in\Omega_0$, $\tht\in\baddir$, and $(x_0,s_0)\in\IG\tht$. Then
\begin{align}\label{x0onI}
\begin{split}
&(x_0,s_0)\text{ is not right-isolated }\Longrightarrow\ I_{(x_0,s_0)}^{\tht+}(s_0)=x_0,\\
&(x_0,s_0)\text{ is not left-isolated }\Longrightarrow\ I_{(x_0,s_0)}^{\tht-}(s_0)=x_0.
\end{split}
\end{align}
\end{lem}

\begin{proof}
We prove the first claim, the second being similar. Therefore, assume $(x_0,s_0)$ is not right-isolated. Then Definition \ref{def:buseinstab}, the cocycle property \eqref{cocycle}, continuity \eqref{W:cont}, monotonicity \eqref{W:mono}, and Lemma \ref{ptinc}\eqref{ptinc.b} imply that 
$W^{\tht-}(x_0,s_0;y,s_0) > W^{\tht+}(x_0,s_0;y,s_0)$ for all $y>x_0$. This and the definition of $I_{(x_0,s_0)}^{\tht+}$ in \eqref{I} imply that 
$I_{(x_0,s_0)}^{\tht+}(s_0)\le x_0$. Lemma \ref{I-prop}\eqref{I-prop.e} gives the reverse inequaliy.
\end{proof}

The following comes similarly to the above lemma, from Definition \ref{def:buseinstab} of an instability point, properties \eqref{W:cont}-\eqref{W:mono} of Busemann functions, Lemma \ref{ptinc}\eqref{ptinc.b}, and parts \eqref{I-prop.b}-\eqref{I-prop.d} of Lemma \ref{I-prop}.

\begin{lem}\label{I-inst}
Let $\w\in\Omega_0$, $\tht\in\baddir$, and $x_0,s_0\in\R$. The following hold for all $t\in\R$.
\begin{enumerate} [label={\rm(\alph*)}, ref={\rm\alph*}] \itemsep=1pt
\item\label{I-inst.a} $(I_{(x_0,s_0)}^{\tht\sig}(t),t)\in\IG\tht$, for both $\sigg\in\{-,+\}$.
\item\label{I-inst.b} If $I_{(x_0,s_0)}^{\tht-}(t)<I_{(x_0,s_0)}^{\tht+}(t)$, then $(I_{(x_0,s_0)}^{\tht-}(t),t)$ is right-isolated and $(I_{(x_0,s_0)}^{\tht+}(t),t)$ is left-isolated.  
\item\label{I-inst.c} If $I_{(x_0,s_0)}^{\tht-}(t)=I_{(x_0,s_0)}^{\tht+}(t)$, then $(I_{(x_0,s_0)}^{\tht-}(t),t)=(I_{(x_0,s_0)}^{\tht+}(t),t)$ is neither left- nor right-isolated.
\end{enumerate}
\end{lem}

The next lemma says that one can replace the base point $(x_0,s_0)$ by any point on the interfaces $I_{(x_0,s_0)}^{\tht\sig}$ and get the same interfaces.

\begin{lem}\label{I-restart}
Let $\w\in\Omega_0$, $\tht\in\baddir$, and $x_0,s_0\in\R$. Then for each $\sigg\in\{-,+\}$ and $r\in\R$,
$I_{(I_{(x_0,s_0)}^{\tht\sig}(r),r)}^{\tht-}=I_{(x_0,s_0)}^{\tht-}$ and $I_{(I_{(x_0,s_0)}^{\tht\sig}(r),r)}^{\tht+}=I_{(x_0,s_0)}^{\tht+}$.
\end{lem}

\begin{proof}
    By Lemma \ref{I-prop}\eqref{I-prop.d}, $W^{\tht-}(I_{(x_0,s_0)}^{\tht\sig}(r),r;x_0,s_0)=W^{\tht+}(I_{(x_0,s_0)}^{\tht\sig}(r),r;x_0,s_0)$. This and the cocycle property \eqref{cocycle} imply that for any $y,t\in\R$,
    \[W^{\tht-}(y,t;x_0,s_0)<W^{\tht+}(y,t;x_0,s_0)\ \Longleftrightarrow\ W^{\tht-}(y,t;I_{(x_0,s_0)}^{\tht\sig}(r),r)<W^{\tht+}(y,t;I_{(x_0,s_0)}^{\tht\sig}(r),r)\]
    and
    \[W^{\tht-}(y,t;x_0,s_0)>W^{\tht+}(y,t;x_0,s_0)\ \Longleftrightarrow\ W^{\tht-}(y,t;I_{(x_0,s_0)}^{\tht\sig}(r),r)>W^{\tht+}(y,t;I_{(x_0,s_0)}^{\tht\sig}(r),r).\]
    The claim follows from this and the definition \eqref{I} of the interfaces.
\end{proof}

\begin{lem}\label{lm:1356}
Let $\w\in\Omega_0$, $\tht\in\baddir$, and $x_0,s_0\in\R$. Then for all  $r<t$ and $x$ in $\R$, setting $a=I_{(x_0,s_0)}^{\tht-}(t)$, we have
\begin{align}\label{claim}
\begin{split}
&W^{\tht-}(a,t;x,r)\ge W^{\tht+}(a,t;x,r)\\
&\qquad\qquad\Longleftrightarrow\ 
\sup_{y\ge a}\{\mathcal L(x,r;y,t)+W^{\tht+}(y,t;a,t)\}\ge\sup_{y\le a}\{\mathcal L(x,r;y,t)+W^{\tht-}(y,t;a,t)\}.
\end{split}
\end{align}
Similarly, if we set $a=I_{(x_0,s_0)}^{\tht+}(t)$, then
\begin{align*}
&W^{\tht-}(a,t;x,r)\le W^{\tht+}(a,t;x,r)\\
&\qquad\qquad\Longleftrightarrow\ 
\sup_{y\ge a}\{\mathcal L(x,r;y,t)+W^{\tht+}(y,t;a,t)\}\le\sup_{y\le a}\{\mathcal L(x,r;y,t)+W^{\tht-}(y,t;a,t)\}.
\end{align*}
\end{lem}

\begin{proof}
We prove \eqref{claim}, the other claim being symmetric.
First we prove the $\impliedby$ direction.
Suppose the inequality on the right-hand side of \eqref{claim} is true. Then the monotonicity \eqref{W:mono} gives us
\[\sup_{y\ge a}\{\mathcal L(x,r;y,t)+W^{\tht-}(y,t;a,t)\}
 \le\sup_{y\ge a}\{\mathcal L(x,r;y,t)+W^{\tht+}(y,t;a,t)\}\]
and, together with the right-hand side of \eqref{claim}, we have
\begin{align*}
\sup_{y\in\R}\{\mathcal L(x,r;y,t)+W^{\tht-}(y,t;a,t)\}
&\le\sup_{y\ge a}\{\mathcal L(x,r;y,t)+W^{\tht+}(y,t;a,t)\}\\
&\le\sup_{y\in\R}\{\mathcal L(x,r;y,t)+W^{\tht+}(y,t;a,t)\}.
\end{align*}
Using the update rule \eqref{eternal}, we get that the left-hand side equals $W^{\tht-}(x,r;a,t)$ and the right-hand side equals $W^{\tht+}(x,r;a,t)$. This shows that the inequality on the left-hand side of \eqref{claim} is true. 
\smallskip

Now we show the $\implies$ direction in \eqref{claim}. Thus, assume 
\begin{align}\label{h=W}
W^{\tht-}(a,t;x,r)\ge W^{\tht+}(a,t;x,r).
\end{align}
We claim that under this assumption, we have 
\begin{align}\label{claim2'}
    \sup_{y\leq a}\{\mathcal{L}(x,r;y,t)+W^{\tht+}(y,t;a,t)\} \le \sup_{y\geq a}\{\mathcal{L}(x,r;y,t)+W^{\tht+}(y,t;a,t)\}.
\end{align}

We first complete the proof of the lemma, assuming the validity of \eqref{claim2'}, and then proceed to establish \eqref{claim2'}.

Lemma \ref{I-prop}\eqref{I-prop.d} gives \begin{align}\label{aux1036}
 W^{\tht-}(a,t;x_0,s_0)=W^{\tht+}(a,t;x_0,s_0).
 \end{align}
This, \eqref{h=W}, and the cocycle property \eqref{cocycle}, give
$W^{\tht-}(x,r;x_0,s_0)\le W^{\tht+}(x,r;x_0,s_0)$. Then, by the definition \eqref{h} and Lemma \ref{hmixed},
\begin{align*}
W^{\tht+}(x,r;x_0,s_0)
&=h^\tht_{(x_0,s_0)}(x,r)\\
&=\sup_{y\le a}\{\mathcal L(x,r;y,t)+W^{\tht-}(y,t;x_0,s_0)\} \vee \sup_{y\geq a} \{\mathcal L(x,r;y,t)+W^{\tht+}(y,t;x_0,s_0)\}.
\end{align*}
This, another use of \eqref{aux1036} and the cocycle property \eqref{cocycle}, followed by \eqref{claim2'}, and \eqref{update}, give 
\begin{align*}
    W^{\tht+}(x,r;a,t) 
    &=\sup_{y\le a}\{\mathcal L(x,r;y,t)+W^{\tht-}(y,t;a,t)\} \vee \sup_{y\geq a} \{\mathcal L(x,r;y,t)+W^{\tht+}(y,t;a,t)\}\\
    &=\sup_{y\le a}\{\mathcal L(x,r;y,t)+W^{\tht-}(y,t;a,t)\} \vee \sup_{y\in\R} \{\mathcal L(x,r;y,t)+W^{\tht+}(y,t;a,t)\}\\
    &=\sup_{y\le a}\{\mathcal L(x,r;y,t)+W^{\tht-}(y,t;a,t)\}\vee W^{\tht+}(x,r;a,t).
\end{align*}
This proves that
\begin{align*}
    \sup_{y\le a}\{\mathcal L(x,r;y,t)+W^{\tht-}(y,t;a,t)\}\le W^{\tht+}(x,r;a,t)=\sup_{y\ge a}\{\mathcal L(x,r;y,t)+W^{\tht+}(y,t;a,t)\}.
\end{align*}
Now both directions of \eqref{claim} have been proved. 

We turn back to proving \eqref{claim2'}.
We proceed by contradiction. 
Suppose that
\begin{align}\label{supposition}
    \sup_{y\leq a}\{\mathcal{L}(x,r;y,t)+W^{\tht+}(y,t;a,t)\} > \sup_{y\geq a}\{\mathcal{L}(x,r;y,t)+W^{\tht+}(y,t;a,t)\}.
\end{align}
Then the supremum on the left cannot be attained at $y=a$.
Consequently, since the functions are all continuous, there exists $a_1<a$ such that 
\begin{align}\label{clover}
    \sup_{y\leq a}\{\mathcal{L}(x,r;y,t)+W^{\tht+}(y,t;a,t)\} = \sup_{y\leq a_1}\{\mathcal{L}(x,r;y,t)+W^{\tht+}(y,t;a,t)\}.
\end{align}
Using the cocycle property \eqref{cocycle} and the monotonicity \eqref{W:mono} gives, for all $y\le a_1$,
\begin{align*}
   &W^{\tht+}(y,t;a,t)-W^{\tht-}(y,t;a,t)\\
   &\qquad=W^{\tht+}(a_1,t;a,t)-W^{\tht-}(a_1,t;a,t)
   +W^{\tht+}(y,t;a_1,t)-W^{\tht-}(y,t;a_1,t)\\
   &\qquad\le W^{\tht+}(a_1,t;a,t)-W^{\tht-}(a_1,t;a,t).
\end{align*}
Rearranging, adding $\mathcal L(x,r;y,t)$, then taking a supremum over $y\ge a_1$ on both sides, we get
\begin{align*}
    &\sup_{y\leq a_1} \{\mathcal{L}(x,r;y,t)+W^{\tht+}(y,t;a,t)\}\\ 
    &\qquad\le\sup_{y\leq a_1} \{\mathcal{L}(x,r;y,t)+W^{\tht-}(y,t;a,t)\}+
    W^{\tht+}(a_1,t;a,t)-W^{\tht-}(a_1,t;a,t)\\
    &\qquad<\sup_{y\leq a_1} \{\mathcal{L}(x,r;y,t)+W^{\tht-}(y,t;a,t)\},
\end{align*}
where the strict inequality came from Lemma \ref{I-prop}\eqref{I-prop.g}, as $a_1<a=I_{(x_0,s_0)}^{\tht-}(t)$.
Consequently,
\begin{align*}
    &\sup_{y\leq a} \{\mathcal{L}(x,r;y,t)+W^{\tht-}(y,t;a,t)\}
    \geq 
    \sup_{y\leq a_1} \{\mathcal{L}(x,r;y,t)+W^{\tht-}(y,t;a,t)\} \\
    &\qquad>\sup_{y\leq a_1} \{\mathcal{L}(x,r;y,t)+W^{\tht+}(y,t;a,t)\}
    =\sup_{y\leq a} \{\mathcal{L}(x,r;y,t)+W^{\tht+}(y,t;a,t)\},
\end{align*}
where the last equality comes from \eqref{clover}.
 
 Recalling \eqref{supposition}, and using \eqref{aux1036} to move the point $(a,t)$ to $(x_0,s_0)$ in all the $W^{\tht\sig}$ terms, 
 we now have
\begin{align*}
    \sup_{y\leq a} \{\mathcal{L}(x,r;y,t)+W^{\tht-}(y,t;x_0,s_0)\} 
    &> \sup_{y\leq a} \{\mathcal{L}(x,r;y,t)+W^{\tht+}(y,t;x_0,s_0)\}\\ 
    &> \sup_{y\geq a}\{\mathcal{L}(x,r;y,t)+W^{\tht+}(y,t;x_0,s_0)\}.
\end{align*}
This gives
    \begin{align*}
    W^{\tht-}(x,r;x_0,s_0)
    &=\sup_{y\in\R} \{\mathcal{L}(x,r;y,t)+W^{\tht-}(y,t;x_0,s_0)\}\\ 
    &\ge\sup_{y\leq a} \{\mathcal{L}(x,r;y,t)+W^{\tht-}(y,t;x_0,s_0)\}\\ 
    &>\sup_{y\in\R} \{\mathcal{L}(x,r;y,t)+W^{\tht+}(y,t;x_0,s_0)\}
    =W^{\tht+}(x,r;x_0,s_0),
    \end{align*}
    where we used the update rule \eqref{update} for the equalities. Applying \eqref{aux1036} and the cocycle property \eqref{cocycle} one more time, we get  $W^{\tht-}(x,r;a,t)>W^{\tht+}(x,r;a,t)$, which contradicts assumption \eqref{h=W}. Thus, \eqref{supposition} is incorrect and we have shown that \eqref{h=W} implies \eqref{claim2'}.
\end{proof}

\begin{lem}\label{Icont}
Let $\w\in\Omega_0$, $\tht\in\baddir$, and $x_0,s_0\in\R$. Then both $I_{(x_0,s_0)}^{\tht-}$ and $I_{(x_0,s_0)}^{\tht+}$ are continuous on $\R$. 
\end{lem}

\begin{proof}
We prove the continuity of $I_{(x_0,s_0)}^{\tht-}$, the other case being similar. Fix $t>s_0$ and abbreviate $a=I_{(x_0,s_0)}^{\tht-}(t)$. For $r<t$ and $x$ in $\R$, define the function  
\[d(x,r)=\sup_{y\ge a}\{\mathcal L(x,r;y,t)+h^\tht_{(x_0,s_0)}(y,t)\}-\sup_{y\le a}\{\mathcal L(x,r;y,t)+h^\tht_{(x_0,s_0)}(y,t)\}.\]
From the definition \eqref{h} and parts \eqref{I-prop.b}-\eqref{I-prop.d} of Lemma \ref{I-prop}, we get that $h^\tht_{(x_0,s_0)}(y,t)=W^{\tht+}(y,t;x_0,s_0)$ when $y\ge a$ and $h^\tht_{(x_0,s_0)}(y,t)=W^{\tht-}(y,t;x_0,s_0)$ when $y\le a$. Thus, 
\[d(x,r)=\sup_{y\ge a}\{\mathcal L(x,r;y,t)+W^{\tht+}(y,t;x_0,s_0)\}-\sup_{y\le a}\{\mathcal L(x,r;y,t)+W^{\tht-}(y,t;x_0,s_0)\}.\]

Lemma \ref{I-prop}\eqref{I-prop.d}, together with the cocyle property \eqref{cocycle} imply
\begin{align*}
W^{\tht-}(I_{(x_0,s_0)}^{\tht-}(r),r;a,t)
&=W^{\tht-}(I_{(x_0,s_0)}^{\tht-}(r),r;x_0,s_0)
-W^{\tht-}(I_{(x_0,s_0)}^{\tht-}(t),t;x_0,s_0)\\
&=W^{\tht+}(I_{(x_0,s_0)}^{\tht-}(r),r;x_0,s_0)
-W^{\tht+}(I_{(x_0,s_0)}^{\tht-}(t),t;x_0,s_0)\\
&=W^{\tht+}(I_{(x_0,s_0)}^{\tht-}(r),r;a,t).
\end{align*}
This, 
together with the cocycle property \eqref{cocycle}, \eqref{claim}, and  \eqref{aux1036}, give 
\begin{align*}
&W^{\tht+}(I_{(x_0,s_0)}^{\tht-}(r),r;x,r)\le W^{\tht-}(I_{(x_0,s_0)}^{\tht-}(r),r;x,r)\\
&\qquad\qquad\Longleftrightarrow 
W^{\tht+}(a,t;x,r)\le W^{\tht-}(a,t;x,r)\\
&\qquad\qquad\Longleftrightarrow 
\sup_{y\ge a}\{\mathcal L(x,r;y,t)+W^{\tht+}(y,t;a,t)\}\ge\sup_{y\le a}\{\mathcal L(x,r;y,t)+W^{\tht-}(y,t;a,t)\}\\
&\qquad\qquad\Longleftrightarrow 
\sup_{y\ge a}\{\mathcal L(x,r;y,t)+W^{\tht+}(y,t;x_0,s_0)\}\ge\sup_{y\le a}\{\mathcal L(x,r;y,t)+W^{\tht-}(y,t;x_0,s_0)\}.
\end{align*}
The right-hand side is equivalent to $d(x,r)\ge0$.
By parts \eqref{I-prop.f}-\eqref{I-prop.h} of Lemma \ref{I-prop}, the left-hand side is equivalent to $x\ge I_{(x_0,s_0)}^{\tht-}(r)$.
Thus, we have
\begin{align*}
    x\ge I_{(x_0,s_0)}^{\tht-}(r)\ \Longleftrightarrow\ d(x,r)\ge0.
\end{align*}
This means
\[I_{(x_0,s_0)}^{\tht-}(r)=\inf\{x\in\R:d(x,r)\ge0\}.\]
Thus, $\{I_{(x_0,s_0)}^{\tht-}(r):r\le t\}$ corresponds to the leftmost competition interface as per the definition \eqref{CIs}, with the boundary condition $y\mapsto h^\tht_{(x_0,s_0)}(y,t)$ at time $t$, and the reference point being $a=I_{(x_0,s_0)}^{\tht-}(t)$. By \eqref{Icont}, this interface is a continuous function from $(-\infty,t]$ to $\R$. 
 Since $t$ was arbitrary, $I_{(x_0,s_0)}^{\tht-}$ is continuous on $\R$.
 \end{proof}

We are now ready to prove Proposition \ref{prop:IGgoup}.

\begin{proof}[Proof of Proposition \ref{prop:IGgoup}] 
Let $\Ipath_{(x,s)}^{\tht\sig}(t)=(I_{(x,s)}^{\tht\sig}(t),t)$.
Part \eqref{IGgoup.a} comes from Lemma \ref{Icont}. Part \eqref{IGgoup.b} follows from Lemma \ref{I-prop}\eqref{I-prop.a}. Parts \eqref{IGgoup.c} and \eqref{IGgoup.d} are in Lemma \ref{I-inst}. 
By Lemma \ref{no-isolated}, any $(x,s)\in\IG\tht$ is either not right-isolated or not left-isolated (or both). Together with Lemma \ref{lm:x0onI}, this yields the first claim as well as the $\Leftarrow$ directions of the two equivalences stated in part \eqref{IGgoup.e}. The converse directions also follow: for instance, if $(x,s)$ is right-isolated, then Lemma \ref{no-isolated} implies that it is not left-isolated, and we have just shown this implies $\Ipath_{(x,s)}^{\tht-}(s)=(x,s)$.
Lemma \ref{I-restart} gives part \eqref{IGgoup.f}.
\smallskip

Part \eqref{IGgoup.g}. 
Suppose there exists an $r\in\R$ such that $\Ipath_{(x,s)}^{\tht\sig}(r)=\Ipath_{(x',s')}^{\tht\sig}(r)$.
Then part \eqref{IGgoup.f} says that for any $t\in\R$, 
\[\Ipath_{(x,s)}^{\tht\sig}(t)=\Ipath_{\Ipath_{(x,s)}^{\tht\sig}(r)}^{\tht\sig}(t)=\Ipath_{\Ipath_{(x',s')}^{\tht\sig}(r)}^{\tht\sig}(t)=\Ipath_{(x',s')}^{\tht\sig}(t).\]
Part \eqref{IGgoup.g} is proved.\smallskip

Part \eqref{IGgoup.h}. Take $t>s'$. 
By Lemma \ref{I-prop}\eqref{I-prop.d}, 
\begin{align*}
&W^{\tht-}\bigl(I_{(x,s)}^{\tht\sig}(t),t;x,s\bigr)=W^{\tht+}\bigl(I_{(x,s)}^{\tht\sig}(t),t;x,s\bigr)
\quad\text{and}\\
&W^{\tht-}\bigl(I_{(x,s)}^{\tht\sig}(s'),s';x,s\bigr)=W^{\tht+}\bigl(I_{(x,s)}^{\tht\sig}(s'),s';x,s\bigr).
\end{align*}
This and the cocycle property \eqref{cocycle} give
\begin{align}\label{1773}
W^{\tht-}\bigl(I_{(x,s)}^{\tht\sig}(t),t;I_{(x,s)}^{\tht\sig}(s'),s'\bigr)=W^{\tht+}\bigl(I_{(x,s)}^{\tht\sig}(t),t;I_{(x,s)}^{\tht\sig}(s'),s'\bigr).
\end{align}
Then by Lemma \ref{lm:1356}, we get
\begin{align}\label{aux1535}
\begin{split}
&\sup_{y\ge I_{(x,s)}^{\tht\sig}(t)}\bigl\{\mathcal L(I_{(x,s)}^{\tht\sig}(s'),s';y,t)+W^{\tht+}(y,t;I_{(x,s)}^{\tht\sig}(t),t)\bigr\}\\
&\qquad\qquad=\sup_{y\le I_{(x,s)}^{\tht\sig}(t)}\bigl\{\mathcal L(I_{(x,s)}^{\tht\sig}(s'),s';y,t)+W^{\tht-}(y,t;I_{(x,s)}^{\tht\sig}(t),t)\bigr\}.
\end{split}
\end{align}

By the monotonicity \eqref{W:mono}, we have
\begin{align*}
&W^{\tht+}\bigl(y,t;I_{(x,s)}^{\tht\sig}(t),t\bigr)
\ge W^{\tht-}\bigl(y,t;I_{(x,s)}^{\tht\sig}(t),t\bigr)
\quad\forall y\ge I_{(x,s)}^{\tht\sig}(t)
\quad\text{and}\\
&W^{\tht-}\bigl(y,t;I_{(x,s)}^{\tht\sig}(t),t\bigr)
\ge W^{\tht+}\bigl(y,t;I_{(x,s)}^{\tht\sig}(t),t\bigr)
\quad\forall y\le I_{(x,s)}^{\tht\sig}(t).
\end{align*}
This and \eqref{aux1535} imply that 
\begin{align*}
&\sup_{y\ge I_{(x,s)}^{\tht\sig}(t)}\bigl\{\mathcal L(I_{(x,s)}^{\tht\sig}(s'),s';y,t)+W^{\tht+}(y,t;I_{(x,s)}^{\tht\sig}(t),t)\bigr\}\\
&\qquad\qquad\ge\sup_{y\le I_{(x,s)}^{\tht\sig}(t)}\bigl\{\mathcal L(I_{(x,s)}^{\tht\sig}(s'),s';y,t)+W^{\tht+}(y,t;I_{(x,s)}^{\tht\sig}(t),t)\bigr\}
\end{align*}
and
\begin{align*}
&\sup_{y\le I_{(x,s)}^{\tht\sig}(t)}\bigl\{\mathcal L(I_{(x,s)}^{\tht\sig}(s'),s';y,t)+W^{\tht-}(y,t;I_{(x,s)}^{\tht\sig}(t),t)\bigr\}\\
&\qquad\qquad\ge\sup_{y\ge I_{(x,s)}^{\tht\sig}(t)}\bigl\{\mathcal L(I_{(x,s)}^{\tht\sig}(s'),s';y,t)+W^{\tht-}(y,t;I_{(x,s)}^{\tht\sig}(t),t)\bigr\}.
\end{align*}
Consequently, the rightmost maximizer of $\mathcal L(I_{(x,s)}^{\tht\sig}(s'),s';y,t)+W^{\tht+}(y,t;I_{(x,s)}^{\tht\sig}(t),t)$ is $\ge I_{(x,s)}^{\tht\sig}(t)$ and the leftmost maximizer of $\mathcal L(I_{(x,s)}^{\tht\sig}(s'),s';y,t)+W^{\tht-}(y,t;I_{(x,s)}^{\tht\sig}(t),t)$ is $\le I_{(x,s)}^{\tht\sig}(t)$. This and the extremality properties \eqref{p2lleftmost} and \eqref{p2lrightmost} imply that 
\begin{align}\label{aux1510}
\geo\from{\Ipath_{(x,s)}^{\tht\sig}(s')}\dir{L}{\tht}{-}(t)=\geo\from{(I_{(x,s)}^{\tht\sig}(s'),s')}\dir{L}{\tht}{-}(t)\le \Ipath_{(x,s)}^{\tht\sig}(t)=\bigl(I_{(x,s)}^{\tht\sig}(t),t\bigr)\le \geo\from{(I_{(x,s)}^{\tht\sig}(s'),s')}\dir{R}{\tht}{+}(t)=\geo\from{\Ipath_{(x,s)}^{\tht\sig}(s')}\dir{R}{\tht}{+}(t).
\end{align}
This proves \eqref{Isplit-strict} with weak inequalities.
It remains to prove the strict inequalities.
We prove the first, the second being similar.

For a contradiction, suppose $\sigg=-$ and the first inequality in \eqref{aux1510} is an equality.
As $y$ strictly increases to $I_{(x,s)}^{\tht-}(s')$, $(y,s')$ converges to $\Ipath_{(x,s)}^{\tht-}(s')$ and, by Corollary \ref{cor:geo:lim}, $\geo\from{(y,s')}\dir{L}{\tht}{-}$ converges in overlap topology to $\geo\from{\Ipath_{(x,s)}^{\tht-}(s')}\dir{L}{\tht}{-}$. Thus, for $y<I_{(x,s)}^{\tht-}(s')$ close enough to $I_{(x,s)}^{\tht-}(s')$, $\geo\from{(y,s')}\dir{L}{\tht}{-}$ coalesces with $\geo\from{\Ipath_{(x,s)}^{\tht-}(s')}\dir{L}{\tht}{-}$ before time $t$ and hence goes through $\geo\from{\Ipath_{(x,s)}^{\tht-}(s')}\dir{L}{\tht}{-}(t)$, which we have assumed to equal $(I_{(x,s)}^{\tht-}(t),t)$. Then, by \eqref{rec}, we have
\begin{align}\label{1521}
\mathcal L\bigl(y,s';I_{(x,s)}^{\tht-}(t),t\bigr)=W^{\tht-}\bigl(y,s';I_{(x,s)}^{\tht-}(t),t\bigr).
\end{align}
But, since $y<I_{(x,s)}^{\tht-}(s')$, Lemma \ref{I-prop}\eqref{I-prop.g} gives 
$W^{\tht-}\bigl(y,s';I_{(x,s)}^{\tht-}(s'),s'\bigr)>W^{\tht+}\bigl(y,s';I_{(x,s)}^{\tht-}(s'),s'\bigr)$ which, together with \eqref{1773} and the cocycle property \eqref{cocycle}, gives
\[W^{\tht-}\bigl(y,s';I_{(x,s)}^{\tht-}(t),t\bigr)>W^{\tht+}\bigl(y,s';I_{(x,s)}^{\tht-}(t),t\bigr).\]
This strict inequality contradicts \eqref{1521} and \eqref{L<W}. 
Part \eqref{IGgoup.h} is proved.\smallskip

Part \eqref{IGgoup.i}. 
We prove the first inequality, the second being similar. For a contradiction, suppose  there exists an $r<s'$ such that $\Upsilon\from{\Ipath_{(x,s)}^{\tht\sig}(s')}\dir{L}{\tht}{+}(r)>\Ipath_{(x,s)}^{\tht\sig}(r)$. By part \eqref{IGgoup.h}, $\geo\from{\Ipath_{(x,s)}^{\tht\sig}(r)}\dir{R}{\tht}{+}$ remains weakly to the right of $\Ipath_{(x,s)}^{\tht\sig}$. On the other hand, the duality \eqref{no-intersection} and the continuity of the paths imply that this geodesic must stay strictly to the left of the shock interface $\Upsilon\from{\Ipath_{(x,s)}^{\tht\sig}(s')}\dir{L}{\tht}{+}$ over the time interval $(r,s')$. Thus,  we must have $\geo\from{\Ipath_{(x,s)}^{\tht\sig}(r)}\dir{R}{\tht}{+}(s')=\Ipath_{(x,s)}^{\tht\sig}(s')$, and now we have a contradiction with Lemma \ref{lm:geosplitsshocks}. Part \eqref{IGgoup.i} is proved.
\smallskip

Part \eqref{IGgoup.j}. The claim follows from parts \eqref{IGgoup.h} and \eqref{IGgoup.i}, the $\tht$-directedness \eqref{geo:directed} of the geodesics, and the $\tht$-directedness \eqref{Upsilondir} of the $\tht\sigg$ shock interfaces. (The $\tht$-directedness as $t\to-\infty$ can also be seen from putting together \eqref{Wgrowth} and \eqref{Idir}.)
\end{proof}

Next, we establish the pathwise connectedness of $\IG\tht$ and justify the terminology of referring to $\IG\tht$ both as a \emph{graph} and as a \emph{web}.

\begin{defn}\label{def:Nanc}
    For $\tht\in\baddir$, $s<t$, and $(x,s),(y,t)\in\IG\tht$, we say that $(y,t)$ is an \emph{ancestor} of $(x,s)$, equivalently that $(x,s)$ is a \emph{descendant} of $(y,t)$, if there is a continuous space-time path $\gamma:[s,t]\to\IG\tht$ such that $\gamma(s) = (x,s)$ and $\gamma(t) = (y,t)$.
\end{defn}

\begin{lem}\label{lem:ancifbtw}
    Let $\w\in\Omega_0$, $\tht\in\baddir$, $s<t$ in $\R$, and $(x,s),(y,t)\in\IG\tht$.
    If $\geo\from{(x,s)}\dir{L}{\tht}{-}(t)\le (y,t)\le\geo\from{(x,s)}\dir{R}{\tht}{+}(t)$, then $(y,t)$ is an ancestor of $(x,s)$ .
\end{lem}

\begin{proof}
    Proposition \ref{prop:IGgoup} provides a continuous space-time path $\gamma$ emanating from $(x,s)$ whose entire trajectory lies within $\IG\tht$.
    If $\gamma(t)=(y,t)$, then we are done. The cases $\gamma(t)<(y,t)$ and $\gamma(t)>(y,t)$ are treated similarly. We describe the first case. 
    By Lemma \ref{lem:IGgodown}, the whole shock interface $\Upsilon\from{(y,t)}\dir{L}{\tht}{+}$ is in $\IG\tht$.
    By duality \eqref{no-intersection} and continuity, if $(y,t) < \geo\from{(x,s)}\dir{R}{\tht}{+}(t)$,  and by Lemma \ref{lm:geosplitsshocks} in the case 
$(y,t) = \geo\from{(x,s)}\dir{R}{\tht}{+}(t)$, 
the shock interface $\Upsilon\from{(y,t)}\dir{L}{\tht}{+}$ remains strictly to the left of the geodesic $\geo\from{(x,s)}\dir{R}{\tht}{+}$ throughout the time interval $(s,t)$.
 Since the interface and the geodesic are both continuous, it must be that
    $\Upsilon\from{(y,t)}\dir{L}{\tht}{+}(s)\le(x,s)$. 
    By continuity, $\Upsilon\from{(y,t)}\dir{L}{\tht}{+}$ must intersect $\gamma$ at some time $r\in[s,t)$. Following $\gamma$ from $(x,s)$ to the intersection point and then following $\Upsilon\from{y}\dir{L}{\tht}{+}$ from that point to $(y,t)$ produces a continuous space-time path that remains on $\IG\tht$. Thus $(y,t)$ is an ancestor of $(x,s)$. 
\end{proof}

The next lemma is the analogous version of the previous one, with shock interfaces in place of geodesics.

\begin{lem}\label{lem:childifbtw}
    Let $\w\in\Omega_0$, $\tht\in\baddir$, $s<t$ in $\R$, $(x,s),(y,t)\in\IG\tht$, and $\sigg,\overline\sigg\in\{-,+\}$. Let $\tau$ {\rm(}respectively $\overline\tau${\rm)} be a $\tht\sigg$ {\rm(}respectively $\tht\overline\sigg${\rm)} shock interface out of $(y,t)$. 
    If $\tau(s)\le(x,s)\le\overline\tau(s)$, then $(x,s)$ is a descendant of $(y,t)$. 
\end{lem}

\begin{proof}
Lemma \ref{lem:IGgodown} implies that both $\tau$ and $\overline\tau$ are entirely contained in $\IG\tht$. By Lemma \ref{no-isolated}, the point $(x,s)$ is either not left-isolated or not right-isolated. Applying Proposition \ref{prop:IGgoup}\eqref{IGgoup.e}, choose $\wh\sigg\in\{+,-\}$ such that $\Ipath_{(x,s)}^{\tht\wh\sig}(s)=(x,s)$. By part \eqref{IGgoup.c} of the same proposition, the path $\Ipath_{(x,s)}^{\tht\wh\sig}$ is entirely contained in $\IG\tht$.

Since $\Ipath_{(x,s)}^{\tht\wh\sig}$ is bi-infinite, $(x,s)$ lies between $\tau$ and $\overline\tau$, and all paths are continuous, it follows that $\Ipath_{(x,s)}^{\tht\wh\sig}$ must intersect one of $\tau$ or $\overline\tau$ at some time $r\in[s,t]$. Following $\Ipath_{(x,s)}^{\tht\wh\sig}$ on $[s,r]$ and then continuing along $\tau$ or $\overline\tau$ on $[r,t]$, according to which one passes through $\Ipath_{(x,s)}^{\tht\wh\sig}(r)$, gives a continuous space-time path from $(x,s)$ to $(y,t)$ that remains entirely within $\IG\tht$, completing the proof.
\end{proof}

\begin{rmk}\label{rk:ancestry}
The converses of Lemmas \ref{lem:ancifbtw} and \ref{lem:childifbtw} remain open. If the converse of Lemma \ref{lem:ancifbtw} holds, then any continuous space-time path $\gamma:[s,\infty)\to\R^2$ contained in $\IG\tht$ is $\tht$-directed. This would give a positive answer to Problem \ref{prob2}. Likewise, if the converse of Lemma \ref{lem:childifbtw} holds, then any continuous space-time path $\tau:(-\infty,s]\to\R^2$ contained in $\IG\tht$ is $\tht$-directed.
Both converses would follow if one shows that, almost surely, for any $\tht\in\baddir$ and any continuous space-time path $\gamma:[s,t]\to\R^2$ with $\gamma([s,t])\subset\IG\tht$, there exists a sequence $r_n\searrow s$ such that $\geo\from{\gamma(r_n)}\dir{L}{\tht}{-}\preceq \gamma \preceq \geo\from{\gamma(r_n)}\dir{R}{\tht}{+}$.
This, in turn, would follow from a positive answer to  Problem \ref{prob3}.
\end{rmk}

\begin{lem}\label{lem:cmnNanc}
    Let $\w\in\Omega_0$, $\tht\in\baddir$, and $(x,s), (y,t) \in \IG\tht$. Then $(x,s)$ and $(y,t)$ have a common ancestor in $\IG\tht$.
\end{lem}

\begin{proof}
    Without loss of generality, we can assume $s\le t$. When $t=s$, we can also assume $y\ne x$. 
    
    If $\geo\from{(x,s)}\dir{L}{\tht}{-}(t)\le(y,t)\le\geo\from{(x,s)}\dir{R}{\tht}{+}(t)$, then we must have $t>s$ and then the claim follows from  Lemma \ref{lem:ancifbtw}.  The cases $(y,t)<\geo\from{(x,s)}\dir{L}{\tht}{-}(t)$ and $\geo\from{(x,s)}\dir{R}{\tht}{+}(t)<(y,t)$ are similar and we treat the second one. By \eqref{geo:coal2},  $\geo\from{(y,t)}\dir{L}{\tht}{-}$ must coalesce with $\geo\from{(x,s)}\dir{L}{\tht}{-}$ and thus, by continuity, it must cross $\geo\from{(x,s)}\dir{R}{\tht}{+}$. Call $(z,r)\in\R^2$ the last intersection point of $\geo\from{(x,s)}\dir{R}{\tht}{+}$ and $\geo\from{(y,t)}\dir{L}{\tht}{-}$. (Such a point exists by \eqref{sign} and the continuity of the paths.)
    By \eqref{geo:restart}, $\geo\from{(z,r)}\dir{L}{\tht}{-}(r')=\geo\from{(y,t)}\dir{L}{\tht}{-}(r')$ and $\geo\from{(z,r)}\dir{R}{\tht}{+}(r')=\geo\from{(x,s)}\dir{R}{\tht}{+}(r')$, for all $r'\ge r$. Since $(z,r)$ is the last intersection point, we have $\geo\from{(z,r)}\dir{L}{\tht}{-}\cap\geo\from{(z,r)}\dir{R}{\tht}{+}=\{(z,r)\}$ and, by Lemma \ref{def:geoinstabpt}, $(z,r)\in\IG\tht$.

    By construction, $\geo\from
    {(x,s)} \dir{L}{\tht}{-}(r)\le(z,r) = \geo\from
    {(x,s)}\dir{R}{\tht}{+}(r)$ and $\geo\from
    {(y,t)} \dir{L}{\tht}{-}(r)= (z,r) \le \geo\from
    {(y,t)}\dir{R}{\tht}{+}(r)$. Thus, by Lemma \ref{lem:ancifbtw}, $(z,r)$ is an ancestor of both $(x,s)$ and $(y,t)$. 
\end{proof}

\begin{lem}\label{lem:cmnSanc}
      Let $\w\in\Omega_0$, $\tht\in\baddir$, and $(x,s), (y,t) \in \IG\tht$. Then $(x,s)$ and $(y,t)$ have a common descendant in $\IG\tht$. 
\end{lem}

\begin{proof}
By \eqref{Itree}, $\Upsilon\from{(x,s)}\dir{R}{\tht}{-}$ and $\Upsilon\from{(y,t)}\dir{R}{\tht}{-}$ coalesce. Let $(z,r)$ be their coalescence point. 
By Lemma \ref{lem:IGgodown}, these paths remain on $\IG\tht$. Thus, $(z,r)\in\IG\tht$ and both $(x,s)$ and $(y,t)$ are ancestors of $(z,r)$.
\end{proof}

\section{Sketching the instability graph using shocks}\label{sec:sketch}

Recall that if $U\subset\R^2$ is open, a \emph{connected component} of $U$ is a maximal connected open subset of $U$. The collection of connected components forms a partition of $U$. By the separability of $\R^2$, any open set has at most countably many connected components. Since $\R^2$ is locally path-connected, any connected subset is also path-connected. See Appendix \ref{app:con} for more.

By Lemma \ref{lem:IGisclosed}, the set $\R^2\setminus\IG\tht$ of stability points is open. In this section, we describe its connected components and characterize their boundaries in terms of shock interfaces. Proposition \ref{prop:islandsdense} shows that the union of the boundaries of all these components is dense in the instability graph. Accordingly, we refer to this union as the \emph{skeleton} of the instability graph.

It is important to emphasize, however, that this skeleton represents only a small portion of the full graph: the remainder of $\IG\tht$ constitutes the bulk of the set. An instructive analogy is provided by the zero set of standard Brownian motion. In that setting, the analogue of the skeleton is the collection of left-isolated/right-isolated pairs of zeros (that is, pairs $s<t$ such that the Brownian motion vanishes at $s$ and $t$, but is nonzero for all $r\in(s,t)$). This collection is countable, whereas the entire zero set has Hausdorff dimension $1/2$ and therefore forms the dominant part of the set. Nevertheless, the closure of the left- and right-isolated zeros recovers the full zero set.
In fact, by \eqref{Hausdorff}, the intersection of $\IG\tht$ with any rational time level exhibits precisely this structure (and we expect the same to hold for arbitrary time levels). Despite occupying only a small portion of the graph, these components play a distinguished structural role: Lemma \ref{shocksintersect} shows that, for each $\sigg\in\{-,+\}$, the coalescence points of $\tht\sigg$ shock interfaces lie in the union of their closures.

\begin{defn}
For $\tht \in \baddir$, a connected component of $\R^2 \setminus \IG\tht$ is called a \emph{stability island}. For $(x,s) \in \R^2 \setminus \IG\tht$, we denote by $\island_{(x,s)}$ the island containing $(x,s)$.
\end{defn}

In this section, we analyze stability islands and their relationship to shock interfaces and the various types of shocks.  Lemma \ref{lm:islands} shows that all islands are bounded and that there are countably infinitely many of them. Lemmas \ref{lem:shockstoisland} and \ref{lem:islandtoshocks} identify the boundary of an island as arising from a specific configuration of shock interfaces (Definition \ref{def:misordered}). By definition, the left and right boundaries of an island consist of right- and left-isolated instability points, respectively. Lemma \ref{lm:isotoisland} explains how an island can be reconstructed from a single such point.
Lemma \ref{lem:islandtoshocks} identifies the top point of an island, while Lemma \ref{lem:shockstoisland} shows how the island can be reconstructed from this tip.
Lemmas \ref{lm:snowbird} and \ref{lem:SBisisland} identify the bottom point as a special type of double shock (Definition \ref{def:snowbird}) and show that the island may also be reconstructed from that point. The full description of the boundary is in Lemma \ref{lm:islandboundary}. Lemma \ref{lm:disjointislands} establishes that the closures of distinct islands are disjoint. 
The results in Subsection \ref{sec:islands+cifs} relate island boundaries to the interfaces constructed in Proposition \ref{prop:IGgoup}.

\begin{lem}\label{lm:islands}
    Let $\w\in\Omega_0$ and $\tht\in\baddir$. All stability islands are bounded and there are countably infinitely many of them.
\end{lem}

\begin{proof}
 Consider an island $\island_{(x,s)}$, $(x,s)\in\R^2\setminus\IG\tht$. By \eqref{Wunbounded} and Definition \ref{def:buseinstab}, there exist $x_1<x<x_2$ such that $(x_1,s)$ and $(x_2,s)$ are in $\IG\tht$. By Lemmas \ref{lem:cmnNanc} and \ref{lem:cmnSanc}, there exist $r<s<t$ and $z,y$ in $\R$ such that $(y,t)$ is a common ancestor of $(x_1,s)$ and $(x_2,s)$ and $(z,r)$ is a common descendant of the two points. This implies that $\island_{(x,s)}$ is surrounded by two continuous paths, that go from $(z,r)$ to $(x_i,s)$ to $(y,t)$, $i\in\{1,2\}$. Thus, $\island_{(x,s)}$ is bounded. Since $\R^2\setminus\IG\tht$ is dense (Lemma \ref{lem:IGisclosed}), there must be infinitely many stability islands. As noted above, the number of connected components is at most countable; hence, the collection of stability islands is countably infinite.
\end{proof}

Recall from Lemma \ref{lm:shockasymorder} that $\theta+$ shock interfaces eventually lie to the left of $\theta-$ shock interfaces. This explains the terminology in the following definition.

\begin{defn}\label{def:misordered}
Let $\tht\in\baddir$. For $(z,t)\in\R^2$, let $\tau^-$ and $\tau^+$ be, respectively, $\tht-$ and $\tht+$ shock interfaces emanating from $(z,t)$. We say that $\tau^-$ and $\tau^+$ are \emph{misordered} if there exists $\ep>0$ such that
$\tau^-(r)<\tau^+(r)$ for all $r\in(t-\ep,t)$.
\end{defn}

The next lemma gives a sufficient condition for determining an island, in terms of misordered shock interfaces.

\begin{lem}\label{lem:shockstoisland}
    Let $\w\in\Omega_0$ and $\tht\in\baddir$. Let $(z,t_2)\in\IG\tht$. Suppose that $\Upsilon_{(z,t_2)}\dir{L}{\tht}{-}$ and $\Upsilon_{(z,t_2)}\dir{R}{\tht}{+}$ are misordered.  Let
    \begin{align}\label{t1}
        t_1=\sup\bigl\{r<t_2:\Upsilon_{(z,t_2)}\dir{L}{\tht}{-}(r) = \Upsilon_{(z,t_2)}\dir{R}{\tht}{+}(r)\bigr\}.
    \end{align}
    Then $-\infty<t_1<t_2$, $\Upsilon_{(z,t_2)}\dir{L}{\tht}{-}(t_1) = \Upsilon_{(z,t_2)}\dir{R}{\tht}{+}(t_1)$, $\Upsilon_{(z,t_2)}\dir{L}{\tht}{-}(r)<\Upsilon_{(z,t_2)}\dir{R}{\tht}{+}(r)$,  
    for all $r\in(t_1,t_2)$, and 
    \begin{align}\label{eq:shockstoisland}
    U=\bigl\{ (y,r)\in\R^2\,:\, r\in(t_1,t_2),\, \Upsilon_{(z,t_2)}\dir{L}{\tht}{-}(r)<y<\Upsilon_{(z,t_2)}\dir{R}{\tht}{+}(r)\bigr\}
    \end{align}
    is a stability island $\island_{(x,s)}$ for any $(x,s)\in U$. 
\end{lem}

\begin{figure}[hpt]
    \includegraphics[width=3cm]{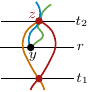}
    \caption{\small The geodesics $\geo\from{(y,r)}\dir{L}{\tht}{-}$ and $\geo\from{(y,r)}\dir{R}{\tht}{+}$ are trapped between the two misordered shock interfaces $\Upsilon_{(z,t_2)}\dir{L}{\tht}{-}$ and $\Upsilon_{(z,t_2)}\dir{R}{\tht}{+}$, and are thus forced to meet at $(z,t_2)$.}
    \label{fig:misordered}
\end{figure}

\begin{proof}
    Let $\e>0$ be as in Definition \ref{def:misordered}.
    It is immediate that $t_1\le t_2-\e<t_2$. By \eqref{+<=-}, there exists an $r<t_2$ such that $\Upsilon_{(z,t_2)}\dir{R}{\tht}{+}(r)\le\Upsilon_{(z,t_2)}\dir{L}{\tht}{-}(r)$. This and the continuity of the interfaces imply that $t_1>-\infty$ and $\Upsilon_{(z,t_2)}\dir{L}{\tht}{-}(t_1) = \Upsilon_{(z,t_2)}\dir{R}{\tht}{+}(t_1)$. Then the definition of $t_1$ implies  $\Upsilon_{(z,t_2)}\dir{L}{\tht}{-}(r)<\Upsilon_{(z,t_2)}\dir{R}{\tht}{+}(r)$,  
    for all $r\in(t_1,t_2)$.

    Let $(y,r)\in U$. 
    Then by duality \eqref{no-intersection},
     $\geo\from{(y,r)}\dir{L}{\tht}{-}$ cannot cross $\Upsilon_{(z,t_2)}\dir{L}{\tht}{-}$ at a time $s\in(r,t_2)$ and $\geo\from{(y,r)}\dir{R}{\tht}{+}$ cannot cross  $\Upsilon_{(z,t_2)}\dir{R}{\tht}{+}$ during that time interval.
     This and the geodesics ordering \eqref{geo:mono} give that $\Upsilon_{(z,t_2)}\dir{L}{\tht}{-}(s)\le\geo\from{(y,r)}\dir{L}{\tht}{-}(s)\le\geo\from{(y,r)}\dir{R}{\tht}{+}(s)\le\Upsilon_{(z,t_2)}\dir{R}{\tht}{+}(s)$ for all $s\in[r,t_2]$. By the continuity of these paths, we get that $\geo\from{(y,r)}\dir{L}{\tht}{-}(t_2)=\geo\from{(y,r)}\dir{R}{\tht}{+}(t_2)=(z,t_2)$. See Figure \ref{fig:misordered}. By Lemma \ref{def:geoinstabpt}, $(y,r)\not\in\IG\tht$. We have thus shown that $U\subset\R^2\setminus\IG\tht$. 

    By Lemma \ref{no-isolated}, $(z,t_2)$ is either not left-isolated or not right-isolated. (In fact, Lemma \ref{lm:islandboundary}\eqref{lm:islandboundary.c} below, together with Lemma \ref{LRisolgeo}\eqref{LRisolgeo.a}, implies that $(z,t_2)$ is neither left- nor right-isolated.) Using Proposition \ref{prop:IGgoup}\eqref{IGgoup.e}, pick $\sigg\in\{-,+\}$ so that $\Ipath_{(z,t_2)}^{\tht\sig}(t_2)=(z,t_2)$.
    Define the functions $f,g:\R\to\R$ as follows. For $t\in[t_1,t_2]$, let $(f(t),t)=\Upsilon_{(z,t_2)}\dir{L}{\tht}{-}(t)$ and $(g(t),t)=\Upsilon_{(z,t_2)}\dir{R}{\tht}{+}(t)$.
     For $t\ge t_2$, let $(f(t),t)=(g(t),t)=\Ipath_{(z,t_2)}^{\tht\sig}(t)$. For $t\le t_1$, let $(f(t),t)=(g(t),t)=\Upsilon_{(z,t_2)}\dir{L}{\tht}{-}(t)$.
     By Lemma \ref{lem:IGgodown} and Proposition \ref{prop:IGgoup}\eqref{IGgoup.c}, both $(f(t),t)$ and $(g(t),t)$ are in $\IG\tht$, for all $t\in\R$. Both functions are also continuous and satisfy $f(t)\le g(t)$ for all $t$. Furthermore, $f(s)<g(s)$ for all $s\in(t_1,t_2)$. Thus, Lemma \ref{f<g:connected} applies and says that $U$ is a connected component of $V=\R^2\setminus\IG\tht$. Consequently, for any $(x,s)\in U$, $U=\island_{(x,s)}$.
    \end{proof}

Next, we show how to construct a stability island from left- or right-isolated instability points. See Figure \ref{fig:isotoisland}

\begin{lem}\label{lm:isotoisland}
Let $\w\in\Omega_0$ and $\tht\in\baddir$, and suppose $(v,s)\in\IG\tht$ is left- or right-isolated. 
If $(v,s)$ is left-isolated, let $(z,t_2)$ be the separation point of 
$\geo\from{(v,s)}\dir{L}{\tht}{-}$ and $\geo\from{(v,s)}\dir{L}{\tht}{+}$; 
if it is right-isolated, let $(z,t_2)$ be the separation point of 
$\geo\from{(v,s)}\dir{R}{\tht}{-}$ and $\geo\from{(v,s)}\dir{R}{\tht}{+}$. 
Then, in either case, $t_2>s$, $(z,t_2)\in\IG\tht$ and 
\[\Upsilon_{(z,t_2)}\dir{L}{\tht}{-}(r)< 
\Upsilon_{(z,t_2)}\dir{R}{\tht}{+}(r)
\quad\text{for all } r\in(s,t_2),\]
so $\Upsilon_{(z,t_2)}\dir{L}{\tht}{-}$ and $\Upsilon_{(z,t_2)}\dir{R}{\tht}{+}$ are misordered. 
Moreover, in the left-isolated case $\Upsilon_{(z,t_2)}\dir{R}{\tht}{+}(s)=(v,s)$ and $\geo\from{(v,s)}\dir{L}{\tht}{+}(r)<\Upsilon_{(z,t_2)}\dir{R}{\tht}{+}(r)$ for all $r\in(s,t_2)$, 
while in the right-isolated case $\Upsilon_{(z,t_2)}\dir{L}{\tht}{-}(s)=(v,s)$ and $\Upsilon_{(z,t_2)}\dir{L}{\tht}{-}(r)<\geo\from{(v,s)}\dir{R}{\tht}{-}(r)$ for all $r\in(s,t_2)$. 
\end{lem}

\begin{figure}[hpt]
    \includegraphics[width=3cm]{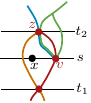}
    \caption{\small The construction of a stability island from a left-isolated point $(v,s)$. The right boundary (in red) is a path of $\tht+$ hugging shocks. Similarly, the left boundary (in orange) is a path of $\tht-$ hugging shocks. The two $\tht-$ and $\tht+$ shock interfaces are misordered.}
    \label{fig:isotoisland}
\end{figure}

\begin{proof}
    We treat the case where $(v,s)$ is left-isolated, the right-isolated case being symmetric. By Lemma \ref{LRisolgeo}\eqref{LRisolgeo.a}, $(v,s)$ is a $\tht+$ hugging shock point, which by Definition \ref{def:hugging} means that $t_2>s$. 
    Once $\geo\from{(v,s)}\dir{L}{\tht}{-}$ and $\geo\from{(v,s)}\dir{L}{\tht}{+}$ separate at $(z,t_2)$, the two geodesics cannot re-intersect, since, by \eqref{p2pleftmost}, they are both leftmost geodesics between any two of their points. This forces $\geo\from{(z,t_2)}\dir{L}{\tht}{-}\cap\geo\from{(z,t_2)}\dir{R}{\tht}{+}=\{(z,t_2)\}$ and, by Lemma \ref{def:geoinstabpt}, $(z,t_2)\in\IG\tht$.
    
    As $(z,t_2)$ is in the relative interior of a geodesic, 
    Lemma \ref{lm:geosplitsshocks} implies that the two  interfaces $\Upsilon\from{(z,t_2)}\dir{L}{\tht}{+}$ and $\Upsilon\from{(z,t_2)}\dir{R}{\tht}{+}$ are distinct and that 
    \begin{align}\label{2014}
    \geo\from{(v,s)}\dir{L}{\tht}{+}(r)<\Upsilon\from{(z,t_2)}\dir{R}{\tht}{+}(r),\text{ for all }r\in(s,t_2).
    \end{align}
    By the geodesic ordering \eqref{geo:mono}, $(z,t_2)=\geo\from{(v,s)}\dir{L}{\tht}{+}(t_2)\leq \geo\from{(v,s)}\dir{R}{\tht}{+}(t_2)$. Since $(v,s)\in\IG\tht$ and $(z,t_2)=\geo\from{(v,s)}\dir{L}{\tht}{-}(t_2)$, Lemma \ref{def:geoinstabpt} says $\geo\from{(v,s)}\dir{R}{\tht}{+}(t_2)\ne(z,t_2)$.
    Consequently, $(z,t_2)=\geo\from{(v,s)}\dir{L}{\tht}{+}(t_2)<\geo\from{(v,s)}\dir{R}{\tht}{+}(t_2)$.
    Therefore, by \eqref{no-intersection} and the continuity of the paths,  $\Upsilon\from{(z,t_2)}\dir{R}{\tht}{+}(r)\le\geo\from{(v,s)}\dir{R}{\tht}{+}(r)$ for all $r\in[s,t_2]$. Now, by the continuity of all these paths, we get $\Upsilon\from{(z,t_2)}\dir{R}{\tht}{+}(s) = (v,s)$. 
    
    Likewise, for all $r\in(s,t_2)$, $\Upsilon\from{(z,t_2)}\dir{L}{\tht}{-}(r)< \geo\from{(v,s)}\dir{L}{\tht}{-}(r)$. This, the geodesic ordering \eqref{geo:mono}, and \eqref{2014} give $\Upsilon\from{(z,t_2)}\dir{L}{\tht}{-}(r)<\Upsilon\from{(z,t_2)}\dir{R}{\tht}{+}(r)$ for all such $r$. The lemma is proved.
\end{proof}

The next lemma gives the converse of Lemma \ref{lem:shockstoisland}, hence showing that misordered shock interfaces characterize stability islands. Together with Lemma \ref{lem:shockstoisland}, this also shows that an island can be reconstructed from its tip point.

\begin{lem}\label{lem:islandtoshocks}
Let $\w\in\Omega_0$ and $\tht\in\baddir$. Let $q\in\R^2\setminus\IG\tht$. Then there exists a unique point $(z,t_2)\in\IG\tht$ such that $\Upsilon_{(z,t_2)}\dir{L}{\tht}{-}$ and $\Upsilon_{(z,t_2)}\dir{R}{\tht}{+}$ are misordered and $\island_q$ is given by \eqref{eq:shockstoisland}. \end{lem}

\begin{proof}
    By Lemma \ref{lm:islands}, $\island_q$ is bounded. Let 
    \[t_1=\inf\{t\in\R:\exists x\text{ with }(x,t)\in\island_q\}\quad\text{and}\quad
    t_2=\sup\{t\in\R:\exists x\text{ with }(x,t)\in\island_q\}.\]
    Since $\island_q$ is an open set, we must have $t_1<t_2$. 
    There exist $z$ and $z'$ such that $(z',t_1)$ and $(z,t_2)$ are in the closure of the island. Take any $t\in(t_1,t_2)$ and let 
    \[y=\inf\{x\in\R:(x,t)\in\island_q\}.\]
    By the connectedness of the island, $\island_q\cap(\R\times\{t\})\ne\varnothing$
    and thus $y<\infty$. Since the island is bounded, $y>-\infty$. 
    $(y,t)$ is on the boundary of the island and must therefore be in $\IG\tht$. Furthermore, by its definition, $(y,t)$ must be right-isolated. The claim now follows from Lemma \ref{lm:isotoisland}, which, in particular, uniquely identifies $(z,t_2)$.
\end{proof}

\begin{defn}
We call the point $(z,t_2)$ the \emph{tip} of the island $\island_q$. The point $\Upsilon_{(z,t_2)}\dir{L}{\tht}{-}(t_1)=\Upsilon_{(z,t_2)}\dir{R}{\tht}{+}(t_1)$, where $t_1$ is given by \eqref{t1}, is called the \emph{bottom} of the island $\island_q$.
\end{defn}

Next, we show that islands cannot touch.

\begin{lem}\label{lm:disjointislands}
Let $\w\in\Omega_0$ and $\tht\in\baddir$. The closures of any two distinct stability islands are disjoint.
\end{lem}

\begin{proof}
Let $U$ and $U'$ be two distinct stability islands, and suppose for contradiction that their closures intersect. By Lemma \ref{lem:islandtoshocks}, each island has a unique tip, denoted $(z,t_2)$ and $(z',t_2')$, respectively. The same lemma shows that the interfaces $\Upsilon_{(z,t_2)}\dir{L}{\tht}{-}$ and $\Upsilon_{(z,t_2)}\dir{R}{\tht}{+}$ are misordered, reintersect at some time $t_1<t_2$, and that $U$ is precisely the open region strictly between the curves $\Upsilon_{(z,t_2)}\dir{L}{\tht}{-}([t_1,t_2])$ and $\Upsilon_{(z,t_2)}\dir{R}{\tht}{+}([t_1,t_2])$. An analogous description holds for $U'$, with tip $(z',t_2')$ and reintersection time $t_1'<t_2'$. In particular, $(z,t_2)\neq(z',t_2')$, since otherwise $U=U'$.

Since $U$ and $U'$ are distinct connected components, their closures can intersect only along their boundaries. More precisely, any intersection must lie in
\[
\Upsilon_{(z,t_2)}\dir{R}{\tht}{+}([t_1,t_2])\cap \Upsilon_{(z',t_2')}\dir{L}{\tht}{-}([t'_1,t'_2])
\quad\text{or}\quad
\Upsilon_{(z,t_2)}\dir{L}{\tht}{-}([t_1,t_2])\cap \Upsilon_{(z',t_2')}\dir{R}{\tht}{+}([t'_1,t'_2]),
\]
and we treat the first case, the second being symmetric.

By \eqref{eq:shockstoisland}, the points $\Upsilon_{(z,t_2)}\dir{R}{\tht}{+}(r)$, $r\in(t_1,t_2)$, are left-isolated and, by Lemma \ref{LRisolgeo}\eqref{LRisolgeo.a}, are $\tht+$ hugging shocks. Likewise, for $r\in(t'_1,t'_2)$, the points $\Upsilon_{(z',t'_2)}\dir{L}{\tht}{-}(r)$ are $\tht-$ hugging shocks. It follows from Lemma \ref{LRisolgeo}\eqref{LRisolgeo.b} that these paths cannot intersect at any $r\in (t_1,t_2)\cap(t'_1,t'_2)$.

We next rule out the tips. By Lemma \ref{lm:isotoisland}, the point $(z,t_2)$ is where the geodesics $\geo\from{(v,t)}\dir{L}{\tht}{\pm}$ separate, for any $t\in(t_1,t_2)$ and $(v,t)=\Upsilon_{(z,t_2)}\dir{R}{\tht}{+}(t)$. Since $(z,t_2)$ lies in the relative interiors of these geodesics, \eqref{noshcoksongeo} implies that $(z,t_2)\notin \NU_1^{\tht-}\cup \NU_1^{\tht+}$. As $(z,t_2)$ is not the tip of $U'$ and the remainder of the boundary of $U'$ consists of $\tht\pm$ shock points, it follows that $(z,t_2)$ is not in the intersection of the two boundaries. The same argument applies to $(z',t_2')$.

Thus there are only remaining possible points of intersection:
$\Upsilon_{(z,t_2)}\dir{R}{\tht}{+}(t'_1)=\Upsilon_{(z',t_2')}\dir{L}{\tht}{-}(t'_1)$ with $t_1\le t'_1<t_2\wedge t_2'$, or $\Upsilon_{(z,t_2)}\dir{R}{\tht}{+}(t_1)=
\Upsilon_{(z',t_2')}\dir{L}{\tht}{-}(t_1)$ with $t'_1\le t_1<t_2\wedge t_2'$.
Since the two cases are symmetric, we treat only the first. See Figure \ref{fig:island3new}.

     \begin{figure}[hpt]
    \includegraphics[width=4.3cm]{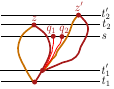}
    \caption{\small Proof of Lemma \ref{lm:disjointislands}. The bottom of the island to the right is on the boundary of the island to the left. This forces the interfaces $\Ipath^{+}_{q_1}$ and $\Ipath^{+}_{q_2}$ to intersect at time $t_1'$, causing a contradiction.}
    \label{fig:island3new}
    \end{figure}

Fix $s\in(t_1',t_2\wedge t_2')$. Then $q=\Upsilon_{(z,t_2)}\dir{R}{\tht}{+}(s)<q'=\Upsilon_{(z',t_2')}\dir{L}{\tht}{-}(s)$.
By Lemma \ref{lem:IGgodown}, $q,q'\in\IG\tht$. 
Since $q$ is left-isolated, Lemma \ref{no-isolated} says it is not right-isolated. Hence, by Lemma \ref{lm:isodense}, there exist $y_1<y_2$ such that $q_1=(y_1,s)$ and $q_2=(y_2,s)$ lie in $\IG\tht$, are not right-isolated, and lie strictly between $q$ and $q'$.
By parts \eqref{IGgoup.a}, \eqref{IGgoup.c}, \eqref{IGgoup.e}, and \eqref{IGgoup.g} of Proposition \ref{prop:IGgoup}, the interfaces $\Ipath_{q_1}^{\tht+}$ and $\Ipath_{q_2}^{\tht+}$ are disjoint, continuous, bi-infinite paths contained in $\IG\tht$. Since they remain in $\IG\tht$, they cannot enter either island, and must therefore lie weakly between $\Upsilon_{(z,t_2)}\dir{R}{\tht}{+}$ and $\Upsilon_{(z',t_2')}\dir{L}{\tht}{-}$ on $[t'_1,s]$. By continuity, this forces $\Ipath_{q_1}^{\tht+}(t_1')=\Ipath_{q_2}^{\tht+}(t_1')$, contradicting their disjointness. This contradiction shows that the closures of the two islands cannot intersect.
\end{proof}

We now define one more type of shock.  For illustrations, see the last configuration on the second row in Figure \ref{fig:geodesics}, the right panel in Figure \ref{fig:instint}, and Figure \ref{fig:island} (where the snowbird shock is the point $(u,t_1)$).

\begin{defn}\label{def:snowbird}
    For $\tht\in\baddir$, a point $(x,s)\in\IG\tht\cap\NU_1^{\tht-}\cap\NU_1^{\tht+}$ is a \emph{snowbird shock} if there exists a $t>s$ such that for all $r\in(s,t), \geo\from{(x,s)}\dir{L}{\tht}{-}(r)<\geo\from{(x,s)}\dir{R}{\tht}{-}(r)=\geo\from{(x,s)}\dir{L}{\tht}{+}(r)< \geo\from{(x,s)}\dir{R}{\tht}{+}(r)$. 
\end{defn}

\begin{lem}\label{topsnowbird}
Let $\w\in\Omega_0$ and $\tht\in\baddir$. Suppose $(x,s)\in\IG\tht\cap\NU_1^{\tht-}\cap\NU_1^{\tht+}$ is a snowbird shock. Let $(z,t)$ be the point at which $\geo\from{(x,s)}\dir{R}{\tht}{-}$ and $\geo\from{(x,s)}\dir{L}{\tht}{+}$ separate. 
Then $\geo\from{(x,s)}\dir{L}{\tht}{-}(r)<\geo\from{(x,s)}\dir{R}{\tht}{-}(r)=\geo\from{(x,s)}\dir{L}{\tht}{+}(r)< \geo\from{(x,s)}\dir{R}{\tht}{+}(r)$ for all $r\in(s,t]$.
\end{lem}

\begin{proof}
 By the geodesic ordering \eqref{geo:mono},    $\geo\from{(x,s)}\dir{L}{\tht}{-}(r)\le\geo\from{(x,s)}\dir{R}{\tht}{-}(r)=\geo\from{(x,s)}\dir{L}{\tht}{+}(r)\le \geo\from{(x,s)}\dir{R}{\tht}{+}(r)$ for all $r\in(s,t]$.
 The claim then follows from the fact that, for each $S\in{L,R}$, the geodesics $\geo\from{(x,s)}\dir{S}{\tht}{-}$ and $\geo\from{(x,s)}\dir{S}{\tht}{+}$ separate immediately and therefore can never reintersect, as any such reintersection would violate the extremality properties \eqref{p2pleftmost} and \eqref{p2prightmost}.
\end{proof}

The next lemma characterizes the bottom point of the island in Lemma \ref{lem:shockstoisland} as being a snowbird shock point.

\begin{lem}\label{lm:snowbird}
Assume the setting of Lemma \ref{lem:shockstoisland} holds. Then $\Upsilon_{(z,t_2)}\dir{L}{\tht}{-}(t_1)=\Upsilon_{(z,t_2)}\dir{R}{\tht}{+}(t_1)$ is a snowbird shock.
\end{lem}

   \begin{proof}
    By Lemma \ref{lem:shockstoisland}, the interfaces $\Upsilon_{(z,t_2)}\dir{L}{\tht}{-}$ and $\Upsilon_{(z,t_2)}\dir{R}{\tht}{+}$ are misordered and the region $U$ strictly between them is a stability island.
   
     Abbreviate $p=(u,t_1)=\Upsilon_{(z,t_2)}\dir{L}{\tht}{-}(t_1) = \Upsilon_{(z,t_2)}\dir{R}{\tht}{+}(t_1)$. By \eqref{shocksplitsgeo}, 
     \[\geo\from{p}\dir{L}{\tht}{+}(r)<\Upsilon_{(z,t_2)}\dir{R}{\tht}{+}(r)<\geo\from{p}\dir{R}{\tht}{+}(r)\quad\text{and}\quad
     \geo\from{p}\dir{L}{\tht}{-}(r)<\Upsilon_{(z,t_2)}\dir{L}{\tht}{-}(r)<\geo\from{p}\dir{R}{\tht}{-}(r),\] 
     for all $t_1<r<t_2$. In particular, $\geo\from{p}\dir{L}{\tht}{-}$ and $\geo\from{p}\dir{R}{\tht}{+}$ immediately separate (as they should, since $p\in\IG\tht$). We consider two cases.\smallskip

     \begin{figure}[hpt]
    \includegraphics[width=4.3cm]{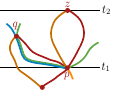}
    \caption{\small An illustration of Case 1 in the proof of Lemma \ref{lm:snowbird}. The two regions between the misordered shock interfaces out of $(z,t_2)$ and $q$ are instability islands. Their boundaries intersect at $p$, which cannot happen.}
    \label{fig:island3}
    \end{figure}

     {\bf Case 1.} Suppose there exists $\ep>0$ such that either $\geo\from{p}\dir{L}{\tht}{-}(r)=\geo\from{p}\dir{L}{\tht}{+}(r)$ for all $r\in[t_1,t_1+\ep]$ or $\geo\from{p}\dir{R}{\tht}{-}(r)=\geo\from{p}\dir{R}{\tht}{+}(r)$ for all $r\in[t_1,t_1+\ep]$. The two situations are symmetric and we treat the first one.  
     See Figure \ref{fig:island3} for an illustration. 
     In this case, $p$ is a $\tht+$ hugging shock and, by Lemma \ref{LRisolgeo}\eqref{LRisolgeo.a}, it is left-isolated. Let $q=(w,r)$ denote the point where $\geo\from{p}\dir{L}{\tht}{-}$ and $\geo\from{p}\dir{L}{\tht}{+}$ separate. 
     By Lemma \ref{lm:isotoisland}, $q\in\IG\tht$, $\Upsilon_{q}\dir{R}{\tht}{+}(t_1)=p$, the geodesics $\Upsilon_{q}\dir{L}{\tht}{-}$ and $\Upsilon_{q}\dir{R}{\tht}{+}$ are misordered, and $\geo\from{p}\dir{L}{\tht}{+}(t)<\Upsilon_{q}\dir{R}{\tht}{+}(t)$ for all $t\in(t_1,r)$. By Lemma \ref{lem:shockstoisland}, the region $U'$ strictly between  $\Upsilon_{q}\dir{L}{\tht}{-}$ and $\Upsilon_{q}\dir{R}{\tht}{+}$ is a stability island. This contradicts Lemma \ref{lm:disjointislands}, as now $p$ belongs the the boundaries of two distinct stability islands. Thus, the situation in Case 1 cannot occur.\smallskip

    {\bf Case 2.} Suppose now that $\geo\from{p}\dir{L}{\tht}{-}(r)<\geo\from{p}\dir{L}{\tht}{+}(r)$ and $\geo\from{p}\dir{R}{\tht}{-}(r)<\geo\from{p}\dir{R}{\tht}{+}(r)$ for all $r>t_1$.  Then by \eqref{no4stars}, it must be the case that there exists a $\delta>0$ such that $\geo\from{p}\dir{L}{\tht}{+}(r)=\geo\from{p}\dir{R}{\tht}{-}(r)$ for all $r\in[t_1,t_1+\delta]$, which means $p=(u,t_1)$ is a snowbird shock point, as claimed. See Figure \ref{fig:island}. The lemma is proved.
\end{proof}

The following lemma is the converse of Lemma \ref{lm:snowbird}. It shows that each snowbird shock is the bottom point of an island and that the island can be reconstructed from its bottom point.

\begin{lem}\label{lem:SBisisland}
    Let $\w\in\Omega_0$, $\tht\in\baddir$, and $(u,t_1)\in\IG\tht$. Suppose that $(u,t_1)$ is a snowbird shock. Then $(u,t_1)$ is the bottom point of a stability island, with the top point being the point $(z,t_2)$, where $\geo\from{(u,t_1)}\dir{R}{\tht}{-}$ and $\geo\from{(u,t_1)}\dir{L}{\tht}{+}$ first separate. See Figure \ref{fig:island}.
\end{lem}

\begin{proof}
    Let $(z,t_2)$ be as in the claim. By Lemma \ref{lm:geosplitsshocks}, $\geo\from{(u,t_1)}\dir{L}{\tht}{+}(r)<\Upsilon\from{(z,t_2)}\dir{R}{\tht}{+}(r)$ for all $r\in(t_1,t_2)$. By Lemma \ref{topsnowbird}, $\geo\from{(u,t_1)}\dir{L}{\tht}{+}(t_2)<\geo\from{(u,t_1)}\dir{R}{\tht}{+}(t_2)$. Therefore, by \eqref{no-intersection} and the continuity of the paths, we also have $\Upsilon\from{(z,t_2)}\dir{R}{\tht}{+}(r)\le\geo\from{(u,t_1)}\dir{R}{\tht}{+}(r)$ for all $r\in(t_1,t_2)$. Since all these paths are continuous, this forces $\Upsilon\from{(z,t_2)}\dir{R}{\tht}{+}(t_1)=(u,t_1)$. Likewise, $\geo\from{(u,t_1)}\dir{L}{\tht}{-}(r)\le\Upsilon\from{(z,t_2)}\dir{L}{\tht}{-}(r)<\geo\from{(u,t_1)}\dir{R}{\tht}{-}(r)$ for all $r\in(t_1,t_2)$ and $\Upsilon\from{(z,t_2)}\dir{L}{\tht}{-}(t_1)=(u,t_1)$. Since $\geo\from{(u,t_1)}\dir{R}{\tht}{-}$ and $\geo\from{(u,t_1)}\dir{L}{\tht}{+}$ match over the interval $(t_1,t_2)$, we get that $\Upsilon\from{(z,t_2)}\dir{L}{\tht}{-}$ remains strictly to the left of $\Upsilon\from{(z,t_2)}\dir{R}{\tht}{+}$ over this time interval. Thus, the two geodesics are misordered and, by Lemma \ref{eq:shockstoisland}, the region strictly between them is a stability island, with its top at $(z,t_2)$ and its bottom at $(u,t_1)$.
\end{proof}

In the proof of Lemma \ref{lm:disjointislands}, we observed that the tip of a stability island is neither a $\tht-$ shock nor a $\tht+$ shock. See the point $(z,t_2)$ in Figures \ref{fig:isotoisland} and \ref{fig:island}.  This motivates the following definition. 

\begin{defn}\label{def:pns}
We call a point in $\IG\tht \setminus (\NU_1^{\tht-} \cup \NU_1^{\tht+})$ a \emph{proper non-shock instability point} (pns point), and denote the set of all such points by $\IGpns\tht$.
See the first configuration on the second line in Figure \ref{fig:geodesics}.
\end{defn}

The following is immediate from \eqref{geo:no-lollipop} and Lemma \ref{def:geoinstabpt}. 

\begin{lem}\label{lm:pns}
Let $\w\in\Omega_0$ and $\tht\in\baddir$. Then $p\in\IGpns\tht$ if and only if $\geo\from{p}\dir{L}{\tht}{\sig}=\geo\from{p}\dir{R}{\tht}{\sig}$, for both $\sigg\in\{-,+\}$, and $\geo\from{p}\dir{L}{\tht}{-}\cap\geo\from{p}\dir{R}{\tht}{+}=\{p\}$.
\end{lem}

The next lemma completes the description of the boundaries of islands. See Figure \ref{fig:island} for an illustration. Recall Definition \ref{def:hugging} of a $\tht\sigg$ hugging shock, $\sigg\in\{-,+\}$. 

\begin{lem}\label{lm:islandboundary}
Assume the setting of Lemma \ref{lem:shockstoisland} holds. We have the following.
    \begin{enumerate}[label={\rm(\alph*)}, ref={\rm\alph*}] \itemsep=1pt 
        \item\label{lm:islandboundary.a} For any $r\in(t_1,t_2)$, $\Upsilon\from{(z,t_2)}\dir{L}{\tht}{-}(r)\in\IG\tht$ is a $\tht-$ hugging shock and $\Upsilon\from{(z,t_2)}\dir{R}{\tht}{+}(r)\in\IG\tht$ is a $\tht+ $ hugging shock.
        \item\label{lm:islandboundary.b} $\Upsilon\from{(z,t_2)}\dir{L}{\tht}{-}(t_1)=\Upsilon\from{(z,t_2)}\dir{L}{\tht}{-}(t_1) \in\IG\tht\cap\NU^{\tht+}\cap \NU^{\tht-}$ and is a snowbird shock.
        \item\label{lm:islandboundary.c} $(z,t_2)\in\IGpns\tht$.
    \end{enumerate}
\end{lem}

\begin{figure}[hpt]
    \includegraphics[width=5.5cm]{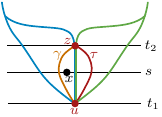}    \caption{\small A sketch of an island. $\island_{(x,s)}$ is the area strictly between the (shock) paths $\gamma=\Upsilon\from{(z,t_2)}\dir{L}{\tht}{-}$ and $\tau=\Upsilon\from{(z,t_2)}\dir{R}{\tht}{+}$. The bottom of the island, $(u,t_1)$, is a snowbird double shock instability point. The tip of the island, $(z,t_2)$, is a pns point. The shocks along $\tau$ are all $\tht+$ hugging shocks and the shocks along $\gamma$ are all $\tht-$ hugging shocks.}
    \label{fig:island}
\end{figure}

\begin{proof}
    Since $(z,t_2)\in\IG\tht$, Lemma \ref{lem:IGgodown} implies that $\Upsilon\from{(z,t_2)}\dir{L}{\tht}{-}$ and $\Upsilon\from{(z,t_2)}\dir{R}{\tht}{+}$ are both in $\IG\tht$. 
    By \eqref{eq:shockstoisland}, $\Upsilon\from{(z,t_2)}\dir{L}{\tht}{-}(r)$ is right-isolated and $\Upsilon\from{(z,t_2)}\dir{R}{\tht}{+}(r)$ is left-isolated, for all $r\in(t_1,t_2)$. 
    Part \eqref{lm:islandboundary.a} then follows from Lemma \ref{LRisolgeo}\eqref{LRisolgeo.a}.\smallskip

    Part \eqref{lm:islandboundary.b} is in Lemma \ref{lm:snowbird}.
    Part \eqref{lm:islandboundary.c} follows from the fact that $(z,t_2)$ is interior to both $\geo\from{(v,s)}\dir{L}{\tht}{-}$ and $\geo\from{(v,s)}\dir{L}{\tht}{+}$ and hence, by \eqref{noshcoksongeo}, 
    cannot be in $\NU_1^{\tht-}\cup\NU_1^{\tht+}$.
\end{proof}




\subsection{Relation to competition interfaces}\label{sec:islands+cifs}

The following lemma relates instability islands to the bubbles formed between the interfaces $\Ipath_{(x,s)}^{\tht\pm}$ introduced in Proposition \ref{prop:IGgoup}. 
See Figure \ref{fig:island} for an illustration.

\begin{lem}\label{lm:bdry}
    Let $\w\in\Omega_0$, $\tht\in\baddir$, and  $(x,s)\in\R^2$. For $t\in\R$, let $\gamma(t)=(I^{\tht-}_{(x,s)}(t),t)$ and $\tau(t)=(I^{\tht+}_{(x,s)}(t),t)$. 
    \begin{enumerate}[label={\rm(\alph*)}, ref={\rm\alph*}] \itemsep=1pt
    \item\label{lm:bdry.a} $\gamma\preceq\tau$ and $(x,s)$ is weakly between the two paths. If $\gamma\ne\tau$, then the connected components of $\{(z,r):\gamma(r)<(z,r)<\tau(r)\}$ are all stability islands.
    \item\label{lm:bdry.b} Suppose $(x,s)$ is in the closure of some stability island. Let $t_1=\sup\{r\le s:\gamma(r)=\tau(r)\}$ and $t_2=\inf\{r\ge s:\gamma(r)=\tau(r)\}$. Then $-\infty<t_1<t_2<\infty$, $s\in[t_1,t_2]$, $\gamma(t_1) = \tau(t_1)$, $\gamma(t_2) = \tau(t_2)$, for any $r\in(t_1,t_2)$, $\gamma(r) < \tau(r)$, and the island is given by
    \begin{align}\label{eqn:bdry}
    \{(z,r):\gamma(r)<(z,r)<\tau(r),\,t_1<r<t_2\}.
    \end{align}
    \end{enumerate}
\end{lem}

\begin{proof} 
Part \eqref{lm:bdry.a}. By Lemma \ref{I-prop}\eqref{I-prop.a}, $\gamma\preceq\tau$ and, by part \eqref{I-prop.e} of that lemma, $(x,s)$ is between the two paths. By Lemma \ref{I-inst}\eqref{I-inst.a}, both paths are on $\IG\tht$. By Lemma \ref{I-prop}\eqref{I-prop.h} and Definition \ref{def:buseinstab}, if $(y,t)$ is strictly between $\gamma(t)$ and $\tau(t)$, then $(y,t)\not\in\IG\tht$. Thus, the whole subset of $\R^2$ that is strictly between $\gamma$ and $\tau$ is inside $\R^2\setminus\IG\tht$. By Lemma \ref{f<g:connected}, if this region is not empty (i.e.\ if $\gamma\ne\tau$), then its connected components are connected components of $\R^2\setminus\IG\tht$ and, hence, are all stability islands.
\smallskip

Part \eqref{lm:bdry.b}. By Lemma \ref{lm:disjointislands}, distinct islands have  disjoint closures. Thus, there is a unique island containing $(x,s)$ in its closure. Let $(x',s')$ be a point in that island. By part \eqref{lm:bdry.a}, this island is the connected component of the region strictly between $\gamma$ and $\tau$ that contains $(x',s')$. In particular, $\gamma(s')<x'<\tau(s')$.
By Lemma \ref{f<g:connected}, this region is given by \eqref{eqn:bdry}, but with $t_1=\sup\{r\le s':\gamma(r)=\tau(r)\}$ and $t_2=\inf\{r\ge s':\gamma(r)=\tau(r)\}$. Since this connected component is not empty, we have $t_1<t_2$.
By Lemma \ref{lm:islands}, this connected component is bounded and, consequently, $t_1>-\infty$ and $t_2<\infty$.
From these definitions of $t_1$ and $t_2$, together with $\gamma\preceq\tau$, and the continuity of the two paths, we get that $\gamma(t_1)=\tau(t_1)$, $\gamma(t_2)=\tau(t_2)$, and $\gamma|_{(t_1,t_2)}\prec\tau|_{(t_1,t_2)}$. 
 Since $(x,s)$ is in the closure of this region, we have $t_1\le s\le t_2$ and the two formulas for $t_1$ and $t_2$ claimed in part \eqref{lm:bdry.b} follow.
\end{proof}

The following is immediate from Proposition \ref{prop:IGgoup}\eqref{IGgoup.f} 
and Lemmas 
\ref{lm:disjointislands} and  \ref{lm:bdry}.

\begin{lem}\label{lm:disjointI}
Let $\w\in\Omega_0$, $\tht\in\baddir$, and $(x,s)\in\IG\tht$. The intervals making up the connected components of $\{t\in\R:\Ipath_{(x,s)}^{\tht-}(t)<\Ipath_{(x,s)}^{\tht+}(t)\}$ are all bounded and their closures are disjoint.
\end{lem}


The last lemma of this section states that both $\tht-$ and $\tht+$ shock points are dense on competition interfaces, off of the boundaries of stability islands. Proposition \ref{prop:islandconfigs}\eqref{prop:islandconfigs.c} below shows that this density continues to hold even on the boundaries of stability islands.

\begin{lem}\label{shocksdenseondust} 
Let $\w\in\Omega_0$, $\tht\in\baddir$, and $\sigg\in\{-,+\}$. Then $\NU_1^{\tht\sig}\cap\Ipath_{(x,s)}^{\tht-}(\R)\cap\Ipath_{(x,s)}^{\tht+}(\R)$ is dense in $\Ipath_{(x,s)}^{\tht-}(\R)\cap \Ipath_{(x,s)}^{\tht+}(\R)$.
\end{lem}

\begin{proof}
Take a time $s'$ such that $\Ipath_{(x,s)}^{\tht-}(s')=\Ipath_{(x,s)}^{\tht+}(s')$. By Lemma \ref{lm:disjointI}, there exist $r'<t'$ such that $s'\in[r',t']$ and 
$\Ipath_{(x,s)}^{\tht-}\big|_{[r',t']}=\Ipath_{(x,s)}^{\tht+}\big|_{[r',t']}$.
We thus drop the sign distinction from the notation.
Take any $r<t$ in $[r',t']$. We prove that for any $\sigg\in\{-,+\}$, $\Ipath_{(x,s)}^{\tht}([r,t])\cap\NU_1^{\tht\sig}\ne\varnothing$.

By \eqref{Isplit-strict}, $\Ipath^{\tht}_{(x,s)}$ goes strictly between $\geo\from{\Ipath^{\tht}_{(x,s)}(r)}\dir{L}{\tht}{-}$ and $\geo\from{\Ipath^{\tht}_{(x,s)}(r)}\dir{R}{\tht}{+}$ on the time interval $(r,t]$. Take $x$ such that $(x,t)$ is strictly between the left geodesic and the competition interface. Then, by the duality \eqref{no-intersection} and the continuity of the paths, any $\tht-$ shock interface from $(x,t)$ must intersect $\Ipath^{\tht}_{(x,s)}([r,t))$. Since $s'$ and $r<t$ are arbitrary, we have shown the density of $\tht-$ shocks. A similar argument works for $\tht+$ shocks.
\end{proof}

\section{Configurations of semi-infinite \texorpdfstring{$\theta$}{theta}-directed geodesics}\label{sec:geodesics}

In this section, we classify all possible configurations of $\theta$-directed geodesics emanating from each point $(x,s)\in\R^2$ and identify the regions in which each type occurs. This completes the preparatory material needed for the proofs of our main theorems, which are given in Section \ref{sec:mainproofs}. The section concludes with Lemma \ref{shocksintersect}, which shows that $\tht\sigg$ shock interfaces do not coalesce at dust points, thereby underscoring the importance of the boundaries of stability islands relative to dust points.

By \eqref{allgeo}, the only semi-infinite geodesics are the Busemann geodesics. Therefore, it suffices to describe the configurations of Busemann geodesics.

Recall that when $\tht \notin \baddir$, there is no sign distinction, and we omit the sign from the notation.
By Remark \ref{rk:Mgeo}, when $\tht \in \baddir$ and there are exactly two $W^{\tht+}$-geodesics emanating from $(x,s)$, we adopt the convention that $\geo\from{(x,s)}\dir{M}{\tht}{+} = \geo\from{(x,s)}\dir{R}{\tht}{+}$. Similarly, if there are exactly two $W^{\tht-}$-geodesics from $(x,s)$, we set $\geo\from{(x,s)}\dir{M}{\tht}{-} = \geo\from{(x,s)}\dir{L}{\tht}{-}$. In the case $\tht \in \R \setminus \baddir$, if exactly two $W^{\tht}$-geodesics emanate from $(x,s)$, Remark \ref{rk:Mgeostable} specifies the convention $\geo\from{(x,s)}\dir{M}{\tht}{} = \geo\from{(x,s)}\dir{L}{\tht}{}$.
Finally, recall that we write $\gamma|_I$ for the restriction of $\gamma$ to the interval $I$, and $\gamma \prec \pi$ to mean that $\gamma(s) < \pi(s)$ for all $s$ in the intersection of the two domains.

The first proposition addresses the configurations of $\tht$-directed semi-infinite geodesics when $\tht$ is a stability direction. 

\begin{prop}\label{prop:geodconfigstable}
    Let $\w \in \Omega_0$, $\tht \in \R\setminus\baddir$, and $(x,s)\in\R^2$. Then the following are the only possible configurations of $\tht$-directed geodesics emanating from $(x,s)$. Each of the configurations occurs at a dense set of starting points $(x,s)$.  See the first four configurations on the first row of Figure \ref{fig:geodesics}.
    \begin{enumerate}[label={\rm(\alph*)}, ref={\rm\alph*}] \itemsep=1pt
        \item\label{prop:stablegeo_a} $(x,s)\not\in \NU_1^\tht$ and hence $\geo\from{(x,s)}\dir{L}{\tht}{} = \geo\from{(x,s)}\dir{R}{\tht}{}$. \item\label{prop:stablegeo_b} $(x,s)\in\NU_1^{\tht}$. There exist times $t>s'>s$ such that $t = s+\age^{\tht}(x,s)$, 
        \begin{align}\label{aux2040}
        \geo\from{(x,s)}\dir{L}{\tht}{}\big|_{(s,t)}\prec\geo\from{(x,s)}\dir{R}{\tht}{}\big|_{(s,t)},\quad
        \geo\from{(x,s)}\dir{L}{\tht}{}\big|_{[t,\infty)} = \geo\from{(x,s)}\dir{R}{\tht}{}\big|_{[t,\infty)},
        \end{align}
        and exactly one of the following occurs:
        \vspace{5pt}
              
        \begin{enumerate}[label={{\rm(}\ref{prop:geods-stability.b}.\rm\roman*)}, ref={\ref{prop:geods-stability.b}.\rm\roman*}] \itemsep=1pt
            \item\label{prop:stablegeo_bi}   $\geo\from{(x,s)}\dir{M}{\tht}{}=\geo\from{(x,s)}\dir{L}{\tht}{}$ for all time, or

            \item\label{prop:stablegeo_bii} $\geo\from{(x,s)}\dir{L}{\tht}{}\big|_{(s,s')}\prec\geo\from{(x,s)}\dir{M}{\tht}{}\big|_{(s,s')}\prec\geo\from{(x,s)}\dir{R}{\tht}{}\big|_{(s,s')}$, and the middle geodesic intersects exactly one of $\geo\from{(x,s)}\dir{L}{\tht}{}$ or $\geo\from{(x,s)}\dir{R}{\tht}{}$ at time $s'$, proceeding with said geodesic for all times beyond $s'$. Each of the two configurations occurs at a dense set of starting points $(x,s)$.
        \end{enumerate}       
    \end{enumerate}
\end{prop}

The next three propositions deal with the case $\tht\in\baddir$.
First, we list all possible configurations of $\tht$-directed semi-infinite geodesics out of stable points.

\begin{prop}\label{prop:geods-stability}
    Let $\w \in \Omega_0$ and $\tht \in \baddir$.  
Suppose $(x,s) \in \R^2 \setminus \IG\tht$.  
Then the following are the only possible configurations of $\tht$-directed geodesics emanating from $(x,s)$. 
Each of the configurations occurs for a dense set of such starting points $(x,s)$. See the last four configurations on the first row in Figure \ref{fig:geodesics}.    
\begin{enumerate}[label={\rm(\alph*)}, ref={\rm\alph*}] \itemsep=1pt
        \item\label{prop:geods-stability.a} $(x,s)\not\in\NU_1^{\tht+}\cup\NU_1^{\tht-}$. There exists $t>s$ such that 
        \begin{align*}
            &\geo\from{(x,s)}\dir{L}{\tht}{-}\big|_{[s,t]} = \geo\from{(x,s)}\dir{R}{\tht}{+}\big|_{[s,t]}\quad\text{and}\quad
            \geo\from{(x,s)}\dir{L}{\tht}{-}(t) = \geo\from{(x,s)}\dir{R}{\tht}{+}(t)\in\IGpns\tht.
        \end{align*}
 
        \item\label{prop:geods-stability.b} $(x,s)\in\NU_1^{\tht+}\cap\NU_1^{\tht-}$. There exist times $t>t'>s'>s$ such that $t' = s+\age^{\tht-}(x,s) = s+\age^{\tht+}(x,s)$, 
        \begin{align*}
        &\geo\from{(x,s)}\dir{L}{\tht}{-}\big|_{(s,t')} = \geo\from{(x,s)}\dir{L}{\tht}{+}\big|_{(s,t')}\prec\geo\from{(x,s)}\dir{R}{\tht}{-}\big|_{(s,t')} = \geo\from{(x,s)}\dir{R}{\tht}{+}\big|_{(s,t')},\\
        &\geo\from{(x,s)}\dir{L}{\tht}{-}\big|_{[t',t]} = \geo\from{(x,s)}\dir{R}{\tht}{+}\big|_{[t',t]},\quad
        \geo\from{(x,s)}\dir{L}{\tht}{-}(t) = \geo\from{(x,s)}\dir{R}{\tht}{+}(t)\in\IGpns\tht,
        \end{align*}
        and exactly one of the following occurs:
        \vspace{5pt}
              
        \begin{enumerate}[label={{\rm(}\ref{prop:geods-stability.b}.\rm\roman*)}, ref={\ref{prop:geods-stability.b}.\rm\roman*}] \itemsep=1pt
            \item\label{prop:geods-stability.bi}   $\geo\from{(x,s)}\dir{M}{\tht}{-}=\geo\from{(x,s)}\dir{L}{\tht}{-}$ and $\geo\from{(x,s)}\dir{M}{\tht}{+}=\geo\from{(x,s)}\dir{R}{\tht}{+}$, or

            \item\label{prop:geods-stability.bii} $\geo\from{(x,s)}\dir{M}{\tht}{-}\bigl|_{[s,t]}=\geo\from{(x,s)}\dir{M}{\tht}{+}\big|_{[s,t]}$, $\geo\from{(x,s)}\dir{L}{\tht}{-}\big|_{(s,s')}\prec\geo\from{(x,s)}\dir{M}{\tht}{-}\big|_{(s,s')}=\geo\from{(x,s)}\dir{M}{\tht}{+}\big|_{(s,s')}\prec\geo\from{(x,s)}\dir{R}{\tht}{+}\big|_{(s,s')}$, and the matching middle geodesics intersect exactly one of $\geo\from{(x,s)}\dir{L}{\tht}{-}$ or $\geo\from{(x,s)}\dir{R}{\tht}{+}$ at time $s'$, proceeding with said geodesic until time $t'$. Each of the two configurations occurs for a dense set of starting points $(x,s)\in\R^2\setminus\IG\tht$.
        \end{enumerate}       
    \end{enumerate}
\end{prop}

Next, we describe the configurations of $\tht$-directed semi-infinite geodesics emanating from dust points---that is, instability points that do not lie on the boundary of any stability island.

\begin{prop}\label{prop:geods-dust}
    Let $\w\in\Omega_0$, $\tht\in\baddir$, and  $(x,s)\in\IG\tht$. Suppose that $(x,s)$ is not on the boundary of any stability island. Then the following are the only three possible $\tht$-directed geodesic configurations out of $(x,s)$, and each of the three occurs on a set of starting points that is dense in $\IG\tht$. See the first three configurations on the second row of Figure \ref{fig:geodesics}.
    \begin{enumerate}[label={\rm(\alph*)}, ref={\rm\alph*}] \itemsep=1pt
        \item\label{prop:geods-dust.a} $(x,s)\not\in\NU_1^{\tht-}\cup\NU_1^{\tht+}$. Then $(x,s)\in\IGpns\tht$.
        
        \item\label{prop:geods-dust.b} $(x,s)\in \NU_1^{\tht-}$. Then $(x,s)\not\in\NU_1^{\tht+}$, there are no distinct middle geodesics out of $(x,s)$, and
        \[
        \geo\from{(x,s)}\dir{L}{\tht}{-}\big|_{(s,\infty)}\preceq \geo\from{(x,s)}\dir{R}{\tht}{-}\big|_{(s,\infty)}\prec\geo\from{(x,s)}\dir{L}{\tht}{+}\big|_{(s,\infty)} = \geo\from{(x,s)}\dir{R}{\tht}{+}\big|_{(s,\infty)}.
        \]

        \item\label{prop:geods-dust.c} $(x,s)\in \NU_1^{\tht+}$. Then $(x,s)\not\in\NU_1^{\tht-}$, there are no distinct middle geodesics out of $(x,s)$, and
        \[
        \geo\from{(x,s)}\dir{L}{\tht}{-}\big|_{(s,\infty)}= \geo\from{(x,s)}\dir{R}{\tht}{-}\big|_{(s,\infty)}\prec\geo\from{(x,s)}\dir{L}{\tht}{+}\big|_{(s,\infty)} \preceq\geo\from{(x,s)}\dir{R}{\tht}{+}\big|_{(s,\infty)}.
        \]
    \end{enumerate}
\end{prop}

Lastly, we give the configurations of $\tht$-directed semi-infinite geodesics out of instability points that are on boundaries of stability islands.

\begin{prop}\label{prop:islandconfigs}
    Let $\w\in\Omega_0$ and $\tht\in\baddir$. Suppose that $\tau:[s,t]\to\R$ is the right boundary of a stability island, including its tip and bottom points. Then for $r\in[s,t]$, $\tau(r)\in\IG\tht$, and the following are the only possible $\tht$-directed geodesic configurations out of $\tau(r)$.  See the first and last configurations on the second row and the configurations on the third row in Figure \ref{fig:geodesics}.
    
    \begin{enumerate}[label={\rm(\alph*)}, ref={\rm\alph*}] \itemsep=1pt
        \item\label{prop:islandconfigs.a} $r = t$, $\tau(t)$ is the tip of the island and is in $\IGpns\tht$. 
       \item\label{prop:islandconfigs.b} $r=s$, $\tau(s)$ is the bottom of the island and is a snowbird shock. 
        \item\label{prop:islandconfigs.c} $r\in(s,t)$, $\tau(r)$ is a $\tht+$ hugging shock, $\geo\from
        {\tau(r)}\dir{L}{\tht}{-}\big|_{[r,t]}=\geo\from
        {\tau(r)}\dir{L}{\tht}{+}\big|_{[r,t]}$,  exactly one of the following happens and each of the three types {\rm(}\ref{prop:islandconfigs.ci}-\ref{prop:islandconfigs.ciii}{\rm)} occurs at a dense set of times $r\in(s,t)$.
         \vspace{5pt}
              
        \begin{enumerate}[label={{\rm(}\ref{prop:islandconfigs.c}.\rm\roman*)}, ref={\ref{prop:islandconfigs.c}.\rm\roman*}] \itemsep=1pt
            
        \item\label{prop:islandconfigs.ci} $\tau(r)\notin\NU_1^{\tht-}$ and $\geo\from
        {\tau(r)}\dir{M}{\tht}{+} = \geo\from
        {\tau(r)}\dir{R}{\tht}{+}$ {\rm(}i.e.\ no middle $W^{\tht+}$-geodesic{\rm)}.
        \item\label{prop:islandconfigs.cii} $\tau(r) \in \NU_1^{\tht-}$ and hence $\geo\from{\tau(r)}\dir{L}{\tht}{-}$ and $\geo\from{\tau(r)}\dir{R}{\tht}{-}$ are initially distinct. In this case, both geodesics go strictly left of $\tau$ on $(r,t)$ and coalesce strictly before $t$. Furthermore, $\geo\from{\tau(r)}\dir{M}{\tht}{+}\big|_{[r,t]}=\geo\from{\tau(r)}\dir{R}{\tht}{-}\big|_{[r,t]}$.
       \item\label{prop:islandconfigs.ciii} $\tau(r)\notin\NU_1^{\tht-}$ and $\geo\from{\tau(r)}\dir{M}{\tht}{+}$ is initially distinct from both $\geo\from{\tau(r)}\dir{L}{\tht}{+}$ and $\geo\from{\tau(r)}\dir{R}{\tht}{+}$. In this case, the middle geodesic coalesces with one of these strictly before the other. If it first coalesces with the left geodesic, this occurs strictly after time $t$.
        Moreover, there exists $r_0\in(s,t)$ such that for all $r\in(s,r_0)$ with $\tau(r)$ of type \eqref{prop:islandconfigs.ciii}, $\geo\from{\tau(r)}\dir{M}{\tht}{+}$ first coalesces with $\geo\from{\tau(r)}\dir{R}{\tht}{+}$. There is a dense subset of $(r_0,t)$ consisting of points $r$ for which $\tau(r)$ is of type \eqref{prop:islandconfigs.ciii} and $\geo\from{\tau(r)}\dir{M}{\tht}{+}$ again first coalesces with $\geo\from{\tau(r)}\dir{R}{\tht}{+}$. 
     On the other hand, the set of points $r\in[r_0,t)$ for which $\tau(r)$ is of type \eqref{prop:islandconfigs.ciii} and $\geo\from{\tau(r)}\dir{M}{\tht}{+}$ first coalesces with $\geo\from{\tau(r)}\dir{L}{\tht}{+}$ contains $r_0$, consists of isolated points, and accumulates at $t$.
    \end{enumerate}
    \end{enumerate}

    The analogous statements hold upon exchanging left and right and swapping $+$ and $-$. See the fourth row in Figure \ref{fig:geodesics}.
\end{prop}

Lemma \ref{shocksdenseondust} asserts that, for any $(x,s)\in\IG\tht$, both $\tht-$ and $\tht+$ shocks are dense in $\Ipath_{(x,s)}^{\tht-}(\R)\cap\Ipath_{(x,s)}^{\tht+}(\R)$. This sharpens the density claim in Proposition \ref{prop:geods-dust} by showing that $\tht-$ and $\tht+$ single shocks (configurations \eqref{prop:geods-dust.b} and \eqref{prop:geods-dust.c}) are dense along the dust portions of the competition interfaces. In addition, Proposition \ref{prop:islandconfigs}\eqref{prop:islandconfigs.c} establishes the same density on the portions of these interfaces that form the boundaries of stability islands. Consequently, this density holds along the entire interface $\Ipath_{(x,s)}^{\tht\sig}(\R)$ for both $\sigg\in\{-,+\}$.\smallskip

The last main result of the section complements the above proposition and says that islands are dense in the instability graph.

\begin{prop}\label{prop:islandsdense}
  Let $\w\in\Omega_0$, $\tht\in\baddir$, $(x,s)\in\IG\tht$, and $\delta>0$. Then there exists a stability island that is entirely contained in the ball with center $(x,s)$ and radius $\delta$.  
\end{prop}

The remainder of this section is devoted to the proofs of the above propositions. We need the following lemma for proving Proposition \ref{prop:geodconfigstable}.

\begin{lem}\label{aux:1651}
Let $\w \in \Omega_0$, $\tht \in \R$, $\sigg\in\{-,+\}$, and $(x,s)\in\NU_1^{\tht\sig}$. Let $t'=s+\age^{\tht\sig}(x,s)$. Then $\geo\from{(x,s)}\dir{M}{\tht}{\sig}$ coalesces with one of $\geo\from{(x,s)}\dir{L}{\tht}{\sig}$ and $\geo\from{(x,s)}\dir{R}{\tht}{\sig}$ strictly before time $t'$: there exist $s'\in[s,t')$ and $S\in\{L,R\}$ such that $\geo\from{(x,s)}\dir{L}{\tht}{\sig}\big|_{(s,s')}\prec\geo\from{(x,s)}\dir{M}{\tht}{\sig}\big|_{(s,s')}\prec\geo\from{(x,s)}\dir{R}{\tht}{\sig}\big|_{(s,s')}$ and  $\geo\from{(x,s)}\dir{M}{\tht}{\sig}\big|_{[s',\infty)}=\geo\from{(x,s)}\dir{S}{\tht}{\sig}\big|_{[s',\infty)}$.
\end{lem}

\begin{proof}
We proceed by contradiction. Suppose $\geo\from{(x,s)}\dir{M}{\tht}{\sig}$ remains distinct from $\geo\from{(x,s)}\dir{L}{\tht}{\sig}$ and $\geo\from{(x,s)}\dir{R}{\tht}{\sig}$ over the entire interval $(s,t')$. 
By \eqref{age} and the geodesics ordering \eqref{geo:mono}, all three geodesics match from time $t'$ onward. 
Take some $t''>t'$ and observe that $\geo\from{(x,s)}\dir{S}{\tht}{\sig}\big|_{[s,t'']}$, $S\in\{L,M,R\}$, are all geodesics between $(x,s)$ and $\geo\from{(x,s)}\dir{L}{\tht}{\sig}(t'')$ and that all three are distinct on $(s,t')$. This is prohibited by \eqref{Duncan-deg3}. 
\end{proof}

\begin{proof}[Proof of Proposition \ref{prop:geodconfigstable}]
Part \eqref{prop:stablegeo_a} is in \eqref{geo:no-lollipop}.
By \eqref{ratnoshock}, this configuration occurs for all starting points in $\Q^2$, which is dense in $\R^2$.\smallskip

Now Suppose $(x,s)\in\NU_1^{\tht}$. Then, \eqref{age} gives \eqref{aux2040} with $t = s+\age^{\tht}(x,s)$.
By \eqref{allgeo}, $\geo\from{(x,s)}\dir{M}{\tht}{}$ is the only remaining possible $\tht$-directed geodesic emanating from $(x,s)$. By the ordering in \eqref{geo:mono}, this geodesic must remain weakly between $\geo\from{(x,s)}\dir{L}{\tht}{}$ and $\geo\from{(x,s)}\dir{R}{\tht}{}$ and coalesce with both by time $t$. By Lemma \ref{aux:1651}, it must coalesce first with exactly one of these two geodesics at some time $s' \in [s,t)$.
If $s'=s$, then  $\geo\from{(x,s)}\dir{L}{\tht}{}$ and $\geo\from{(x,s)}\dir{R}{\tht}{}$ are the only two geodesics out of $(x,s)$ and we have the configuration in \eqref{prop:stablegeo_bi}. 
Taking $s\in\Q$, \eqref{NUcountable} says there are infinitely many $x\in\R$ such that $(x,s)\in\NU_1^\tht$ and Lemma \ref{lm:Duncan}\eqref{Duncan.b} says none of these points have a middle $W^\tht$-geodesic. By Lemma \ref{lem:shocksdense}, the set of these points $x$ is dense in $\R$. Thus, the configuration in \eqref{prop:stablegeo_bi} occurs for a dense set of starting points in $\R^2$.\smallskip 

If $s'\in (s,t)$, then we are in exactly one of the two cases described in \eqref{prop:stablegeo_bii}. We  show that both configurations occur infinitely often. See the right panel of Figure \ref{fig:pf-geods-stability} for an illustration.

Take $s\in\Q$. As we have just shown, there exist infinitely many $(x,s)\in\NU_1^\tht$ at which the configuration in \eqref{prop:stablegeo_bi} occurs.  Take any such point and let $t = s+\age^{\tht}(x,s)$ and  $p' = \geo\from{(x,s)}\dir{L}{\tht}{}(t) = \geo\from{(x,s)}\dir{R}{\tht}{}(t)$. Then, by Lemma \ref{lm:geocoal}, there exists a $\tht$ shock interface $\tau$ proceeding down from $p'$, strictly between the two geodesics on the time interval $(s,t)$. By the continuity of the paths, $\tau$ must go through $(x,s)$. Take $r'\in(s,t)$ and $z'$ such that $(z',r')$ is strictly between $\geo\from{(x,s)}\dir{L}{\tht}{}(r')$ and $\tau(r')$. Then by \eqref{no-intersection} and \eqref{Itree}, $\Upsilon\from{(z',r')}\dir{L}{\tht}{}$ must coalesce with $\tau$ at a time $u'\in[s,r')$. 
By  Lemma \ref{lm:intcoal},  $\geo\from{\tau(u')}\dir{M}{\tht}{}$ goes strictly between $\Upsilon\from{(z',r')}\dir{L}{\tht}{}$ and $\tau$ and is  initially distinct from $\geo\from{\tau(u')}\dir{L}{\tht}{}$ and $\geo\from{\tau(u')}\dir{R}{\tht}{}$. 
Since there are exactly two $\tht$-directed geodesics out of $(x,s)$, it must be that $u'\in(s,r')$.  
By \eqref{no-intersection}, $\geo\from{\tau(u')}\dir{M}{\tht}{}$ cannot cross $\tau$ and thus must coalesce with $\geo\from{(x,s)}\dir{L}{\tht}{}$, and hence also with $\geo\from{\tau(u')}\dir{L}{\tht}{}$, at a time $s'<t$ before coalescing with $\geo\from{(x,s)}\dir{R}{\tht}{}$, and hence with $\geo\from{\tau(u')}\dir{R}{\tht}{}$, at $p'$. This gives one of the two configurations in \eqref{prop:stablegeo_bii}. Starting with $z'$ such that $(z',r')$ is strictly between $\tau(r')$ and $\geo\from{(x,s)}\dir{R}{\tht}{}(r')$ yields the other configuration in \eqref{prop:stablegeo_bii}. Since both configurations are contained between the two geodesics out of $(x,s)$, each distinct $(x,s)$ gives two distinct such configurations. 
Since we have shown that the configuration in \eqref{prop:stablegeo_bi} occurs at a dense set of points $(x,s)$ and since we can take $r'$ arbitrarily close to $s$, which makes $\tau(u')$ arbitrarily close to $(x,s)$, we get that each of the two configurations in \eqref{prop:stablegeo_bii} occurs at a dense set of starting points in $\R^2$.
\end{proof}

To prove Proposition \ref{prop:geods-stability}, we need two more preliminary lemmas.

\begin{lem}\label{aux:1634}
Let $\w \in \Omega_0$, $\tht \in \baddir$, and $(x,s)\in\R^2$. Suppose $\geo\from{(x,s)}\dir{L}{\tht}{-}\big|_{[a,b]}=\geo\from{(x,s)}\dir{R}{\tht}{+}\big|_{[a,b]}$, for some $b>a\ge s$. Then there exists $t\ge b$ such that 
\[\geo\from{(x,s)}\dir{L}{\tht}{-}\big|_{[a,t]}=\geo\from{(x,s)}\dir{R}{\tht}{+}\big|_{[a,t]}\quad\text{and}\quad
\geo\from{(x,s)}\dir{L}{\tht}{-}(t)=\geo\from{(x,s)}\dir{R}{\tht}{+}(t)\in\IGpns\tht.\]
\end{lem}

\begin{proof}
     By \eqref{geo:mono} and \eqref{sign}, we must have $\geo\from{(x,s)}\dir{L}{\tht}{-}(r)<\geo\from{(x,s)}\dir{R}{\tht}{+}(r)$ for all $r$ sufficiently large. Hence, the two geodesics must eventually separate at some time $t\ge b$. By \eqref{no bubble}, once separated after their common segment, the two geodesics cannot reintersect. Thus, $\geo\from{(x,s)}\dir{L}{\tht}{-}\big|_{[a,t]}=\geo\from{(x,s)}\dir{R}{\tht}{+}\big|_{[a,t]}$ and $\geo\from{(x,s)}\dir{L}{\tht}{-}\big|_{(t,\infty)}\prec\geo\from{(x,s)}\dir{R}{\tht}{+}\big|_{(t,\infty)}$. Furthermore, by monotonicity \eqref{geo:mono}, we have that that for $\sigg\in\{-,+\}$, $\geo\from{(x,s)}\dir{L}{\tht}{\sig}\big|_{[a,t]}=\geo\from{(x,s)}\dir{M}{\tht}{\sig}\big|_{[a,t]}=\geo\from{(x,s)}\dir{R}{\tht}{\sig}\big|_{[a,t]}$. Then coalescence \eqref{geo:coal1} implies that this holds on the whole time interval $[a,\infty)$. To see the last claim, let $p=\geo\from{(x,s)}\dir{L}{\tht}{-}(t)=\geo\from{(x,s)}\dir{R}{\tht}{+}(t)$. Using \eqref{geo:restart}, 
     \[\geo\from{p}\dir{L}{\tht}{-}\big|_{(t,\infty)}=\geo\from{p}\dir{R}{\tht}{-}\big|_{(t,\infty)}=\geo\from{(x,s)}\dir{L}{\tht}{-}\big|_{(t,\infty)}\prec\geo\from{(x,s)}\dir{R}{\tht}{+}\big|_{(t,\infty)}=\geo\from{p}\dir{L}{\tht}{+}\big|_{(t,\infty)}=\geo\from{p}\dir{R}{\tht}{+}\big|_{(t,\infty)}.\] 
     By Lemma \ref{lm:pns}, $p\in\IGpns\tht$. 
\end{proof}

The next lemma documents a simple extension fact, for easier reference.

\begin{lem}\label{lem:movingonshocks}
    Let $\w\in\Omega_0$, $\tht \in \R$, $\sigg\in\{-,+\}$, and $s<t$ in $\R$. Let $\tau:[s,t]\to\R^2$ be a continuous space-time path such that for all $r\in(s,t)$, $\tau(r)\in\NU_1^{\tht\sig}$. Then one can extend $\tau$ to a $\tht\sigg$ shock interface on the time interval $(-\infty,t]$.
\end{lem}

\begin{proof}
  Extend $\tau$ by letting $\tau(r)=\Upsilon_{\tau(s)}\dir{L}{\tht}{\sig}(r)$ for $r<s$. 
\end{proof}

\begin{proof}[Proof of Proposition \ref{prop:geods-stability}]
    First, suppose that $(x,s)\not\in\NU_1^{\tht-}\cup\NU_1^{\tht+}$. 
    Then, by \eqref{geo:no-lollipop}, for both $\sigg\in\{-,+\}$, $\geo\from{(x,s)}\dir{L}{\tht}{\sig}=\geo\from{(x,s)}\dir{R}{\tht}{\sig}$. By Lemma \ref{def:geoinstabpt}, $(x,s)\not\in\IG\tht$ implies that $\geo\from{(x,s)}\dir{L}{\tht}{-}$ and $\geo\from{(x,s)}\dir{R}{\tht}{+}$ must intersect at some time $b>s$. Then the extremality \eqref{p2pleftmost}-\eqref{p2prightmost} forces $\geo\from{(x,s)}\dir{S}{\tht}{-}$ and $\geo\from{(x,s)}\dir{S}{\tht}{+}$ to match on the time interval $[s,b]$, for both $S\in\{L,R\}$. But now we have $\geo\from{(x,s)}\dir{L}{\tht}{-}(r)=\geo\from{(x,s)}\dir{R}{\tht}{+}(r)$ for all $r\in[s,b]$. Applying Lemma \ref{aux:1634} with $a=s$ gives the claim in part \eqref{prop:geods-stability.a}. 
    By \eqref{ratnoshock} and \eqref{QnotIG}, $\Q^2\subset\R^2\setminus(\IG\tht\cup\NU_1^{\tht-}\cup\NU_2^{\tht+})$. Thus, this configuration occurs at a dense set of starting points.\smallskip

    Now suppose that $(x,s)\in\NU_1^{\tht\sig}$ for some $\sigg\in\{+,-\}$. Then by Lemma \ref{lem:stableshocks}, $(x,s)\in\NU_1^{\tht+}\cap\NU_1^{\tht-}$ and setting  $t'=s+\age^{\tht-}(x,s)=s+\age^{\tht+}(x,s)>s$ gives 
    \begin{align}\label{1649}
    \geo\from{(x,s)}\dir{L}{\tht}{-}\big|_{(s,t')} =\geo\from{(x,s)}\dir{L}{\tht}{+}\big|_{(s,t')}\prec\geo\from{(x,s)}\dir{R}{\tht}{-}\big|_{(s,t')}=\geo\from{(x,s)}\dir{R}{\tht}{+}\big|_{(s,t')}.
    \end{align}
    By the coalescence \eqref{geo:coal1}, $\geo\from{(x,s)}\dir{L}{\tht}{\sig}\big|_{[t',\infty)} =\geo\from{(x,s)}\dir{R}{\tht}{\sig}\big|_{[t',\infty)}$ for both $\sigg\in\{-,+\}$. 
    
   Since the double shock resolves at time $t'$, we have $\geo\from{(x,s)}\dir{L}{\tht}{-}(t')=\geo\from{(x,s)}\dir{R}{\tht}{+}(t')$. By Lemma \ref{lem:stableshocks}, this point is not in $\IG\tht$. Then \eqref{geo:restart} and Lemma \ref{def:geoinstabpt} imply that $\geo\from{(x,s)}\dir{L}{\tht}{-}$ and $\geo\from{(x,s)}\dir{R}{\tht}{+}$ must re-intersect at a time $b>t'$. By the extremality \eqref{p2pleftmost}-\eqref{p2prightmost}, we must  have that for both $S\in\{L,R\}$, $\geo\from{(x,s)}\dir{S}{\tht}{-}$ and $\geo\from{(x,s)}\dir{S}{\tht}{+}$ match over the interval $[t',b]$. Thus, all four geodesics proceed together over this time interval. Applying Lemma \ref{aux:1634} with $a=t'$ produces $t\ge b>t'$ such that  $\geo\from{(x,s)}\dir{L}{\tht}{-}\big|_{[t',t]}=\geo\from{(x,s)}\dir{R}{\tht}{+}\big|_{[t',t]}$ and $\geo\from{(x,s)}\dir{L}{\tht}{-}(t)=\geo\from{(x,s)}\dir{R}{\tht}{+}(t)\in\IGpns\tht$. We next sort out the situation with the middle geodesics $\geo\from{(x,s)}\dir{M}{\tht}{\sig}$.\smallskip

By \eqref{allgeo}, the only remaining $\tht$-directed geodesics are $\geo\from{(x,s)}\dir{M}{\tht}{\sig}$, with $\sigg\in\{-,+\}$. By monotonicity \eqref{geo:mono} and coalescence \eqref{geo:coal1}, for $\geo\from{(x,s)}\dir{M}{\tht}{\sig}$ to be distinct from $\geo\from{(x,s)}\dir{S}{\tht}{\sig}$, $S\in\{L,R\}$, it must initially lie strictly between the two.
If this does not occur for either sign $\sigg\in\{-,+\}$, then the configuration in part \eqref{prop:geods-stability.bi} holds. We therefore assume that at least one of the middle geodesics initially lies strictly between the left and right geodesics, and show that this leads to one of the two (symmetric) configurations described in part \eqref{prop:geods-stability.bii}.\smallskip

Without loss of generality, assume 
$\geo\from{(x,s)}\dir{M}{\tht}{-}$ is distinct from $\geo\from{(x,s)}\dir{S}{\tht}{-}$, $S\in\{L,R\}$, the other case being analogous.  
By Lemma \ref{aux:1651}, $\geo\from{(x,s)}\dir{M}{\tht}{-}$ must coalesce with either $\geo\from{(x,s)}\dir{L}{\tht}{-}$ or $\geo\from{(x,s)}\dir{R}{\tht}{-}$ at some time $s'\in[s,t')$. Since the middle geodesic is distinct from the other two, we necessarily have $s'>s$. We show that $\geo\from{(x,s)}\dir{M}{\tht}{+}\big|_{[s,s']}=
\geo\from{(x,s)}\dir{M}{\tht}{-}\big|_{[s,s']}$.
See the left panel of Figure \ref{fig:pf-geods-stability} for an illustration.

\begin{figure}[hpt]
    \includegraphics[width=3.25cm]{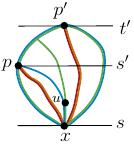}\qquad\qquad
    \includegraphics[width=3.25cm]{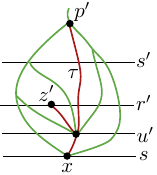}
    \caption{\small Left: The interfaces $\Upsilon_{p}\dir{M}{\tht}{-}$ and $\Upsilon_{p'}\dir{M}{\tht}{-}$ are separated by $\geo\from{(x,s)}\dir{M}{\tht}{-}$ and coalesce at $(x,s)$. Right: When $(x,s)$ has the configuration in Proposition \ref{prop:geodconfigstable}\eqref{prop:stablegeo_bi} or Proposition \ref{prop:geods-stability}\eqref{prop:geods-stability.bi}, the point $\tau(u')$ has the configuration in Proposition \ref{prop:geodconfigstable}\eqref{prop:stablegeo_bii} or Proposition \ref{prop:geods-stability}\eqref{prop:geods-stability.bii}, respectively.}
    \label{fig:pf-geods-stability}
\end{figure}

Let $p=\geo\from{(x,s)}\dir{M}{\tht}{-}(s')$
and $p'=\geo\from{(x,s)}\dir{M}{\tht}{-}(t')$.
Since $\geo\from{(x,s)}\dir{R}{\tht}{-}$ and 
$\geo\from{(x,s)}\dir{M}{\tht}{-}$ coalesce at $p'$, Lemma \ref{lm:geocoal} gives a $\tht-$ shock interface out of $p'$ that remains strictly between the two geodesics on the interval $(s,t')$. Denote this interface by $\Upsilon_{p'}\dir{M}{\tht}{-}$. Similarly, there exists an interface we denote by $\Upsilon_{p}\dir{M}{\tht}{-}$ that 
 remains strictly between 
$\geo\from{(x,s)}\dir{M}{\tht}{-}$ and 
$\geo\from{(x,s)}\dir{L}{\tht}{-}$ on the interval $(s,s')$.  
Together with continuity of the interfaces, this implies that 
$\Upsilon_{p'}\dir{M}{\tht}{-}$ and $\Upsilon_{p}\dir{M}{\tht}{-}$ first 
intersect at $(x,s)$.

For any $r\in(s,t')$ and any $z$ such that $(z,r)$ lies strictly between 
$\geo\from{(x,s)}\dir{L}{\tht}{-}$ and 
$\geo\from{(x,s)}\dir{R}{\tht}{+}$, both 
$\geo\from{(z,r)}\dir{L}{\tht}{-}$ and 
$\geo\from{(z,r)}\dir{R}{\tht}{+}$ are forced to pass through $p'$.  
Hence, by Lemma \ref{def:geoinstabpt}, the entire region strictly between $\geo\from{(x,s)}\dir{L}{\tht}{-}([s,t'])$ and 
$\geo\from{(x,s)}\dir{R}{\tht}{+}([s,t'])$ lies in 
$\R^2\setminus\IG\tht$.  
By Lemma \ref{lem:stableshocks}, any $\tht-$ shock in this region must also be a 
$\tht+$ shock.   
This and Lemma \ref{lem:movingonshocks} give that $\Upsilon_{p}\dir{M}{\tht}{-}\bigl|_{(s,s')}$ and $\Upsilon_{p'}\dir{M}{\tht}{-}\bigl|_{(s,t')}$ are both on $\tht+$ shock interfaces.  
Since they meet for the first time at $(x,s)$, this point
is a coalescence point of two $\tht+$ shock interfaces and, by 
Lemma \ref{lm:intcoal}, must be a trifurcation point of $\tht+$ geodesics. This means that $\geo\from{(x,s)}\dir{M}{\tht}{+}$ is initially separated from both $\geo\from{(x,s)}\dir{L}{\tht}{+}$ and $\geo\from{(x,s)}\dir{R}{\tht}{+}$.
Then \eqref{no4stars} implies the existence of $\varepsilon>0$ such that
$\geo\from{(x,s)}\dir{M}{\tht}{+}\big|_{[s,s+\e]}=
 \geo\from{(x,s)}\dir{M}{\tht}{-}\big|_{[s,s+\e]}$.  
We now prove that these two geodesics coincide on all of $[s,s']$.
Suppose instead that $\geo\from{(x,s)}\dir{M}{\tht}{-}$ and
$\geo\from{(x,s)}\dir{M}{\tht}{+}$ separate at $u\in\R^2$. 
Since $u$ is on the relative interior of both $\geo\from{(x,s)}\dir{M}{\tht}{-}$ and $\geo\from{(x,s)}\dir{M}{\tht}{+}$, by \eqref{geo:restart}, for each sign $\sigg\in\{-,+\}$, $\geo\from{(x,s)}\dir{S}{\tht}{\sig}$, $S\in\{L,M,R\}$ all match. Thus, from $u$, $\geo_u\dir{L}{\tht}{-}$ and $\geo_u\dir{L}{\tht}{+}$ immediately separate but later pass through $p'$, contradicting the extremality \eqref{p2pleftmost}. Consequently, $\geo\from{(x,s)}\dir{M}{\tht}{+}$ and $\geo\from{(x,s)}\dir{M}{\tht}{-}$ cannot separate before time $s'$. \smallskip

We now show that the configuration in \eqref{prop:geods-stability.bi} occurs for a dense set of starting points. Take any rational point $q=(z,r)\in\Q^2$. Then $q\in\island_q$ and, since $\island_q$ is open, there is an open neighborhood $O\subset\island_q$ containing $q$.
By continuity, there exists an $\ep>0$ such that $\Upsilon_{q}\dir{L}{\tht}{+}(r')\in O$ for all $r'\in[r-\ep,r]$. Then for any $r'\in[r-\ep,r)$, $\Upsilon_{q}\dir{L}{\tht}{+}(r')$ is in $\NU_1^{\tht+}\setminus\IG\tht$. Thus, the $\tht$-directed geodesics out of $\Upsilon_{q}\dir{L}{\tht}{+}(r')$ follow the configuration in either \eqref{prop:geods-stability.bi} or \eqref{prop:geods-stability.bii}. Take $r'\in[r-\ep,r)\cap\Q$. By Lemma \ref{lm:Duncan}\eqref{Duncan.b}, we have configuration \eqref{prop:geods-stability.bi} out of $\Upsilon_{q}\dir{L}{\tht}{+}(r')$.
This happens for a dense set of starting points because the set of rational points $q\in\Q^2$ is dense in $\R^2$ and we can take $r'\in\Q$ close enough to $r$ to make $\Upsilon_{q}\dir{L}{\tht}{+}(r')$ as close as desired to $q$.\smallskip

Finally, we show that each of the two configurations in \eqref{prop:geods-stability.bii} occurs for a dense set of starting points. The argument is identical to that in the final paragraph of the proof of Proposition \ref{prop:geodconfigstable}, which we reproduce here for the reader’s convenience. See the right panel of Figure \ref{fig:pf-geods-stability} for an illustration.

Take an $(x,s)$ for which configuration \eqref{prop:geods-stability.bi} occurs. Let $t' = s+\age^{\tht-}(x,s)=s+\age^{\tht+}(x,s)$ and  $p' = \geo\from{(x,s)}\dir{L}{\tht}{-}(t') = \geo\from{(x,s)}\dir{R}{\tht}{+}(t')$. Then, by Lemma \ref{lm:geocoal}, there exists a $\tht+$ shock interface $\tau$ proceeding down from $p'$, strictly between the geodesics $\geo\from{(x,s)}\dir{L}{\tht}{+}$ and $\geo\from{(x,s)}\dir{R}{\tht}{+}$ on the time interval $(s,t')$. By the continuity of the paths, $\tau$ must go through $(x,s)$. Take $r'\in(s,t')$ and $z'$ such that $(z',r')$ is strictly between $\geo\from{(x,s)}\dir{L}{\tht}{+}(r')$ and $\tau(r')$. Then by \eqref{no-intersection} and \eqref{Itree}, $\Upsilon\from{(z',r')}\dir{L}{\tht+}{}$ must coalesce with $\tau$ at a time $u'\in[s,r')$. 
By  Lemma \ref{lm:intcoal},  $\geo\from{\tau(u')}\dir{M}{\tht}{+}$ goes strictly between $\Upsilon\from{(z',r')}\dir{L}{\tht}{+}$ and $\tau$ and is  initially distinct from $\geo\from{\tau(u')}\dir{L}{\tht}{+}$ and $\geo\from{\tau(u')}\dir{R}{\tht}{+}$. 
Since there are exactly two $\tht$-directed geodesics out of $(x,s)$, it must be that $u'\in(s,r')$.  
By \eqref{no-intersection}, $\geo\from{\tau(u')}\dir{M}{\tht}{+}$ cannot cross $\tau$ and thus must coalesce with $\geo\from{(x,s)}\dir{L}{\tht}{+}$, and hence also with $\geo\from{\tau(u')}\dir{L}{\tht}{+}$, at a time $s'<t$ before coalescing with $\geo\from{(x,s)}\dir{R}{\tht}{+}$, and hence with $\geo\from{\tau(u')}\dir{R}{\tht}{+}$, at $p'$. This gives one of the two configurations in \eqref{prop:geods-stability.bii}. Starting with $z'$ such that $(z',r')$ is strictly between $\tau(r')$ and $\geo\from{(x,s)}\dir{R}{\tht}{+}(r')$ yields the other configuration in \eqref{prop:geods-stability.bii}. 
Since we have shown that the configuration in \eqref{prop:geods-stability.bi} occurs for a dense set of points $(x,s)$ and we can take $r'$ arbitrarily close to $s$, making $\tau(u')$ arbitrarily close to $(x,s)$, we see that each of the two configurations in \eqref{prop:geods-stability.bii} occurs for a dense set of starting points.
\end{proof}

\begin{proof}[Proof of Proposition \ref{prop:geods-dust}]
    Suppose $(x,s)\notin\NU_1^{\tht-}\cup\NU_1^{\tht+}$. Then by Definition \ref{def:pns}, $(x,s)\in\IGpns\tht$.
   To see that such configurations are dense in $\IG\tht$, fix a rational time level $s \in \Q$. By \eqref{Hausdorff}, for any rational $a<b$, the set $\IG\tht \cap ((a,b)\times \{s\})$ has Hausdorff dimension $1/2$, whereas by \eqref{NUcountable}, the set $(\NU_1^{\tht-} \cup \NU_1^{\tht+}) \cap (\R \times \{s\})$ is countable. It follows that there exist infinitely many instability points $(x,s)$, with $x\in(a,b)$, that are neither $\tht-$ nor $\tht+$ shock points. By Definition \ref{def:pns}, these points lie in $\IGpns\tht$.
    Although some of them may be island tips (Lemma \ref{lm:islandboundary}\eqref{lm:islandboundary.c}), Lemma \ref{lm:islands} shows that island tips form a countable set. Thus, removing these from $\IGpns\tht \cap ((a,b) \times \{s\})$ still leaves a set of Hausdorff dimension $1/2$, consisting of points of type \eqref{prop:geods-dust.a}. This argument shows that the set of points of type \eqref{prop:geods-dust.a} is dense in $\IG\tht$.\smallskip

    Now suppose that $(x,s)\in \NU_1^{\tht-}$. It cannot be neither left- nor right-isolated, else Lemmas \ref{lem:shockstoisland} and \ref{lm:isotoisland} would place it on the boundary of an island. Then, by Lemma \ref{LRisolgeo}\eqref{LRisolgeo.a}, $(x,s)$ cannot be a $\tht\sigg$ hugging shock for either $\sigg \in \{-,+\}$.
     By Lemma \ref{lm:hugging}(\ref{lm:hugging.b}-\ref{lm:hugging.c}) and the extremality \eqref{p2pleftmost}-\eqref{p2prightmost}, \begin{align}\label{aux2253}
     \geo\from{(x,s)}\dir{L}{\tht}{-}\big|_{(s,\infty)}\prec \geo\from{(x,s)}\dir{L}{\tht}{+}\big|_{(s,\infty)}
     \quad\text{and}\quad 
     \geo\from{(x,s)}\dir{R}{\tht}{-}\big|_{(s,\infty)}\prec \geo\from{(x,s)}\dir{R}{\tht}{+}\big|_{(s,\infty)}.
     \end{align}
     Thus, $\geo\from{(x,s)}\dir{L}{\tht}{-}$, $\geo\from{(x,s)}\dir{R}{\tht}{-}$, and $\geo\from{(x,s)}\dir{R}{\tht}{+}$ are initially distinct. Then, by \eqref{no4stars}, there is no fourth geodesic out of $(x,s)$ that is initially separate from the three geodesics. This and the coalescence \eqref{geo:coal1} imply that there is no distinct $W^{\tht-}$ middle geodesic.
     Furthermore, $\geo\from{(x,s)}\dir{L}{\tht}{+}$ cannot initially proceed with $\geo\from{(x,s)}\dir{R}{\tht}{-}$, else Definition \ref{def:snowbird} makes $(x,s)$ a snowbird shock and  Lemma \ref{lem:SBisisland} would place it on the boundary of an island. This, the first inequality in \eqref{aux2253}, \eqref{no4stars}, and the coalescence \eqref{geo:coal1} imply that $\geo\from{(x,s)}\dir{L}{\tht}{+}=\geo\from{(x,s)}\dir{R}{\tht}{+}$. In particular, there is no distinct $W^{\tht+}$-geodesic out of $(x,s)$. We are thus in the configuration in \eqref{prop:geods-dust.b} 
     (the second configuration on the second row in Figure \ref{fig:geodesics}). 
     The case when $(x,s)\in \NU_1^{\tht+}$ is analogous and leads to the configuration in \eqref{prop:geods-dust.c}.
    
It remains to show that the sets of points of type \eqref{prop:geods-dust.b} and \eqref{prop:geods-dust.c} are both dense in $\IG\tht$. We have already established that points of type \eqref{prop:geods-dust.a} are dense in $\IG\tht$, so fix such a point $(x,s)$. 
By Lemma \ref{lem:IGgodown}, the shock interfaces $\Upsilon_{(x,s)}^{L,\tht\sig}$, for $\sigg\in\{-,+\}$, remain within $\IG\tht$. 
Since the closures of distinct stability islands are disjoint (Lemma \ref{lm:disjointislands}) and there are countably many of them (Lemma \ref{lm:islands}), Lemma \ref{thm:sierpinski-interval} implies that for any $r<r'<s$, there exists $r''\in[r,r']$ such that $\Upsilon_{(x,s)}^{L,\tht\sig}(r'')$ is not in the closure of any stability island. Consequently, there exists a sequence $r_n$, strictly increasing to $s$, and such that, for each $n$, $\Upsilon_{(x,s)}^{L,\tht\sig}(r_n)$ is not contained in the closure of any stability island.
It follows that points of both types \eqref{prop:geods-dust.b} and \eqref{prop:geods-dust.c} occur arbitrarily close to $(x,s)$. Hence each of these types is dense in $\IG\tht$.    
\end{proof}

To prove Proposition \ref{prop:islandconfigs}, we need a number of lemmas.
The first of the lemmas says that the left $W^{\tht\pm}$ geodesics out of the right boundary of a stability island go together inside the island and must pass through its tip.  See Figure \ref{fig:isotoisland}.

\begin{lem}\label{aux:RbndryLgeo}
    Let $\omega\in\Omega_0$ and $\tht\in\baddir$. Suppose $\tau:[s,t]\to\R^2$ is the right boundary of a stability island. Then for any $r<r'$ in $(s,t)$, $\geo_{\tau(r)}\dir{L}{\tht}{-}(r') = \geo_{\tau(r)}\dir{L}{\tht}{+}(r') < \tau(r')$. Moreover, $\geo_{\tau(r)}\dir{L}{\tht}{-}(t) = \geo_{\tau(r)}\dir{L}{\tht}{+}(t) = \tau(t)$. The analogous statement holds upon exchanging left and right and replacing $<$ with $>$.
\end{lem}

\begin{proof}
 By Lemma \ref{lem:islandtoshocks}, any island is given by the region strictly between two misordered geodesics. Then, for any point on the right boundary of the island, Lemma \ref{lm:isotoisland} identifies the tip of the island as being exactly the point where the left $W^{\tht\pm}$-geodesics separate and, similarly, for points on the left boundary. It also says the geodesics run inside the island, strictly between its boundaries.  
\end{proof}

Recall from Lemma \ref{LRisolgeo}\eqref{LRisolgeo.b} that no instability point can simultaneously be both a $\tht-$ and a $\tht+$ hugging shock. The next lemma shows that there do exist instability points which are both $\tht-$ and $\tht+$ shocks, while being a $\tht\sigg$ hugging shock for exactly one value of $\sigg\in\{-,+\}$.

\begin{lem}\label{aux:-shocksalongRboundary}
    Let $\omega\in\Omega_0$ and $\tht\in\baddir$. Let $\tau:[s,t]\to\R$ be the right boundary of a stability island. Then for any $r<r'$ in $(s,t)$, there exists a time $s'\in[r,r')$ such that $\tau(s')\in\NU_1^{\tht-}$. The analogous statement holds when $\tau$ is the left boundary of a stability island and we replace $\NU_1^{\tht-}$ by $\NU_1^{\tht+}$.
\end{lem}

\begin{proof}
    We only prove the first statement, the second being symmetric.
   By Lemma \ref{aux:RbndryLgeo}, $\geo\from{\tau(r)}\dir{L}{\tht}{-}$ stays strictly left of $\tau$ on the time interval $(r,t)$. Take $z$ such that  $(z,r')$ is strictly between $\geo\from{\tau(r)}\dir{L}{\tht}{-}(r')$ and $\tau(r')$. By \eqref{no-intersection} and the continuity of the paths, $\Upsilon\from{(z,r')}\dir{L}{\tht}{-}$ must intersect $\tau$ at a time $s'\in[r,r')$. Then $\tau(s') \in \NU_1^{\tht-}$, as desired.
\end{proof}

The following lemma complements the previous one by showing that an island boundary cannot contain a nontrivial continuous segment composed of double shocks.

\begin{lem}\label{doubleisolareisol}
    Let $\omega\in\Omega_0$ and $\tht\in\baddir$. Let $\tau:[s,t]\to\R^2$ be the right boundary of a stability island. Then there does not exist a sub-interval $(r,r')\subset(s,t)$ such that $\tau(u) \in \NU_1^{\tht-}$ for all $u \in (r, r')$. The analogous statement holds when $\tau$ is a left boundary and we replace $\NU_1^{\tht-}$ with $\NU_1^{\tht+}$.
\end{lem}

\begin{figure}[hpt]
    \includegraphics[width=3.7cm]{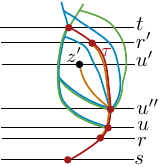}
    \caption{\small The proof of Lemma \ref{doubleisolareisol}. $\tau(u'')$ is forced to be a double hugging shock, which does not exist.}
    \label{fig:doubleisolareisol}
\end{figure}

\begin{proof}
    See Figure \ref{fig:doubleisolareisol} for an illustration of the proof.
    To argue by contradiction, suppose there exists an interval $(r,r')\subset(s,t)$ such that for any $u\in(r,r')$, $\tau(u)\in \NU_1^{\tht-}$. Then, by Lemma \ref{lem:movingonshocks}, $\tau|_{[r,r']}$ is part of a $\tht-$ shock interface out of $\tau(r')$.
    Take $u<u'$ in $(r,r')$. By Lemma \ref{aux:RbndryLgeo}, $\geo\from{\tau(u)}\dir{L}{\tht}{-}(u')<\tau(u')$. Take $z'$ such that $(z',u')$ is strictly between  these two points. By \eqref{no-intersection} and the continuity of the paths, $\Upsilon\from{(z',u')}\dir{L}{\tht}{-}$ must coalesce with $\tau$ at some time $u''\in[u,u')$. From this point, by Lemma \ref{aux:RbndryLgeo}, the left $W^{\tht-}$ and $W^{\tht+}$-geodesics must initially proceed together (all the way up to time $t$).  Moreover, by Lemma \ref{lm:intcoal}, the middle $W^{\tht-}$-geodesic out of $\tau(u'')$ must proceed strictly between $\Upsilon\from{(z',u')}\dir{L}{\tht}{-}$ and $\tau$, while by \eqref{shocksplitsgeo}, the right $W^{\tht-}$-geodesic must proceed initially strictly to the right of $\tau$. But by Lemma \ref{lm:islandboundary}\eqref{lm:islandboundary.a}, $\tau$ is on a $\tht+$ shock interface and, since $\tau(u'')$ is in its relative interior,  \eqref{shocksplitsgeo} implies that $\tau|_{(u'',t)} \prec \geo\from{\tau(u'')}\dir{R}{\tht}{+}\big|_{(u'',t)}$. But by \eqref{no4stars}, this means that $\geo\from{\tau(u'')}\dir{R}{\tht}{-}$ and $\geo\from{\tau(u'')}\dir{R}{\tht}{+}$ must initially proceed together, forming a double hugging shock at $\tau(u'')$ and contradicting Lemma \ref{LRisolgeo}\eqref{LRisolgeo.b}. This contradiction proves the lemma. 
\end{proof}

The next lemma describes the behavior of the rightmost $W^{\theta-}$-geodesic emanating from a point on the right boundary of a stability island and, symmetrically, the leftmost $W^{\theta+}$-geodesic emanating from a point on the left boundary.

\begin{lem}\label{aux:R-goin}
    Let $\w\in\Omega_0$ and $\tht\in\baddir$. Let $\tau:[s,t]\to\R^2$ be the right boundary of a stability island. Then, for every $s'<t'$ in $(s,t)$, $\geo\from{\tau(s')}\dir{R}{\tht}{-}(t') < \tau(t') $. Moreover, if $\geo\from{\tau(s')}\dir{R}{\tht}{-}$ and $\geo\from{\tau(s')}\dir{L}{\tht}{-}$ are initially distinct, then $\geo\from{\tau(s')}\dir{M}{\tht}{+}(t') = \geo\from{\tau(s')}\dir{R}{\tht}{-}(t')$.
\end{lem}

\begin{figure}[hpt]
    \includegraphics[width=3.25cm]{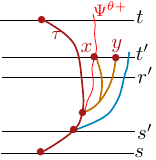}\qquad\qquad
    \includegraphics[width=3.25cm]{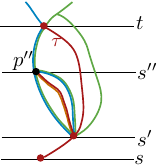}
    \caption{\small The proof of Lemma \ref{aux:R-goin}.}
    \label{fig:R-goin}
\end{figure}

\begin{proof}
First, observe that if $\geo\from{\tau(s')}\dir{R}{\tht}{-}$ and $\geo\from{\tau(s')}\dir{L}{\tht}{-}$ initially proceed together, then by coalescence \eqref{geo:coal1}, they must remain together, and thus by Lemma \ref{aux:RbndryLgeo}, it follows that $\geo\from{\tau(s')}\dir{R}{\tht}{-}(t') < \tau(t') $. Going forward, assume that $\geo\from{\tau(s')}\dir{R}{\tht}{-}$ and $\geo\from{\tau(s')}\dir{L}{\tht}{-}$ are initially distinct. 

Suppose there exists a $t'\in(s',t)$, such that $\geo\from{\tau(s')}\dir{R}{\tht}{-}(t') > \tau(t') $. See the left panel of Figure \ref{fig:R-goin}. Then, by \eqref{geo:restart} and the continuity of the paths, there exists a time $u'\in[s',t')$  such that $\geo\from{\tau(s')}\dir{R}{\tht}{-}(u') = \tau(u') $ and, for all $r\in(u',t']$, $\geo\from{\tau(u')}\dir{R}{\tht}{-}(r) > \tau(r) $. Thus, to simplify notation, we assume, without loss of generality, that $s' = u'$, i.e.\ out of $\tau(s')$, the right $W^{\tht-}$-geodesic proceeds strictly to the right of $\tau$ until at least time $t'$. Take a time level $r'\in(s',t')$. By Lemmas \ref{no-isolated}, \ref{LRisolgeo}\eqref{LRisolgeo.a}, \ref{lem:islandtoshocks}, and \ref{lm:islandboundary}\eqref{lm:islandboundary.a},  
$\tau(r')$ is not right-isolated. Therefore, there exists a $y$ such that 
$(y,r')\in\IG\tht$ and is strictly between $\tau(r')$ and $\geo\from{\tau(s')}\dir{R}{\tht}{-}(r')$. 
Then, by Lemma \ref{lm:isodense}, there exists an $x$ such that $(x,r')\in\IG\tht$, it is strictly between $\tau(r')$ and $(y,r')$, and is not right-isolated.
By \eqref{no-intersection} and the continuity of the paths, $\Upsilon\from{(y,r')}\dir{R}{\tht}{-}$ cannot intersect $\geo\from{\tau(s')}\dir{R}{\tht}{-}$ and must intersect $\tau$ at or after time $s'$.
By Proposition \ref{prop:IGgoup}\eqref{IGgoup.i}, $\Ipath^{\tht+}_{(x,r')}\preceq \Upsilon\from{(x,r')}\dir{R}{\tht}{-}$, and by the coalescence \eqref{Itree}, $\Upsilon\from{(x,r')}\dir{R}{\tht}{-}\preceq \Upsilon\from{(y,r')}\dir{R}{\tht}{-}$. 
Thus, $\Ipath^{\tht+}_{(x,r')}$ must intersect $\tau$. However, Lemma \ref{lm:bdry}\eqref{lm:bdry.b} says that $\tau$ is part of the interface $\Ipath^{\tht+}_{\tau(s')}$ and Proposition \ref{prop:IGgoup}\eqref{IGgoup.e} says that $\Ipath^{\tht+}_{(x,r')}(r')=(x,r')$. Thus, since  $\Ipath^{\tht+}_{\tau(s')}(r')=\tau(r')<(x,r')$, Proposition \ref{prop:IGgoup}\eqref{IGgoup.g} prohibits $\Ipath^{\tht+}_{(x,r')}$ from intersecting $\tau$. This contradiction shows that there cannot exist a $t'\in(s',t)$, such that $\geo\from{\tau(s')}\dir{R}{\tht}{-}(t') > \tau(t') $. In other words, on the interval $(s',t)$, $\geo\from{\tau(s')}\dir{R}{\tht}{-}$ has to proceed weakly left of $\tau$.

Now suppose that there exists a $t'\in(s',t)$ such that $\geo\from{\tau(s')}\dir{R}{\tht}{-}(t') = \tau(t')$. 
By Lemma \ref{aux:RbndryLgeo}, for any time $r\in(s,t)$, $\geo\from{\tau(r)}\dir{L}{\tht}{-}$ and $\geo\from{\tau(r)}\dir{L}{\tht}{+}$ are identical until time $t$ and proceed strictly to the left of $\tau$. 
Thus, if $\geo\from{\tau(s')}\dir{R}{\tht}{-}\big|_{(s',t')}=\tau|_{(s',t')}$, then all the points on this path would be $\tht-$ shocks, which contradicts Lemma \ref{doubleisolareisol}. Consequently, there exists an $r\in(s',t')$ such that $\geo\from{\tau(s')}\dir{R}{\tht}{-}(r)\ne\tau(r)$. Since we have shown that the geodesic remains weakly left of $\tau$, we in fact have $\geo\from{\tau(s')}\dir{R}{\tht}{-}(r)<\tau(r)$.  But then, by the coalescence \eqref{geo:coal1} and Lemma \ref{aux:RbndryLgeo}, $\geo\from{\tau(r)}\dir{L}{\tht}{-}$ remains weakly right of  $\geo\from{\tau(s')}\dir{R}{\tht}{-}$ and strictly left of $\tau$, on the time interval $(r,t)$. This is in contradiction with having $\geo\from{\tau(s')}\dir{R}{\tht}{-}(t')=\tau(t')$. 
Thus it must be that for any $t'\in(s',t)$,  $\geo\from{\tau(s')}\dir{R}{\tht}{-}(t') < \tau(t')$.

 It remains to show that for any $t'\in(s',t)$, $\geo\from{\tau(s')}\dir{R}{\tht}{-}(t') = \geo\from{\tau(s')}\dir{M}{\tht}{+}(t')$. See the left panel in Figure \ref{fig:R-goin} for an illustration.
 We first prove that the two geodesics proceed initially together. 

By Lemma \ref{aux:RbndryLgeo}, $\geo\from{\tau(s')}\dir{L}{\tht}{-}$ remains strictly left of $\tau$ until it passes through $\tau(t)$.  By what we proved above $\geo\from{\tau(s')}\dir{R}{\tht}{-}$ remains strictly left of $\tau$ on the time interval $(s',t)$. This, the monotonicity \eqref{geo:mono}, and the fact that $\geo\from{\tau(s')}\dir{R}{\tht}{-}$ and $\geo\from{\tau(s')}\dir{L}{\tht}{-}$ are initially distinct give us that the two geodesics must coalesce at some time $s''\in(s',t]$. 
Denote the coalescence point by $p''$. By Lemma \ref{lm:geocoal}, there is a $\tht-$ shock interface $\gamma$ out of $p''$ that runs strictly between the two geodesics and hence goes through $\tau(s')$ and, in fact, intersects $\tau$ at that point for the first time. Since both geodesics are inside the island, so is $\gamma|_{(s',s'')}$. Then by Lemma \ref{lem:stableshocks}, $\gamma_{(s',s'')}\subset\NU_1^{\tht+}$ and, by Lemma \ref{lem:movingonshocks}, $\gamma$ must be part of a $\tht+$ shock interface out of $p''$.
 By Lemma \ref{lm:islandboundary}\eqref{lm:islandboundary.a}, $\tau\subset\Upsilon\from{\tau(t)}\dir{R}{\tht}{+}$. Thus, by \eqref{Itree}, $\tau(s')$ is a coalescence point of two $\tht+$ shock interfaces. By Lemma \ref{lm:intcoal}, $\geo\from{\tau(s')}\dir{M}{\tht}{+}$ goes strictly between $\gamma$ and $\tau$, on the time interval $(s',s'')$, and is initially distinct from $\geo\from{\tau(s')}\dir{L}{\tht}{+}$ and $\geo\from{\tau(s')}\dir{R}{\tht}{+}$. By  \eqref{shocksplitsgeo}, $\geo\from{\tau(s')}\dir{R}{\tht}{+}$ goes strictly right of $\tau$, which we know is itself strictly right of $\geo\from{\tau(s')}\dir{R}{\tht}{-}$, which in turn is initially strictly right of $\geo\from{\tau(s')}\dir{L}{\tht}{-}$, which goes strictly left of $\gamma$. Thus, by \eqref{no4stars}, $\geo\from{\tau(s')}\dir{R}{\tht}{-}$ and $\geo\from{\tau(s')}\dir{M}{\tht}{+}$ must initially proceed together.

Take $r\in(s',t)$ such that $\geo\from{\tau(s')}\dir{R}{\tht}{-}(r)=\geo\from{\tau(s')}\dir{M}{\tht}{+}(r)$. Denote this point by $q$. Such a time exists because we have just shown that the two geodesics proceed together initially.
By \eqref{geo:restart}, $\geo\from{\tau(s')}\dir{M}{\tht}{+}\big|_{[r,\infty)}=\geo\from{q}\dir{R}{\tht}{+}$. Since both this geodesic and $\geo\from{\tau(s')}\dir{R}{\tht}{-}$ go through $q$ and $\tau(t)$ (because they are both trapped between $\geo\from{\tau(s')}\dir{L}{\tht}{+}\big|_{[s',t]}=\geo\from{\tau(s')}\dir{L}{\tht}{-}\big|_{[s',t]}$ and $\tau|_{[s',t]}$), the extremality \eqref{p2prightmost} implies they must match between times $r$ and $t$. Thus, $\geo\from{\tau(s')}\dir{M}{\tht}{+}$ and $\geo\from{\tau(s')}\dir{R}{\tht}{-}$ cannot separate before time $t$ and we are done.
\end{proof}

The last lemma we need shows the existence of single shock points on the boundary of a stability island. At these points, the shock point has a middle geodesic that is distinct from the other two. See Figure \ref{fig:shockconfluenceexist}. 

\begin{lem}\label{aux:shockconfluenceexist}
    Let $\w\in\Omega_0$ and  $\tht\in\baddir$. Let $\tau:[s,t]\to\R^2$ be the right boundary of a stability island. Then, for any $r<r'$ in $(s,t)$, there exists a time $s'\in[r,r')$ such that $\tau(s')\notin\NU_1^{\tht-}$ but $\geo\from{\tau(s')}\dir{M}{\tht}{+}$ and $\geo\from{\tau(s')}\dir{R}{\tht}{+}$ are initially distinct and $\tau|_{(s',t)}\prec\geo\from{\tau(s')}\dir{M}{\tht}{+} \preceq \geo\from{\tau(s')}\dir{R}{\tht}{+}$. The analogous statement holds upon exchanging left and right and swapping $-$ and $+$.
\end{lem}

\begin{figure}[hpt]
    \includegraphics[width=3.25cm]{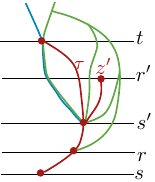}
    \caption{\small The proof of Lemma \ref{aux:shockconfluenceexist}.}
    \label{fig:shockconfluenceexist}
\end{figure}

\begin{proof}
  By Lemma \ref{lem:islandtoshocks}, $\tau \subset \Upsilon\from{\tau(t)}\dir{R}{\tht}{+}$. Thus, by \eqref{shocksplitsgeo}, $\geo\from{\tau(r)}\dir{R}{\tht}{+}$ goes strictly right of $\tau$ on $(r,t)$. Take $z'$ such that $(z',r')$ is strictly between $\tau(r')$ and $\geo\from{\tau(r)}\dir{R}{\tht}{+}(r')$. By \eqref{no-intersection}, the continuity of the paths, and \eqref{Itree}, $\Upsilon\from{(z',r')}\dir{R}{\tht}{+}$ must coalesce with $\tau$ at a time $s'\in[r,r')$. By Lemma \ref{aux:RbndryLgeo}, $\geo\from{\tau(s')}\dir{L}{\tht}{+}$ must proceed inside the island. By \eqref{shocksplitsgeo}, $\geo\from{\tau(s')}\dir{R}{\tht}{+}$ must proceed to the right of $\Upsilon\from{(z',r')}\dir{R}{\tht}{+}$ and, by Lemma \ref{lm:intcoal}, $\geo\from{\tau(s')}\dir{M}{\tht}{+}$ must proceed, on the time interval $(s',r')$, strictly between $\tau$ and $\Upsilon\from{(z',r')}\dir{R}{\tht}{+}$, distinct from $\geo\from{\tau(s')}\dir{R}{\tht}{+}$, as claimed. By \eqref{no-intersection}, $\geo\from{\tau(s')}\dir{M}{\tht}{+}$ remains strictly to the right of $\tau$ on the whole interval $(s',t)$. It remains to show that $\tau(s')$ is not a $\tht-$ shock.
  
  By Lemma \ref{aux:R-goin},  $\geo\from{\tau(s')}\dir{R}{\tht}{-}$ must proceed into the island, moving strictly left of $\tau$. Since $\geo\from{\tau(s')}\dir{M}{\tht}{+}$ and $\geo\from{\tau(s')}\dir{R}{\tht}{+}$ are initially distinct and both proceed strictly right of $\tau$, and since, by Lemma \ref{aux:RbndryLgeo}, $\geo\from{\tau(s')}\dir{L}{\tht}{-}$ proceeds strictly left of $\tau$, \eqref{no4stars} implies that $\geo\from{\tau(s')}\dir{L}{\tht}{-}$ and $\geo\from{\tau(s')}\dir{R}{\tht}{-}$ must initially proceed together and hence, by Definition \ref{def:shock}, $\tau(s')\notin\NU_1^{\tht-}$.
\end{proof}

\begin{proof}[Proof of Proposition \ref{prop:islandconfigs}]
We treat the right boundary; the left-boundary statements are analogous.
Parts \eqref{prop:islandconfigs.a} and \eqref{prop:islandconfigs.b} follow from Lemmas \ref{lem:islandtoshocks}, \ref{lm:snowbird}, and \ref{lm:islandboundary}\eqref{lm:islandboundary.c}. We prove \eqref{prop:islandconfigs.c}.

Before proceeding, we record several facts that we use repeatedly. Take any $r \in (s,t)$. By Lemmas \ref{lem:islandtoshocks} and \ref{lm:islandboundary}\eqref{lm:islandboundary.a}, $\tau(r)$ is a $\tht+$ hugging shock, and hence, by Lemma \ref{LRisolgeo}\eqref{LRisolgeo.b}, it cannot be a $\tht-$ hugging shock. Moreover, Lemma \ref{aux:RbndryLgeo} and \eqref{geo:restart} imply
\begin{align}\label{useful1}
\geo\from{\tau(r)}\dir{L}{\tht}{-}\big|_{(r,t)}
=
\geo\from{\tau(r)}\dir{L}{\tht}{+}\big|_{(r,t)}
\prec
\tau\big|_{(r,t)},\quad
\geo\from{\tau(r)}\dir{L}{\tht}{-}(t)
=
\geo\from{\tau(r)}\dir{L}{\tht}{+}(t)
=
\tau(t),\quad\text{and}\quad
\geo\from{\tau(r)}\dir{L}{\tht}{+}\big|_{[t,\infty)}
=
\geo\from{\tau(t)}\dir{R}{\tht}{+}.
\end{align}

We begin by showing that exactly three mutually exclusive alternatives---namely \eqref{prop:islandconfigs.ci}-\eqref{prop:islandconfigs.ciii}---can occur, and that each arises on a dense subset of $(s,t)$.

\medskip{\bf Case 1 (type (}\ref{prop:islandconfigs.ci}{\bf)).}
Since $\tau(r)$ is not a $\tht-$ hugging shock, $\geo\from{\tau(r)}\dir{R}{\tht}{-}$ cannot initially run with $\geo\from{\tau(r)}\dir{R}{\tht}{+}$. Then, if $r\in(s,t)$ is rational, Lemma \ref{lm:Duncan}\eqref{Duncan.b} implies $\geo\from{\tau(r)}\dir{M}{\tht}{+}=\geo\from{\tau(r)}\dir{R}{\tht}{+}$ and Lemma \ref{aux:R-goin} forces $\geo\from{\tau(r)}\dir{R}{\tht}{-}$ to run initially with $\geo\from{\tau(r)}\dir{L}{\tht}{-}$; hence $\tau(r)\notin\NU_1^{\tht-}$ and $\tau(r)$ is of type \eqref{prop:islandconfigs.ci}. Rational times are dense, so type \eqref{prop:islandconfigs.ci} occurs densely in $(s,t)$.

Now, consider $r \in (s,t)$ such that $\tau(r)$ is not of type \eqref{prop:islandconfigs.ci}. This means that either $\tau(r) \in \NU_1^{\tht-}$, or $\tau(r) \notin \NU_1^{\tht-}$ and $\geo\from{\tau(r)}\dir{M}{\tht}{+}$ is initially distinct from $\geo\from{\tau(r)}\dir{R}{\tht}{+}$. 

\medskip{\bf Case 2 (type (}\ref{prop:islandconfigs.cii}{\bf)).}
If $\tau(r)\in\NU_1^{\tht-}$ then Lemma \ref{aux:R-goin} implies that, up to time $t$, $\geo\from{\tau(r)}\dir{M}{\tht}{+}$ matches $\geo\from{\tau(r)}\dir{R}{\tht}{-}$ and both $\geo\from{\tau(r)}\dir{L}{\tht}{-}$ and $\geo\from{\tau(r)}\dir{R}{\tht}{-}$ lie strictly left of $\tau$. By \eqref{useful1}, the left one passes through $\tau(t)$, so the two $W^{\tht-}$ geodesics must coalesce by time $t$.
To show that $\tau(r)$ is of type \eqref{prop:islandconfigs.cii}, we must show this coalescence occurs strictly before $t$. 

Suppose, for contradiction, they coalesce at $\tau(t)$. Then by Lemma \ref{lm:geocoal} there is a $\tht-$ shock interface $\overline\tau$ from $\tau(t)$ running strictly between those two geodesics; because $\geo\from{\tau(r)}\dir{R}{\tht}{-}$ is strictly left of $\tau$ on $(r,t)$, the interface $\overline\tau$ also lies strictly left of $\tau$ on $(r,t)$. But the left boundary of the island is $\Upsilon_{\tau(t)}\dir{L}{\tht}{-}|_{[s,t]}$, which by Lemma \ref{lm:geosplitsshocks} is strictly left of $\geo\from{\tau(r)}\dir{L}{\tht}{-}|_{(r,t)}$. Hence $\overline\tau|_{(r,t)}$ would lie inside the island, contradicting Lemma \ref{lem:IGgodown} (since $\tau(t)\in\IG\tht$). Thus the coalescence is strictly before $t$, so $\tau(r)$ is of type \eqref{prop:islandconfigs.cii}. Lemma \ref{aux:-shocksalongRboundary} gives the density of such $r\in(s,t)$.

\medskip{\bf Case 3 (type (}\ref{prop:islandconfigs.ciii}{\bf)).}
The remaining possibility is
\begin{equation}\label{aux2462}
\tau(r)\notin\NU_1^{\tht-}\quad\text{and}\quad
\geo\from{\tau(r)}\dir{M}{\tht}{+}\ \text{is initially distinct from}\ \geo\from{\tau(r)}\dir{R}{\tht}{+}.
\end{equation}
The $M$-geodesic convention in Remark \ref{rk:Mgeo} implies the middle geodesic is also initially distinct from $\geo\from{\tau(r)}\dir{L}{\tht}{+}$. Since $\tau(r)\in\NU_1^{\tht+}$, Lemma \ref{aux:1651} ensures 
\begin{align}\label{useful2}
\geo\from{\tau(r)}\dir{M}{\tht}{+}\text{ coalesces with one of }\geo\from{\tau(r)}\dir{L}{\tht}{+}\text{ or }\geo\from{\tau(r)}\dir{R}{\tht}{+}\text{ strictly before it coalesces with the other.}
\end{align}

Before completing the proof of the claim that $\tau(r)$ is of type \eqref{prop:islandconfigs.ciii}, we first show that for any $r\in(s,t)$ satisfying \eqref{aux2462}, we have
    \begin{align}\label{aux2466}
    \tau|_{(r,t)}\prec\geo\from{\tau(r)}\dir{M}{\tht}{+}\big|_{(r,t)}.
    \end{align}
If not, continuity and \eqref{no-intersection} force  $\geo\from{\tau(r)}\dir{M}{\tht}{+}$ into the island, so it would coalesce with $\geo\from{\tau(r)}\dir{L}{\tht}{+}$ at some $s'\in(r,t]$. Lemma \ref{lm:geocoal} would then produce a $\tht+$ shock interface from that coalescence point which passes through $\tau(r)$ and (by Lemmas \ref{lem:stableshocks} and \ref{lem:movingonshocks}) is part of a $\tht-$ shock interface, contradicting $\tau(r)\notin\NU_1^{\tht-}$. This proves \eqref{aux2466}.    

To conclude that $\tau(r)$ is of type \eqref{prop:islandconfigs.ciii}, it remains to show that if $\geo\from{\tau(r)}\dir{M}{\tht}{+}$ first coalesces with $\geo\from{\tau(r)}\dir{L}{\tht}{+}$, that  coalescence must occur strictly after time $t$. 
Indeed, \eqref{useful1} says $\geo\from{\tau(r)}\dir{L}{\tht}{+}$ stays strictly left of $\tau$ on $(r,t)$ while \eqref{aux2466} places the middle geodesic strictly right of $\tau$ on the same interval; hence they cannot meet before $t$. 

Suppose the two geodesics coalesce exactly at $\tau(t)$. Fix a rational $r' \in (r,t)$. 
By \eqref{useful1} and the coalescence \eqref{geo:coal1}, $\geo\from{\tau(r')}\dir{L}{\tht}{-}$ lies strictly to the left of $\tau$, passes through $\tau(t)$, and then continues as $\geo\from{\tau(t)}\dir{L}{\tht}{-}$. 
Since $\tau$ lies on a $\tht+$ shock interface (Lemma \ref{lem:islandtoshocks}), \eqref{shocksplitsgeo} implies that $\geo\from{\tau(r')}\dir{R}{\tht}{+}$ lies strictly to the right of $\tau$ and, by \eqref{no-intersection} and \eqref{geo:coal1}, coalesces with $\geo\from{\tau(r)}\dir{M}{\tht}{+}$, thus going through $\tau(t)$, and thereafter following $\geo\from{\tau(t)}\dir{R}{\tht}{+}$. 
Since $\tau(t) \in \IG\tht$, Lemma \ref{def:geoinstabpt} states that $\geo\from{\tau(t)}\dir{L}{\tht}{-}$ and $\geo\from{\tau(t)}\dir{R}{\tht}{+}$ immediately separate. Now we have a contradiction with Lemma \ref{lm:Duncan}\eqref{Duncan.a}. 
Hence $\tau(t)<\geo\from{\tau(r)}\dir{M}{\tht}{+}(t)$, and therefore $\geo\from{\tau(r)}\dir{L}{\tht}{+}$ and $\geo\from{\tau(r)}\dir{M}{\tht}{+}$ must coalesce strictly after time $t$.

We have established that in the last case, where \eqref{aux2462} holds, $\tau(r)$ is of type \eqref{prop:islandconfigs.ciii}. Lemma \ref{aux:shockconfluenceexist} supplies a dense set of $r$ satisfying \eqref{aux2462}, so type \eqref{prop:islandconfigs.ciii} occurs densely as well.

To complete the proof of the proposition, we must determine, for points of type \eqref{prop:islandconfigs.ciii}, whether the middle geodesic first coalesces with the left or the right geodesic.

\begin{figure}[hpt]
    \includegraphics[width=3.5cm]{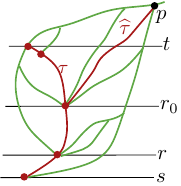}\qquad\qquad
    \includegraphics[width=4.5cm]{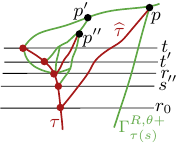}
    \caption{\small Left: Construction of $r_0$ and Cases 3.1 and (3.3.1). Right: Case 3.2.}
    \label{fig:islandconfigs-a}
\end{figure}

\medskip{\bf Refinement: existence of} $r_0$ {\bf and the coalescence pattern.}
We now construct the distinguished time $r_0\in(s,t)$ and describe the coalescence pattern for type \eqref{prop:islandconfigs.ciii} points. See the left panel of Figure \ref{fig:islandconfigs-a}.

Let $p$ be the coalescence point of $\geo\from{\tau(t)}\dir{R}{\tht}{+}$ and $\geo\from{\tau(s)}\dir{R}{\tht}{+}$. By Lemmas \ref{lem:islandtoshocks}, \ref{lm:snowbird} and \ref{lem:SBisisland}, $\tau(t)$ is the first separation point of $\geo\from{\tau(s)}\dir{R}{\tht}{-}$ and $\geo\from{\tau(s)}\dir{L}{\tht}{+}$, and Lemma \ref{topsnowbird} implies $p$ is at a time level strictly larger than $t$. By Lemma \ref{lm:geocoal} there is a $\tht+$ shock interface $\wh\tau$ from $p$ running strictly between $\geo\from{\tau(t)}\dir{R}{\tht}{+}$ and $\geo\from{\tau(s)}\dir{R}{\tht}{+}$. Continuity and \eqref{no-intersection} force $\wh\tau$ to meet $\tau$ at some time $r_0\in[s,t]$. We show that $r_0\in(s,t)$.\smallskip

Since $\tau$ is a $\tht+$ shock interface, \eqref{Itree} says that the intersection point is in fact a coalescence point of the two $\tht+$ shock interfaces. 
By Lemma \ref{lm:intcoal}, $\geo\from{\tau(r_0)}\dir{M}{\tht}{+}$ runs strictly between $\wh\tau$ and $\tau$, and $\geo\from{\tau(r_0)}\dir{R}{\tht}{+}$ strictly to the right of $\wh\tau$.
By Lemma \ref{lm:islandboundary}\eqref{lm:islandboundary.b}, $\tau(s)$ is a snowbird shock (Definition \ref{def:snowbird}). This and \eqref{no4stars} rule out distinct middle geodesics at $\tau(s)$. Therefore, we must have $r_0>s$. 

On the other hand, Theorem \ref{th:convanydir} says that along a subsequence $r'\nearrow t$, $\geo\from{\tau(r')}\dir{R}{\tht}{+}$ must converge, in overlap topology, to $\geo\from{\tau(t)}\dir{S}{\tht}{+}$ for some $S\in\{L,M,R\}$. But since, by \eqref{no-intersection}, $\geo\from{\tau(r')}\dir{R}{\tht}{+}$ cannot cross $\tau|_{(r',t)}$ and, by \eqref{geo:coal1}, it coalesces with $\geo\from{\tau(t)}\dir{R}{\tht}{+}$ as soon as the two meet, we have $\geo\from{\tau(r')}\dir{R}{\tht}{+}\succeq\geo\from{\tau(t)}\dir{R}{\tht}{+}$. This and the monotonicity \eqref{geo:mono} implies that the limit, as $r'\nearrow t$, must be $\geo\from{\tau(t)}\dir{R}{\tht}{+}$. Thus, for $r'\in(s,t)$ close enough to $t$, $\geo\from{\tau(r')}\dir{R}{\tht}{+}$ coalesces with $\geo\from{\tau(t)}\dir{R}{\tht}{+}$ before the latter reaches $p$. But then \eqref{no-intersection} and continuity force the interface $\wh\tau$ to remain right of $\geo\from{\tau(r')}\dir{R}{\tht}{+}$ and we get that $r_0<t$. We have established that $r_0\in(s,t)$. Now, we describe the coalescence pattern in three cases.
    
\medskip{\bf Case 3.1 (middle-with-right in} $(s,r_0)${\bf).} See the left panel in Figure \ref{fig:islandconfigs-a}. Consider $r\in(s,r_0)$ such that $\tau(r)$ is of type \eqref{prop:islandconfigs.ciii}. 
By \eqref{aux2466}, \eqref{shocksplitsgeo}, and the ordering \eqref{geo:mono}, we have
$\geo\from{\tau(r_0)}\dir{R}{\tht}{+}
\preceq
\geo\from{\tau(r)}\dir{M}{\tht}{+}
\preceq
\geo\from{\tau(r)}\dir{R}{\tht}{+}
\preceq
\geo\from{\tau(s)}\dir{R}{\tht}{+}$.
Since $\geo\from{\tau(s)}\dir{R}{\tht}{+}$ passes through $p$ and $\geo\from{\tau(r_0)}\dir{R}{\tht}{+}$ lies strictly to the right of $\wh\tau$ before reaching $p$, it follows that $\geo\from{\tau(r)}\dir{M}{\tht}{+}$ and $\geo\from{\tau(r)}\dir{R}{\tht}{+}$ either coalesce strictly before $p$ and hence before they coalesce with $\geo\from{\tau(t)}\dir{R}{\tht}{+}$, or the latter three geodesics coalesce simultaneously at $p$.
By \eqref{useful1}, the same conclusion holds if $\geo\from{\tau(t)}\dir{R}{\tht}{+}$ is replaced by $\geo\from{\tau(r)}\dir{L}{\tht}{+}$. However, \eqref{useful2} rules out the possibility that the three geodesics emanating from $\tau(r)$ coalesce simultaneously at $p$. Therefore, $\geo\from{\tau(r)}\dir{M}{\tht}{+}$ must coalesce with $\geo\from{\tau(r)}\dir{R}{\tht}{+}$ strictly before it coalesces with $\geo\from{\tau(r)}\dir{L}{\tht}{+}$.

\medskip{\bf Case 3.2 (middle-with-right in} $(r_0,t)${\bf).} See the right panel in Figure \ref{fig:islandconfigs-a}. Take any $s'<t'$ in $(r_0,t)$. By \eqref{shocksplitsgeo}, $\geo\from{\tau(s')}\dir{R}{\tht}{+}$ and $\geo\from{\tau(t')}\dir{R}{\tht}{+}$ both go strictly right of $\tau$ and, by \eqref{no-intersection}, the two must remain strictly left of $\wh\tau$. 
    By \eqref{geo:coal2},  $\geo\from{\tau(t')}\dir{R}{\tht}{+}$ must coalesce with $\geo\from{\tau(t)}\dir{R}{\tht}{+}$ at say a point $p'$, which is at or earlier than $p$. Similarly to the argument we made above to show that $r_0<t$, using Theorem \ref{th:convanydir} and geodesic ordering \eqref{geo:mono}, we see that by taking $s''\in(s',t')$ close enough to $t'$, we can guarantee that $\geo\from{\tau(s'')}\dir{R}{\tht}{+}$ and $\geo\from{\tau(t')}\dir{R}{\tht}{+}$ coalesce at a point $p''$ strictly before the latter gets to $p'$.
    Now, by Lemma \ref{lm:geocoal}, there exists a $\tht+$ shock interface $\wt\tau$ from $p''$, that runs strictly between the two geodesics. By \eqref{no-intersection}, continuity, and coalescence \eqref{Itree}, the interface $\wt\tau$ has to coalesce with $\tau$ at some time $r\in[s'',t']$. Then Lemma \ref{lm:intcoal} says  $\geo\from{\tau(r)}\dir{M}{\tht}{+}$ moves strictly between $\wt\tau$ and $\tau$, while $\geo\from{\tau(r)}\dir{R}{\tht}{+}$ moves strictly right of $\wt\tau$, making $\tau(r)$ a point of type \eqref{prop:islandconfigs.ciii}. By continuity, the ordering \eqref{geo:mono}, and the fact that all these geodesics remain strictly right of $\tau$,   $\geo\from{\tau(r)}\dir{M}{\tht}{+}$ and $\geo\from{\tau(r)}\dir{R}{\tht}{+}$ are sandwiched between $\geo\from{\tau(s'')}\dir{R}{\tht}{+}$ and $\geo\from{\tau(t')}\dir{R}{\tht}{+}$ and must thus coalesce by the time $p''$ is reached, after which they coalesce with $\geo\from{\tau(t)}\dir{R}{\tht}{+}$ when $p'$ is reached. By \eqref{useful1}, the same holds with $\geo\from{\tau(t)}\dir{R}{\tht}{+}$ replaced by $\geo\from{\tau(r)}\dir{L}{\tht}{+}$. 
     Thus, the middle $W^{\tht+}$-geodesic out of $\tau(r)$ coalesces with the right one, strictly before the left one.  Since $s'<t'$ were arbitrary in $(s,t)$, we get that this scenario happens for a dense set of times $r\in(r_0,t)$.
    
\medskip{\bf Case 3.3 (middle-with-left in} $[r_0,t)${\bf).} It remains to describe the set of times $r\in[r_0,t)$ where $\tau(r)$ is of type \eqref{prop:islandconfigs.ciii} and $\geo\from{\tau(r)}\dir{M}{\tht}{+}$ coalesces first with $\geo\from{\tau(r)}\dir{L}{\tht}{+}$. Denote this set of times by $\mathcal R$. We show: (3.3.1) $r_0\in \mathcal R$, (3.3.2) every point of $\mathcal R$ is isolated, and (3.3.3) $\mathcal R$ accumulates at $t$.

\medskip{\bf(3.3.1)} See the left panel in Figure \ref{fig:islandconfigs-a}. Since $r_0\in(s,t)$ and $\geo\from{\tau(r_0)}\dir{M}{\tht}{+}$ is distinct from $\geo\from{\tau(r_0)}\dir{R}{\tht}{+}$, we get that $\tau(r_0)$ is of type \eqref{prop:islandconfigs.ciii}. Since $\geo\from{\tau(r_0)}\dir{M}{\tht}{+}$ goes strictly left of $\wh\tau$ while $\geo\from{\tau(r_0)}\dir{R}{\tht}{+}$ goes strictly right of it, $\geo\from{\tau(r_0)}\dir{M}{\tht}{+}$ must either coalesce with $\geo\from{\tau(t)}\dir{L}{\tht}{+}$ before coalescing with $\geo\from{\tau(r_0)}\dir{R}{\tht}{+}$, or all three coalesce simultaneously at $p$. 
By \eqref{useful1}, the same holds with $\geo\from{\tau(t)}\dir{R}{\tht}{+}$ replaced by $\geo\from{\tau(r_0)}\dir{L}{\tht}{+}$. By \eqref{useful2}, the simultaneous three-way coalescence cannot happen and thus, the middle geodesic coalesces with the left one strictly before the right one. Consequently, $r_0\in\mathcal R$. 

\begin{figure}[hpt]
    \includegraphics[width=4.5cm]{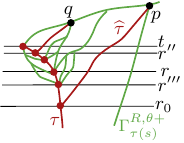}\qquad\qquad
    \includegraphics[width=4.5cm]{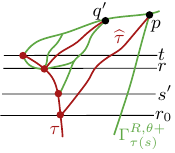}
    \caption{\small Left: Cases (3.3.2). Right: Case (3.3.3).}
    \label{fig:islandconfigs-b}
\end{figure}

\medskip{\bf(3.3.2)} See the left panel in Figure \ref{fig:islandconfigs-b}. Take $r\in \mathcal R$. Let $q$ denote the coalescence point of  $\geo\from{\tau(r)}\dir{M}{\tht}{+}$ and $\geo\from{\tau(t)}\dir{R}{\tht}{+}$. By \eqref{lm:geocoal}, there exists a $\tht+$ shock interface $\gamma$ out of $q$ that runs strictly between the two geodesics. By \eqref{no-intersection}, continuity, and the coalescence \eqref{Itree}, this interface must coalesce with $\tau$ at a point $\tau(r'')$ with $r''\in[r,t]$.  It cannot be that $r''=r$ since then $\tau(r)$ would be a coalescence point of two $\tht+$ shock interfaces, and Lemma \ref{lm:intcoal} would force $\geo\from{\tau(r)}\dir{M}{\tht}{+}$ to go strictly between $\tau$ and $\gamma$, which is not the case. Thus, $r''\in(r,t]$. Take any time $s''\in(r,r'')$. If $\tau(s'')$ is of type \eqref{prop:islandconfigs.ciii}, then \eqref{no-intersection} implies that $\geo\from{\tau(s'')}\dir{M}{\tht}{+}$ must remain strictly to the right of $\gamma$, which implies that $\geo\from{\tau(s'')}\dir{M}{\tht}{+}$ and $\geo\from{\tau(s'')}\dir{R}{\tht}{+}$ must coalesce together and with $\geo\from{\tau(r)}\dir{M}{\tht}{+}$ by the time $q$ is reached. This means the two coalesce  before (or at the same time as) coalescing with $\geo\from{\tau(t)}\dir{R}{\tht}{+}$. By \eqref{useful1}, the same holds when replacing $\geo\from{\tau(t)}\dir{R}{\tht}{+}$ with $\geo\from{\tau(s'')}\dir{L}{\tht}{+}$, and \eqref{useful2} prevents the simultaneous coalescence at $q$. Thus, $\geo\from{\tau(s'')}\dir{M}{\tht}{+}$ coalesces with $\geo\from{\tau(s'')}\dir{R}{\tht}{+}$  strictly before $\geo\from{\tau(s'')}\dir{L}{\tht}{+}$.  Hence, $s''\notin \mathcal R$ and $\mathcal R$ consists of right-isolated points.

    To see that the times in $\mathcal R$ are all left-isolated, it suffices to consider $r\in \mathcal R\setminus\{r_0\}$, since we have already established that $(s,r_0)\cap\mathcal R=\varnothing$.  By the same argument we used above to show that $r_0<t$, applying Theorem \ref{th:convanydir} and the ordering \eqref{geo:mono}, we get that for $r'''\in(r_0,r)$ close enough to $r$, $\geo\from{\tau(r''')}\dir{R}{\tht}{+}$ coalesces with $\geo\from{\tau(r)}\dir{R}{\tht}{+}$, before the latter coalesces with $\geo\from{\tau(t)}\dir{R}{\tht}{+}$. Then, for any $s''\in[r''',r)$, if $\tau(s'')$ is of type \eqref{prop:islandconfigs.ciii}, the ordering \eqref{geo:mono} sandwiches $\geo\from{\tau(s'')}\dir{M}{\tht}{+}$ between $\geo\from{\tau(s'')}\dir{R}{\tht}{+}$ and $\geo\from{\tau(r)}\dir{R}{\tht}{+}$ and forces all three to coalesce strictly before $\geo\from{\tau(t)}\dir{R}{\tht}{+}$ (and hence, by \eqref{useful1}, $\geo\from{\tau(s'')}\dir{L}{\tht}{+}$) is reached. Thus, we get that $s''\notin\mathcal R$ and so $\mathcal R$ consists of left-isolated points. 

\medskip{\bf(3.3.3)} See the right panel in Figure \ref{fig:islandconfigs-b}. To see $\mathcal R$ accumulates at $t$, pick any $s'\in(r_0,t)$ and let $q'$ be the coalescence point of $\geo\from{\tau(s')}\dir{R}{\tht}{+}$ with $\geo\from{\tau(t)}\dir{R}{\tht}{+}$. Lemma \ref{lm:geocoal} yields a $\tht+$ shock interface from $q'$ that runs strictly between the two geodesics and, by \eqref{no-intersection} and  continuity,  meets $\tau$ at a time $r\in[s',t]$. Since $\tau(t)$ is on the relative interior of the geodesic $\geo\from{\tau(s)}\dir{L}{\tht}{+}$, \eqref{no-intersection} also prevents $r=t$. Now, Lemma \ref{lm:intcoal}, \eqref{useful1}, and \eqref{useful2} force $\tau(r)$ to be of type \eqref{prop:islandconfigs.ciii} and to have the middle $W^{\tht+}$-geodesic coalesce first with the left one, so $r\in\mathcal R$. Since $s'$ was arbitrary in $(r_0,t)$, such $r$ values can be taken arbitrarily close to $t$, so $\mathcal R$ accumulates at $t$.

    \smallskip
We have shown that every $r\in(s,t)$ falls into exactly one of the types (\ref{prop:islandconfigs.ci}-\ref{prop:islandconfigs.ciii}), and that each type occurs on a dense set, with the additional structural description for type \eqref{prop:islandconfigs.ciii} given by the existence of $r_0\in(s,t)$ such that the middle geodesic first coalesces with the right one, for all $r\in(s,r_0)$ and for a dense set of $r\in(r_0,t)$, and it coalesces first with the left geodesic, a set of isolated times $r\in[r_0,t)$ that includes $r_0$ and that accumulates at $t$. This completes the proof of the proposition.
\end{proof}

\begin{proof}[Proof of Proposition \ref{prop:islandsdense}]
See Figure \ref{fig:islandsdense} for an illustration  of the proof.
 Let $B_\delta$ denote the open ball centered at $(x,s)$ with radius $\delta$. By continuity, there exists $r<s$ such that the interface $\Upsilon_{(x,s)}\dir{L}{\tht}{+}$ is contained in $B_\delta$ on $[r,s]$. Choose a rational time $s'\in(r,s)$ and let $(x',s')=\Upsilon_{(x,s)}\dir{L}{\tht}{+}(s')$. Then $(x',s')\in B_\delta$. 
Again by continuity, there exists $t>s$ such that the geodesics $\geo\from{(x',s')}\dir{S}{\tht}{\sig}$, for $S\in\{L,M,R\}$ and $\sigg\in\{-,+\}$, as well as $\geo\from{(x,s)}\dir{L}{\tht}{-}$, all remain inside $B_\delta$ up to time $t$.
  
By Theorem \ref{th:convanydir}, along a subsequence $s'' \nearrow s'$, the geodesics $\geo\from{\Upsilon_{(x,s)}\dir{L}{\tht}{+}(s'')}\dir{L}{\tht}{+}$ converge in the overlap topology to $\geo\from{(x',s')}\dir{S}{\tht}{+}$ for some $S \in \{L,M,R\}$. However, by \eqref{shocksplitsgeo}, we have $\geo\from{\Upsilon_{(x,s)}\dir{L}{\tht}{+}(s'')}\dir{L}{\tht}{+}(s') < (x',s')$, and then the ordering \eqref{geo:mono} forces this geodesic to lie weakly to the left of $\geo\from{(x',s')}\dir{L}{\tht}{+}$. Hence the only possible limit is $\geo\from{(x',s')}\dir{L}{\tht}{+}$.
Consequently, for $s'' \in (r,s')$ sufficiently close to $s'$, writing $p = \Upsilon_{(x,s)}\dir{L}{\tht}{+}(s'')$, the geodesic $\geo\from{p}\dir{L}{\tht}{+}$ remains inside $B_\delta$ on $[s'',t]$ and coalesces with $\geo\from{(x',s')}\dir{L}{\tht}{+}$ by time $t$.

Similarly, possibly taking $s''$ even closer to $s'$, we can ensure that $\geo\from{p}\dir{R}{\tht}{+}$ remains inside $B_\delta$ on $[s'',t]$ and coalesces with $\geo\from{(x',s')}\dir{R}{\tht}{+}$ by time $t$. Applying Theorem \ref{th:convanydir} to $\geo\from{\Upsilon_{(x,s)}\dir{L}{\tht}{+}(s'')}\dir{L}{\tht}{-}$ shows that these geodesics converge in overlap topology to $\geo\from{(x',s')}\dir{S}{\tht}{-}$ for some $S \in \{L,M,R\}$. Since all the limiting geodesics lie inside $B_\delta$ on $[s',t]$, choosing $s''$ sufficiently close to $s'$ also ensures that $\geo\from{p}\dir{L}{\tht}{-}$ remains inside $B_\delta$ on $[s'',t]$.

\begin{figure}[hpt]
    \includegraphics[width=6cm]{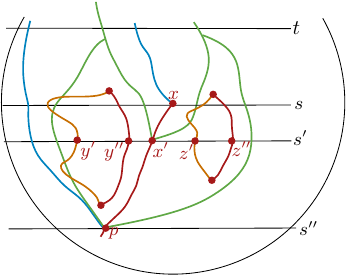}
    \caption{\small The proof of Proposition \ref{prop:islandsdense}. The arc is part of the ball $B_\delta$, centered at $(x,s)$.}
    \label{fig:islandsdense}
\end{figure}

By Lemma \ref{lem:IGgodown}, we have $(x',s') \in \IG\tht$, and by Lemma \ref{no-isolated}, this point is either not left-isolated or not right-isolated. The two cases are similar and we first treat the former. Then there exists a sequence $y_n \nearrow x'$ with $(y_n,s') \in \IG\tht$ for all $n$. 
Choose $n$ large enough so that $(y_n,s')$ lies strictly between $\geo\from{p}\dir{L}{\tht}{+}(s')$ and $(x',s')$. Let $a \in (y_n,y_{n+1})$ be rational and define
$y' = \sup\{\, y < a : (y,s') \in \IG\tht \,\}$. 
Since $\IG\tht$ is closed by Lemma \ref{lem:IGisclosed}, we have $(y',s') \in \IG\tht$. By \eqref{QnotIG}, $(a,s') \notin \IG\tht$, so $y' < a$, and the maximality of $y'$ implies that $(y',s')$ is right-isolated. Furthermore, $(y_n,s')\in\IG\tht$ implies $(y',s')\ge(y_n,s')>\geo\from{p}\dir{L}{\tht}{+}(s')$.
Similarly, defining
$y'' = \inf\{\, y > a : (y,s') \in \IG\tht \,\}$
yields a left-isolated point $(y'',s') \in \IG\tht$ with $y'' > a$ and $(y'',s')\le(y_{n+1},s')<(x',s')$. By construction, no point strictly between $(y',s')$ and $(y'',s')$ belongs to $\IG\tht$.

Consider the island $\island_{(a,s')}$. By the above, its left boundary must pass through $(y',s')$ and its right boundary passes through $(y'',s')$. By Lemma \ref{lm:islandboundary}\eqref{lm:islandboundary.c}, the right boundary lies on a $\tht+$ shock interface. Since $(y'',s') < (x',s')$, the duality \eqref{no-intersection} implies that this right boundary must remain strictly to the left of $\geo\from{(x',s')}\dir{L}{\tht}{+}$.
On the other hand, because $\geo\from{p}\dir{L}{\tht}{+}$ passes strictly to the left of $(y',s')$, the right boundary of the island must remain to the right of $\geo\from{p}\dir{L}{\tht}{+}$. Furthermore, Lemma \ref{lem:IGgodown} implies that $\Upsilon_{(x,s)}\dir{L}{\tht}{+}$ consists entirely of instability points and hence cannot enter the island. Since this interface passes through $(x',s')$, the right boundary of the island must also remain weakly to its left.
Combining these observations, we conclude that the right boundary of the island lies in the region bounded by $\geo\from{p}\dir{L}{\tht}{+}$, $\geo\from{(x',s')}\dir{L}{\tht}{+}$, and $\Upsilon_{(x,s)}\dir{L}{\tht}{+}$, which is contained in $B_\delta$. This also shows that the island is contained in the time zone $[s'',t]$.

The left boundary of the island lies on a $\tht-$ shock interface (by Lemma \ref{lem:IGgodown}). Since $\geo\from{p}\dir{L}{\tht}{-} \preceq \geo\from{p}\dir{L}{\tht}{+}$ and the latter passes strictly to the left of $(y',s')$, it follows that the left boundary of the island must remain  to the left of $\geo\from{p}\dir{L}{\tht}{-}$. Together with the fact that the left boundary lies to the left of the right boundary, this shows that the entire island is contained in $B_\delta$.

An analogous argument applies when $(x',s')$ is not right-isolated. In this case, we construct an island whose  right boundary is confined between $\geo\from{p}\dir{R}{\tht}{+}$, $\geo\from{(x',s')}\dir{R}{\tht}{+}$, and $\Upsilon_{(x,s)}\dir{L}{\tht}{+}$, and whose left boundary is constrained by the right boundary together with $\Upsilon_{(x,s)}\dir{L}{\tht}{+}$ and $\geo\from{(x,s)}\dir{L}{\tht}{-}$.
\end{proof}

We conclude the section with an observation that we do not use in this work but nevertheless find interesting, as it highlights the importance of stability islands. In words, the lemma says that $\tht-$ and $\tht+$ shock interfaces cannot intersect at a dust point and that, for $\sigg\in\{-,+\}$, $\tht\sigg$ shock interfaces can never coalesce at a dust point. 

\begin{lem}\label{shocksintersect}
Let $\w\in\Omega_0$, $\tht\in\baddir$, and $\sigg,\overline\sigg\in\{+,-\}$. Let $\tau$ and $\overline\tau$ be $\tht\sigg$ and $\tht\overline\sigg$ shock interfaces, respectively, and suppose that $(z,r)$ lies in the relative interior of both. 
If $(z,r)\in\IG\tht$, then it lies on the boundary of a stability island.
\end{lem}

\begin{proof}
    By parts \eqref{prop:geods-dust.b}-\eqref{prop:geods-dust.c} of Proposition \ref{prop:geods-dust}, a dust  point cannot be a double shock. This proves the claim in the case $\sigg\ne\overline\sigg$.\smallskip

    Suppose $\sigg = \overline\sigg$.  We treat the case $\sigg=+$, the case $\sigg=-$ being similar. By \eqref{Itree}, $(z,r)$ is the coalescence point of the two interfaces. As it is assumed to be in the relative interiors of both interfaces, we may restart the at a common time $s>r$ and, by continuity (and possibly renaming the paths), assume that $\tau$ remains strictly left of $\overline\tau$ on $(r,s]$. 
    Then the coalescence point $(z,r)$ is also the first intersection point.
    Now, by \eqref{shocksplitsgeo},  $\geo\from{(z,r)}\dir{L}{\tht}{+}$ proceeds strictly to the left of $\tau$ and $\geo\from{(z,r)}\dir{R}{\tht}{+}$ goes strictly to the right of $\overline\tau$, while by Lemma \ref{lm:intcoal}, $\geo\from{(z,r)}\dir{M}{\tht}{+}$ must proceed strictly between $\tau$ and $\overline\tau$. However, by the ordering \eqref{geo:mono}, $\geo\from{(z,r)}\dir{L}{\tht}{-}$ is left of $\geo\from{(z,r)}\dir{L}{\tht}{+}$ and then by \eqref{no4stars} the two geodesics must initially coincide. Then, Lemma \ref{lm:hugging}\eqref{lm:hugging.b} says $(z,r)$ is a $\tht+$ hugging shock and, by Lemma \ref{lm:isotoisland}, it must be on the boundary of a stability island.
\end{proof}


\section{Proofs of the main theorems}\label{sec:mainproofs}

\begin{proof}[Proof of Theorem \ref{main:geodesics}]
The theorem follows from Propositions \ref{prop:geodconfigstable}-\ref{prop:islandconfigs}.
\end{proof}

\begin{proof}[Proof of Theorem \ref{main:IG}]
We first show that Definition \ref{def:stht} is equivalent to \eqref{IGdef1}. Let $\whIG\tht$ denote the set whose complement is given by the right-hand side of \eqref{IGdef1}. 
Suppose $(x,s)\notin\whIG\tht$, and let $O$ be the corresponding open set in \eqref{IGdef1}. For $z<x$ and $y>x$ sufficiently close to $x$, we have $(z,s),(y,s)\in O$ and hence $W^{\tht-}(z,s;y,s)=W^{\tht+}(z,s;y,s)$. It follows from Definition \ref{def:stht} that $(x,s)\notin\IG\tht$.
Conversely, suppose $(x,s)\notin\IG\tht$, and let $O=\island_{(x,s)}$ be the stability island containing $(x,s)$. By Lemma \ref{lm:bdry}\eqref{lm:bdry.b}, $O$ is the region strictly between the interfaces $\gamma([t_1,t_2])$ and $\tau([t_1,t_2])$ defined there. For any $(z,r),(y,t)\in O$, we have $r,t\in(t_1,t_2)$,
$\gamma(r)<(z,r)<\tau(r)$, and $\gamma(t)<(y,t)<\tau(t)$.
Lemma \ref{I-prop}\eqref{I-prop.d}, together with the cocycle property \eqref{cocycle}, then yields $W^{\tht-}(z,r;y,t)=W^{\tht+}(z,r;y,t)$. Hence $(x,s)\notin\whIG\tht$. Thus, $\whIG\tht=\IG\tht$ and we may proceed with the remainder of the proof.\smallskip

Part \eqref{main:IG.closed} follows from Lemmas \ref{no-isolated} and \ref{lem:IGisclosed}. Part \eqref{main:IG.islands} is in Lemmas \ref{lm:islands} and \ref{lm:disjointislands}. Lemma \ref{lem:islandtoshocks} says that any stability island fits the setting of Lemma \ref{lem:shockstoisland}, which then gives the claim in part \eqref{main:IG.boundary} about the boundary of the island. Lemma \ref{lm:islandboundary} identifies the tip and bottom of the island. This proves part \eqref{main:IG.boundary}.  Proposition \ref{prop:islandsdense} implies part \eqref{main:IG.islandsdense}.\smallskip

Part \eqref{main:IG.tbi}.
Proposition \ref{prop:IGgoup} constructs, for each $(x,s)\in\IG\tht$ and $\sigg\in\{-,+\}$, a continuous bi-infinite space-time interface $\Ipath_{(x,s)}^{\tht\sig}$ that lies entirely in $\IG\tht$. In particular, $\IG\tht$ is bi-infinite in both time directions, and by part \eqref{IGgoup.j} of that proposition these interfaces are $\tht$-directed forward and backward in time. It remains to show that there are uncountably many pairwise disjoint such interfaces that run entirely through dust points.

By Proposition \ref{prop:IGgoup}\eqref{IGgoup.g}, any two distinct $\Ipath_{(x,s)}^{\tht+}$ interfaces are disjoint. By \eqref{Hausdorff}, there are uncountably many points in $\IG\tht\cap(\R\times\{0\})$ that are not right-isolated. For each such point $(x,0)$, part \eqref{IGgoup.e} of the same proposition implies that $(x,0)\in\Ipath_{(x,0)}^{\tht+}$, and hence these points generate uncountably many disjoint interfaces.
By Lemma \ref{lm:bdry}\eqref{lm:bdry.a}, the connected components of the region strictly between  $\Ipath_{(x,0)}^{\tht-}$ and $\Ipath_{(x,0)}^{\tht+}$ are all stability islands. Since there are only countably many such islands and the $(x,0)\in\Ipath_{(x,0)}^{\tht+}$ interfaces are disjoint, only countably many interfaces satisfy $\Ipath_{(x,0)}^{\tht-}\ne\Ipath_{(x,0)}^{\tht+}$. The remaining uncountably many interfaces avoid all island boundaries and hence lie entirely within the dust. This completes the proof of Part \eqref{main:IG.tbi}.\smallskip

Part \eqref{main:IG.xbi}. By Lemma \ref{no-isolated}, $\IG\theta$, and hence $\dust$, has no isolated points. 
Fix $a<b$ and suppose $(a,b)\times\{s\}$ contains a point $(x,s)\in\IG\theta$. Then $(x,s)$ is either not left- or not right-isolated. The two cases are similar and we treat the latter.
Then $(x,b)\times\{s\}$ contains infinitely many points of $\IG\theta$, and by Lemma \ref{lm:isodense}, infinitely many of these are not right-isolated. Choose $(y,s)<(z,s)$ among them. Let $r<s$ be rational and set
$(y',r)=\Ipath_{(y,s)}^{\theta+}(r)$ and $(z',r)=\Ipath_{(z,s)}^{\theta+}(r)$.
By parts \eqref{IGgoup.a}, \eqref{IGgoup.c}, \eqref{IGgoup.e}, \eqref{IGgoup.g} of Proposition \ref{prop:IGgoup}, we have $(y',r)<(z',r)$, both points lie in $\IG\theta$, and neither is right-isolated. Hence $\IG\theta\cap((y',z')\times\{r\})$ is infinite, and by \eqref{Hausdorff}, uncountable.

As in the proof of part \eqref{main:IG.tbi}, removing the countable set of points $(x',r)\in\IG\theta$ for which $\Ipath_{(x',r)}^{\theta+}\neq \Ipath_{(x',r)}^{\theta-}$ leaves an uncountable set for which the paths $\Ipath_{(x',r)}^{\theta+}$ are disjoint and entirely contained in $\dust$. The disjointness and continuity of these interfaces (part \eqref{IGgoup.a} of Proposition \ref{prop:IGgoup}) imply that the points $\Ipath_{(x',r)}^{\theta+}(s)$ are distinct and lie in $\IG\theta\cap((y,z)\times\{s\})$. Since they are all dust points, it follows that $\dust\cap((y,z)\times\{s\})$ is uncountable, and hence so is $\dust\cap((a,b)\times\{s\})$.

Finally, Lemma \ref{kS inf} shows that $\IG\theta\cap(\R\times\{s\})$ is unbounded in both directions, and the density of $\dust\cap(\R\times\{s\})$ implies the same for $\dust$.
\smallskip

Part \eqref{main:IG.directed}.
By Lemma \ref{no-isolated}, the point $(x,s)$ is either not left-isolated or not right-isolated. Hence, by Proposition \ref{prop:IGgoup}\eqref{IGgoup.e}, there exists $\sigg\in\{-,+\}$ such that the interface $\Ipath_{(x,s)}^{\tht\sig}$ passes through $(x,s)$. Parts \eqref{IGgoup.a} and \eqref{IGgoup.c} of that proposition show that the path is continuous and contained in $\IG\tht$. Its directedness follows from part \eqref{IGgoup.j} of the same proposition.
It is worth noting that, by Lemma \ref{lem:IGgodown} together with \eqref{shocksdir}, $\tht\pm$ shock interfaces emanating from $(x,s)$ provide alternative backward-in-time paths that are $\tht$-directed.
\smallskip

Part \eqref{main:IG.graph} follows from Lemmas \ref{lem:cmnNanc} and \ref{lem:cmnSanc}.
\end{proof}

\begin{proof}[Proof of Theorem \ref{main:IGgeods}]
One direction of part \eqref{IGgeods.a} follows from Proposition \ref{prop:geods-stability}, while the other follows from the absence of the configurations in question from Propositions \ref{prop:geods-dust} and \ref{prop:islandconfigs}. 
One direction of part \eqref{IGgeods.b}, together with the density statement, follows from Propositions \ref{prop:geods-dust} and \ref{prop:islandconfigs}, while the reverse direction follows from the absence of the relevant configurations in Proposition \ref{prop:geods-stability}.
\end{proof}

\begin{proof}[Proof of Theorem \ref{main:shocks}]
Part \eqref{shocks.a}. The first claim is from Proposition \ref{prop:geods-stability}. The second claim comes from Lemma \ref{def:geoinstabpt}.\smallskip

Part \eqref{shocks.b}. 
If $(x,s)$ is the tip of an instability island, then  Lemma \ref{lem:islandtoshocks} implies $\tau^-=\Upsilon_{(x,s)}^{L,\tht-}$ and $\tau^+=\Upsilon_{(x,s)}^{R,\tht+}$ are misordered and Lemma \ref{doubleisolareisol} implies that for any $r<s$, $\tau^-([r,s])\setminus\NU_1^{\tht+}\ne\varnothing$ and $\tau^+([r,s])\setminus\NU_1^{\tht-}\ne\varnothing$.

For the converse direction, suppose there exist shock interfaces $\tau^-$ and $\tau^+$ as in the statement. If $(x,s)\notin\IG\tht$, then, by Lemma \ref{lem:IGisclosed}, there exists an open neighborhood $O\subset\R^2\setminus\IG\tht$ with $(x,s)\in O$. By Lemma \ref{lem:stableshocks}, there is no sign distinction inside $O$. Then, for $r<s$ sufficiently close to $s$ and for both $\sigg\in\{-,+\}$, we have $\tau^\sig([r,s])\subset O$, and thus $\tau^\sig([r,s])\subset \NU_1^{\tht-}\cap\NU_1^{\tht+}$. This contradicts the defining properties of $\tau^-$ and $\tau^+$, so $(x,s)\in\IG\tht$.
Since $\tau^-$ and $\tau^+$ are misordered, the ordering \eqref{shocksmono} implies that the pair $\Upsilon_{(x,s)}^{L,\tht-}$ and $\Upsilon_{(x,s)}^{R,\tht+}$ is also misordered. Lemma \ref{lem:shockstoisland} then shows that $(x,s)$ is the tip of a stability island. Lemma \ref{lm:islandboundary}\eqref{lm:islandboundary.c} implies that $(x,s)$ is a pns point.
\smallskip

Part \eqref{shocks.c} follows from Lemmas \ref{lem:shockstoisland}, \ref{lem:islandtoshocks}, \ref{lem:SBisisland}, and  \ref{lm:islandboundary}\eqref{lm:islandboundary.b}.\smallskip

Part \eqref{shocks.d}. If $(x,s)$ lies on the boundary of a stability island and is neither its tip nor its bottom, then it is left-isolated when on the right boundary and right-isolated when on the left boundary. By Lemma \ref{LRisolgeo}\eqref{LRisolgeo.a}, it is therefore a $\tht+$ hugging shock in the former case and a $\tht-$ hugging shock in the latter.
Conversely, suppose $(x,s)$ is a $\tht-$ hugging shock (the argument being similar for a $\tht+$ hugging shock). By part \eqref{shocks.a}, $(x,s)\in\IG\tht$, and Lemma \ref{LRisolgeo}\eqref{LRisolgeo.a} implies that it is right-isolated. Hence, by Lemma \ref{lm:isotoisland}, $(x,s)$ lies on the left boundary of a stability island. By parts \eqref{shocks.b} and \eqref{shocks.c}, it is neither the tip nor the bottom of the island. This proves part \eqref{shocks.d}. \smallskip

Part \eqref{shocks.e}. The first claim follows from Proposition \ref{prop:geods-dust}. The second and third claims follow from parts \eqref{shocks.b}-\eqref{shocks.d}, which imply that single shocks do not occur on island boundaries and that the only pns point on the boundary of an island is its tip.
\smallskip

Part \eqref{shocks.f} follows from Proposition \ref{prop:islandconfigs}.
\end{proof}

\appendix
\section{Analysis facts}
\subsection{Connected components}\label{app:con}
We recall some standard facts about connected components of open sets. 
An open set $U\subset\R^2$ is \emph{connected} if it cannot be written as the union of two nonempty, disjoint open subsets. Open convex sets, and in particular open balls, are connected because they are path-connected. 

The union of pairwise intersecting connected open sets is connected: if $U_i$ are connected and $U_i\cap U_j\neq \varnothing$ when $i\ne j$, 
then $\bigcup_i U_i$ is connected. 
To see this note that  $\bigcup_i U_i$ is open and if $\bigcup_i U_i = O \cup O'$ with $O,O'$ open and disjoint, then each $U_i$ lies entirely in $O$ or $O'$. 
Since the sets $U_i$ are pairwise intersecting, they all must lie in the same set, forcing the other to be empty. 

A \emph{connected component} of an open set $U$ is a connected subset $O\subset U$ that is maximal: if $O'\subset U$ is connected and $O\subset O'$, then $O=O'$. Maximality and the previous paragraph imply that any two connected components are either equal or disjoint. 
Since $\R^2$ is locally path-connected, any connected subset is also path-connected. 
For $p\in U$, the connected component containing $p$ is
\[
O_p = \!\!\!\!\!\!\bigcup_{\substack{p\in O\subset U\\ O \text{ open and connected}}} \!\!\!\!\!\!\!\!\!\!\!\!\!\!\!\!O.
\]
Using separability (e.g., density of the rationals), one sees that $U$ is the union of countably many disjoint connected components. The connected components define an equivalence relation on $U$: $p\sim q$ when $O_p=O_q$.

We need the following result. We include its elementary proof for the sake of completeness. 

\begin{lem}\label{f<g:connected}
Let $f,g:\R\to\R$ be continuous functions such that $f(t)\le g(t)$ for all $t\in\R$.
Suppose that for some $s\in\R$, $f(s)<g(s)$.
Define $t_1 = \sup\{t<s : f(t)=g(t)\}$ and $t_2 = \inf\{t>s : f(t)=g(t)\}$,
with the conventions that $\sup\varnothing = -\infty$ 
and $\inf\varnothing = +\infty$.
\begin{enumerate} [label={\rm(\alph*)}, ref={\rm\alph*}] \itemsep=1pt 
\item\label{f<g:connected.a} The set $U = \{(x,t)\in\R^2 : f(t) < x < g(t),\ t_1 < t < t_2\}$
is connected. 
\item\label{f<g:connected.b} Let $\Gamma_f=\{(f(t),t):t\in\R\}$ and $\Gamma_g=\{(g(t),t):t\in\R\}$. Let $V$ be any open subset of $\R^2\setminus(\Gamma_f\cup\Gamma_g)$ that contains $U$ {\rm(}possibly $V=U${\rm)}. Then $U$ is a connected component of $V$.
\end{enumerate}
\end{lem}

\begin{proof}
Define 
\[O_1=\{(x,t)\in V : x > g(t)\},\ 
O_2=\{(x,t)\in V : f(t) < x < g(t)\},
\ \text{and}\ 
O_3= \{(x,t)\in V : x < f(t)\}.
\]
Because $V$ is open and the maps $(x,t)\mapsto x - g(t)$ and $(x,t)\mapsto x - f(t)$ are continuous,
each of the sets $O_i$, $i\in\{1,2,3\}$, is open.
They are pairwise disjoint and their union is $V$.  Therefore, if $O$ is a connected open subset of $V$, then it must be entirely inside one of the three open subsets. Since $U\subset O_2$, if $O$ intersects $U$, then it must be entirely inside $O_2$.

Next, let 
$H = \{t \in \R : f(t) < g(t)\}$.
Since $f$ and $g$ are continuous, $H$ is open in $\R$,
and thus is a union of disjoint open intervals (its connected components). 
For each connected component $I\subset H$, define
$O_I= \{(x,t)\in O_2 : t\in I\}$.
The sets $O_I$ are open, pairwise disjoint, and
\[O_2 = \bigcup_{I\text{ component of }H} O_I.\]
For a given component $I$,
\[
\Phi_I:I\times(0,1)\longrightarrow O_I,\qquad 
\Phi_I(t,u)=\bigl((1-u)f(t)+u g(t),\,t\bigr),
\]
provides a homeomorphism. Thus, $O_I$ is path-connected and therefore connected. This shows that 
the sets $O_I$ are precisely the connected components of $O_2$.

Now let $s$, $t_1$, and $t_2$ as in the statement. 
Because $f(s) < g(s)$ and both functions are continuous, we have $t_1 < s < t_2$. This and the definitions of $t_1$ and $t_2$ imply that $(t_1,t_2)$ is the connected component of $H$ containing $s$ and $U=O_{(t_1,t_2)}$.
Consequently, $U$ is connected.  Moreover, since any connected subset of $V$ that meets $U$ must lie in $O_2$, and within $O_2$ the connected components are exactly the $O_I$'s, it follows that $U$ is a connected component of $V$.
\end{proof}

\subsection{Points of increase}\label{app:ptinc}

Recall the Definition \ref{inc-pts} of  points of increase of a nondecreasing function $f:\R\to\R$.

\begin{lem}\label{ptinc}
Let $f:\R\to\R$ be nondecreasing. Let $I$ denote the set of its points of increase.
\begin{enumerate} [label={\rm(\alph*)}, ref={\rm\alph*}] \itemsep=1pt 
\item\label{ptinc.a} $I$ is closed.
\item\label{ptinc.b} If furthermore $f$ is continuous, then $x\in I$ is right-isolated in $I$ if and only if there exists $y>x$ such that $f(y)=f(x)$. Similarly, $x\in I$ is left-isolated in $I$ if and only if there exists $z<x$ such that $f(z)=f(x)$.
\item\label{ptinc.c} If $f$ is continuous, then each point $x \in I$ is either not left-isolated in $I$ or not right-isolated in $I$ {\rm(}or both{\rm)}.
\end{enumerate}
\end{lem}


\begin{proof}
Part \eqref{ptinc.a}. If $x\in\R\setminus I$, then there exists $\delta>0$ such that $f$ is constant on $(x-\delta,x+\delta)$.  
Then every point $y$ in $(x-\delta,x+\delta)$ also belongs to $\R\setminus I$.  
This proves $I$ is closed.\smallskip

Part \eqref{ptinc.b}. We prove the first claim, the second being symmetric. Assume first that there exists $y>x$ with $f(y)=f(x)$. Since $f$ is nondecreasing, $f$ must be constant on the interval $[x,y]$. In particular there is a right neighborhood of $x$ on which $f$ is constant, so no point of that neighborhood belongs to $I$. Thus, $x$ is right-isolated in $I$.

Conversely, assume $x\in I$ is right-isolated. Then there exists $y>x$ such that
$(x,y)\cap I=\varnothing$.
Thus every point $u\in (x,y)$ is not a point of increase, so for each such $u$ there is an open interval $O_u\subset(x,y)$ containing $u$ on which $f$ is constant. Now define the set 
\[U=\{u\in(x,y):f(u)=f((x+y)/2)\}.\]
Since for each $u\in U$, $O_u\subset U$, we have that $U$ is open. On the other hand, the continuity of $f$ implies that $(x,y)\setminus U$ is also open.
Since $(x,y)=U\cup\bigl((x,y)\setminus U)$ and $(x,y)$ is connected, it must be that one of the two sets is empty.
However, $(x+y)/2\in U$. Therefore, it must be that $U=(x,y)$ and $f$ is constant on this interval. By the continuity of $f$, it is constant on $[x,y]$, i.e.\ $f(y)=f(x)$. Part \eqref{ptinc.b} is proved.\smallskip

Part \eqref{ptinc.c}. Suppose $x\in I$ is an isolated point in $I$. Then by part \eqref{ptinc.b}, there exist $z<x$ and $y>x$ such that $f(z)=f(x)=f(y)$. This implies $x\not\in I$, a contradiction.
\end{proof}

\subsection{Sierpi\'nski's theorem}\label{sec:sierpinski}

\begin{thm}\label{thm:sierpinski-interval}
Let $a<b$ and let $\mathbb I\subset \mathbb N$ be nonempty. Suppose that $[a,b]=\bigcup_{i\in\mathbb I} F_i$,
where the sets $F_i$ are pairwise disjoint, nonempty, and closed. Then $\mathbb I$ is a singleton.
\end{thm}

\begin{proof}
Suppose, for contradiction, that $\mathbb I$ has cardinality at least two.
We first prove the following claim. Let $I\subset [a,b]$ be a nonempty closed interval and let $\mathbb J\subset \mathbb I$ have cardinality at least two such that
$I\cap F_j \neq \varnothing$ for all $j\in \mathbb J$ 
and $I\cap F_j = \varnothing$ for all $j\in \mathbb I\setminus \mathbb J$.
Then there exist a closed interval $I'\subset I$ and $\mathbb J'\subset \mathbb J$, still of cardinality at least two, such that
\[
\inf \mathbb J \notin \mathbb J', \qquad
I'\cap F_j \neq \varnothing \ \ \forall j\in \mathbb J', \qquad
I'\cap F_j = \varnothing \ \ \forall j\in \mathbb I\setminus \mathbb J'.
\]

Indeed, let $j_0=\inf \mathbb J$ and choose $j_1\in \mathbb J\setminus\{j_0\}$. 
Pick $s\in I\cap F_{j_1}$ 
and let $t\in I\cap F_{j_0}$ minimize $|t-s|$. 
Then $t\neq s$.
Assume $s<t$, the case $s>t$ being similar. 
Then $[s,t)\subset I\setminus F_{j_0}$.
Since $t\notin F_{j_1}$, we may choose $r\in (s,t)\setminus F_{j_1}$. Set $I'=[s,r]\subset I$ and $\mathbb J'=\{j\in \mathbb J: I'\cap F_j\neq \varnothing\}\subset\mathbb J$. Then $I'\cap F_{j_0}=\varnothing$, so $j_0\notin \mathbb J'$, and $j_1\in \mathbb J'$, since $s\in I\cap F_{j_1}$. 
Since $r\in I'\subset I$ and $r\notin F_{j_1}$, we have $r\in F_{j_2}$ for some $j_2\in \mathbb J\setminus\{j_1\}$, 
hence $j_2\in \mathbb J'$.
Thus, $\mathbb J'$ has cardinality at least two, and the claim follows.

Starting from $I_1=[a,b]$ and $\mathbb J_1=\mathbb I$, apply the claim inductively to obtain nested nonempty closed intervals $I_1\supset I_2\supset\cdots$ and nested index sets $\mathbb J_1\supset \mathbb J_2\supset\cdots$, each with cardinality at least two and such that $j_n=\inf \mathbb J_n\notin \mathbb J_{n+1}$ for all $n\in\N$. In particular, $j_n$ is strictly increasing with $n$.

Since the intervals $I_n$ are nested and nonempty, their intersection is nonempty; pick $x$ in this intersection. Then $x\in F_k$ for some $k\in \mathbb I$. For all $n\in\N$, $x\in I_n$ implies $k\in \mathbb J_n$, so $j_n\le k$. This contradicts the fact that $j_n$ is strictly increasing.
\end{proof}

\section{Proofs of auxiliary results}\label{sec:auxlems}

\begin{lem}\label{lm:geo-ordering}
The following holds on a full $\P$-probability event. 
For all $s$, $x<y<z$, and 
$\tht$ in $\R$, 
and all $\sigg\in\{-,+\}$ and $S\in\{L,M,R\}$:
\begin{align}
\geo\from{(x,s)}\dir{R}{\tht}{\sig}\preceq\geo\from{(y,s)}\dir{L}{\tht}{\sig}
\preceq\geo\from{(y,s)}\dir{M}{\tht}{\sig}
\preceq\geo\from{(y,s)}\dir{R}{\tht}{\sig}
\preceq\geo\from{(z,s)}\dir{L}{\tht}{\sig}
\quad\text{and}\quad
\geo\from{(x,s)}\dir{S}{\tht}{-}\preceq\geo\from{(x,s)}\dir{S}{\tht}{+}.\label{geo:mono-aux}
\end{align}
\end{lem}

\begin{proof}
The first four inequalities are a consequence of the definitions of the geodesics $\geo\from{(x,s)}\dir{S}{\tht}{\sig}$, $S\in\{L,M,R\}$, the coalescence \eqref{geo:coal1}, and the continuity \eqref{geo:cont} of the paths. For $S\in\{L,R\}$, the last inequality in \eqref{geo:mono-aux} follows from \cite[(5.22)]{Bus-Sep-Sor-24}. We prove the last inequality for the case $S=M$. We distinguish three cases.

If $\geo\from{(x,s)}\dir{M}{\tht}{-}=\geo\from{(x,s)}\dir{L}{\tht}{-}$, then the already proved inequalities give $\geo\from{(x,s)}\dir{M}{\tht}{-}=\geo\from{(x,s)}\dir{L}{\tht}{-}\preceq\geo\from{(x,s)}\dir{L}{\tht}{+}\preceq\geo\from{(x,s)}\dir{M}{\tht}{+}$. Similarly, if $\geo\from{(x,s)}\dir{M}{\tht}{+}=\geo\from{(x,s)}\dir{R}{\tht}{+}$, then $\geo\from{(x,s)}\dir{M}{\tht}{-}\preceq\geo\from{(x,s)}\dir{R}{\tht}{-}\preceq\geo\from{(x,s)}\dir{R}{\tht}{+}=\geo\from{(x,s)}\dir{M}{\tht}{+}$.

Now suppose $\geo\from{(x,s)}\dir{L}{\tht}{-}\precneq\geo\from{(x,s)}\dir{M}{\tht}{-}$ and $\geo\from{(x,s)}\dir{M}{\tht}{+}\precneq\geo\from{(x,s)}\dir{R}{\tht}{+}$. In this case, by our convention when defining the geodesics \eqref{geodesics}, it must be that $\geo\from{(x,s)}\dir{L}{\tht}{-}\precneq\geo\from{(x,s)}\dir{M}{\tht}{-}\precneq\geo\from{(x,s)}\dir{R}{\tht}{-}$ and $\geo\from{(x,s)}\dir{L}{\tht}{+}\precneq\geo\from{(x,s)}\dir{M}{\tht}{+}\precneq\geo\from{(x,s)}\dir{R}{\tht}{+}$. Then by \eqref{no4stars}, there exists a $t>s$ such that 
\[\geo\from{(x,s)}\dir{L}{\tht}{-}(r)=\geo\from{(x,s)}\dir{L}{\tht}{+}(r)<\geo\from{(x,s)}\dir{M}{\tht}{-}(r)=\geo\from{(x,s)}\dir{M}{\tht}{+}(r)<\geo\from{(x,s)}\dir{R}{\tht}{-}(r)=\geo\from{(x,s)}\dir{R}{\tht}{+}(r)\]
for all $r\in[s,t]$. Then \eqref{no bubble} says that if $\geo\from{(x,s)}\dir{M}{\tht}{-}$ and $\geo\from{(x,s)}\dir{M}{\tht}{+}$ ever split, they will never reintersect. But the coalescence \eqref{geo:coal1}-\eqref{geo:coal2} and the ordering $\geo\from{(x,s)}\dir{L}{\tht}{-}\preceq\geo\from{(x,s)}\dir{L}{\tht}{+}$ imply that $\geo\from{(x,s)}\dir{M}{\tht}{-}(r)\le\geo\from{(x,s)}\dir{M}{\tht}{+}(r)$ for all sufficiently large $r$. Therefore, we must have  $\geo\from{(x,s)}\dir{M}{\tht}{-}\preceq\geo\from{(x,s)}\dir{M}{\tht}{+}$.
\end{proof}

\begin{proof}[Proof of Lemma \ref{lm:shockintallshocks}]
The inclusion $\subset$ comes from the definition of a $\tht\sigg$ shock interface. For the other direction, take $(x,s)\in\NU_1^{\tht\sig}$. Then, by \eqref{age}, $\geo\from{(x,s)}\dir{L}{\tht}{\sig}$ and $\geo\from{(x,s)}\dir{R}{\tht}{\sig}$ immediately separate and then reunite at time $s+\age^{\tht\sig}(x,s)$. Take any $t$ strictly between $s$ and $s+\age^{\tht\sig}(x,s)$. Then at this time, the two geodesics are still separated. Take $y$ such that $(y,t)$ is strictly between the two geodesics. Take any $\tht\sigg$ shock interface from $(y,t)$. By \eqref{no-intersection}, this interface cannot touch any of the two geodesics at a time $>s$ and, by continuity, has to pass through their starting point $(x,s)$.  Thus, $(x,s)$ lies on a $\tht\sigg$ shock interface.
\end{proof}

\begin{rmk}\label{rk:age}
Since the above construction produces a $\tht\sigg$ shock interface through $(x,s)$ for any $t\in(s,\,s+\age^{\tht\sig}(x,s))$, and since, by \eqref{shocksplitsgeo}, any such interface must remain strictly between $\geo\from{(x,s)}\dir{L}{\tht}{\sig}$ and $\geo\from{(x,s)}\dir{R}{\tht}{\sig}$ and cannot pass through their coalescence point at time $s+\age^{\tht\sig}(x,s)$, the $\tht\sigg$ shock at $(x,s)$ is said to have \emph{age} $\age^{\tht\sig}(x,s)$. See Figure \ref{fig:shock}, where the red path depicts the ``oldest'' shock interface, originating at the coalescence point and passing through $(x,s)$.
\end{rmk}

\begin{proof}[Proof of Lemma \ref{lm:geosplitsshocks}]
Take $r\in(s,t)$.
  By \eqref{geomaximizes}, $(y,t)=\gamma(t)$ maximizes the supremum 
  \[\sup_{y\in\R}\{\mathcal L(\gamma(r);y,t)+W^{\tht\sig}(y,t;\gamma(t))\}.\]
This says that both suprema in the definition \eqref{ddef} of $d_{\gamma(t)}(\gamma(r))$ (with boundary condition $f(y)=W^{\tht\sig}(y,t;\gamma(t))$ are achieved at $y$ such that $(y,t)=\gamma(t)$. This implies that $d_{\gamma(t)}(\gamma(r)) = 0$, which by \eqref{CIs} and the definition \eqref{Upsilondef} gives 
\[\Upsilon_{\gamma(t)}\dir{L}{\tht}{\sig}(r)\le\gamma(r)\le\Upsilon_{\gamma(t)}\dir{R}{\tht}{\sig}(r).\]
By \eqref{no-intersection}, $\tht\sigg$ shock interfaces cannot intersect $W^{\tht\sig}$-geodesics. This turns the weak inequalities into strict ones and proves the lemma.
\end{proof}

\begin{proof}[Proof of Theorem \ref{th:convanydir}]
Take $\Omega_0$ to be the intersection of $\Omega_{15}$ from Section \ref{sec:cif+shock} (page \pageref{Om15}) and the full $\P$-probability event in Theorem 1.18 in \cite{Bat-Gan-Ham-22}. Fix $\w \in \Omega_0$. 
There exists an $S' \in \{L,M,R\}$ such that $S_n = S'$ for infinitely many $n$. 
Restricting to this subsequence, we may assume without loss of generality that $S_n = S'$ for all $n$. 
By the coalescence \eqref{geo:coal1} and the continuity of the paths, we have that for any integers $m$ and $n$, either  
$\geo\from{(x_n,s_n)}\dir{S'}{\tht}{\sig}\preceq\geo\from{(x_m,s_m)}\dir{S'}{\tht}{\sig}$ or $\geo\from{(x_n,s_n)}\dir{S'}{\tht}{\sig}\succeq\geo\from{(x_m,s_m)}\dir{S'}{\tht}{\sig}$.
Therefore, we can pass to a subsequence along which $\geo\from{(x_n,s_n)}\dir{S'}{\tht}{\sig}$ is monotone. This implies that the limit of $\geo\from{(x_n,s_n)}\dir{S'}{\tht}{\sig}(r)$, as $n\to\infty$, exists (in $[-\infty,\infty]$) for each $r>s$. Denote this limit by $\gamma(r)$. We show that
\begin{align}\label{aux3226}
\forall r>s:\ 
\geo\from{(x,s)}\dir{L}{\tht}{\sig}(r)
\le\gamma(r)
\le\geo\from{(x,s)}\dir{R}{\tht}{\sig}(r).
\end{align}

Take a sequence $\underline x_m$ that is strictly increasing to $x$ and a sequence $\overline x_m$ that strictly decreases to $x$. We inductively define a sequence of times $r_m$ that strictly increases to $s$. 

By the continuity of the $\tht\sigg$ shock interfaces,
there exists a time $r_1<s$ such that $\Upsilon\from{(\underline x_1,s)}\dir{L}{\tht}{\sig}(r_1)<\Upsilon\from{(\overline x_1,s)}\dir{R}{\tht}{\sig}(r_1)$. Then the coalescence \eqref{Itree} implies $\Upsilon\from{(\underline x_1,s)}\dir{L}{\tht}{\sig}\big|_{[r_1,s]}
\prec\Upsilon\from{(\overline x_1,s)}\dir{R}{\tht}{\sig}\big|_{[r_1,s]}$.

Using again the coalescence \eqref{Itree} and the continuity of the $\tht\sigg$ shock interfaces, we can inductively find times $r_m>r_{m-1}$ such that $s-1/m<r_m<s$ and  
\begin{align}\label{3227}
\Upsilon\from{(\underline x_{m-1},s)}\dir{L}{\tht}{\sig}\big|_{[r_m,s]}\prec\Upsilon\from{(\underline x_m,s)}\dir{L}{\tht}{\sig}\big|_{[r_m,s]}\prec
\Upsilon\from{(\overline x_m,s)}\dir{R}{\tht}{\sig}\big|_{[r_m,s]}\prec\Upsilon\from{(\overline x_{m-1},s)}\dir{R}{\tht}{\sig}\big|_{[r_m,s]},
\end{align}
for all $m\ge2$.
Abbreviate $(\underline z_m,r_m)=\Upsilon\from{(\underline x_m,s)}\dir{L}{\tht}{\sig}(r_m)$ and $(\overline z_m,r_m)=\Upsilon\from{(\overline x_m,s)}\dir{R}{\tht}{\sig}(r_m)$.

By \eqref{shocksplitsgeo}, $\geo\from{(\underline z_{m},r_{m})}\dir{L}{\tht}{\sig}$ remains left of $\Upsilon\from{(\underline x_{m},s)}\dir{L}{\tht}{\sig}$ and $\geo\from{(\overline z_{m},r_{m})}\dir{R}{\tht}{\sig}$ goes right of $\Upsilon\from{(\overline x_{m},s)}\dir{L}{\tht}{\sig}$. Therefore, by the middle inequality in \eqref{3227}, the former geodesic is strictly left of the latter on $[r_m,s]$ and, by continuity, the region strictly between the two is an open subset of $\R^2$. Since $\underline x_{m}<x<\overline x_{m}$, this region contains $(x,s)$. Thus, there exists $n_m$ such that $(x_n,s_n)$ is in this region, for all $n\ge n_m$. By the coalescence \eqref{geo:coal1} and the continuity of the paths, we have $\geo\from{(\underline z_{m},r_{m})}\dir{L}{\tht}{\sig}\preceq\geo\from{(x_n,s_n)}\dir{S'}{\tht}{\sig}
\preceq\geo\from{(\overline z_{m},r_{m})}\dir{R}{\tht}{\sig}$,
for all $n\ge n_m$. Taking $n\to\infty$, we get, for all integers $m\ge1$ and all $r>s$,
\begin{align}\label{aux3249}
\geo\from{(\underline z_{m},r_{m})}\dir{L}{\tht}{\sig}(r)\le\gamma(r)
\le\geo\from{(\overline z_{m},r_{m})}\dir{R}{\tht}{\sig}(r).
\end{align}

On the other hand, the first and last inequalities in \eqref{3227} and the duality \eqref{no-intersection} imply that 
\[(\underline x_{m-1},s)\le \geo\from{(\underline z_{m},r_{m})}\dir{L}{\tht}{\sig}(s)\quad\text{and}\quad
\geo\from{(\overline z_{m},r_{m})}\dir{R}{\tht}{\sig}(s)\le(\overline x_{m-1},s).\]
Then, \eqref{geo:restart} and the ordering \eqref{geo:mono} give, for any $r>s$,
\[\geo\from{(\underline x_{m-1},s)}\dir{L}{\tht}{\sig}(r)\le \geo\from{(\underline z_{m},r_{m})}\dir{L}{\tht}{\sig}(r)\quad\text{and}\quad
\geo\from{(\overline z_{m},r_{m})}\dir{R}{\tht}{\sig}(r)\le\geo\from{(\overline x_{m-1},s)}\dir{R}{\tht}{\sig}(r).\]
Combining this with \eqref{aux3249}, then taking $m\to\infty$ and using the limits \eqref{geo:lim} proves \eqref{aux3226}. 

Using the monotonicity of $\geo\from{(x_n,s_n)}\dir{S'}{\tht}{\sig}$, the restart property \eqref{geo:restart}, and the limits \eqref{geo:lim}, we obtain that for any $r>s$, the limit $\gamma|_{[r,\infty)}$ of $\geo\from{(x_n,s_n)}\dir{S'}{\tht}{\sig}\big|_{[r,\infty)}$ is a $W^{\tht\sig}$-geodesic. In particular, $\gamma$ is continuous on $(s,\infty)$. By \eqref{aux3226}, $\gamma(r)\to(x,s)$ as $r\searrow s$, so defining $\gamma(s)=(x,s)$ extends $\gamma$ to a continuous space-time path on $[s,\infty)$.
Since $\gamma|_{[r,\infty)}$ is a $W^{\tht\sig}$-geodesic for every $r>s$, and both $\mathcal L$ and $W^{\tht\sig}$ are continuous, the defining property of a $W^{\tht\sig}$-geodesic (see \eqref{rec}) extends to the entire path $\gamma:[s,\infty)\to\R^2$. Hence $\gamma=\geo\from{(x,s)}\dir{S}{\tht}{\sig}$ for some $S\in\{L,M,R\}$.
Furthermore, by Dini's theorem \cite[Theorem 7.13]{Rud-76}, the convergence of $\geo\from{(x_n,s_n)}\dir{S'}{\tht}{\sig}$ to $\gamma=\geo\from{(x,s)}\dir{S}{\tht}{\sig}$ is uniform on $[r,t]$, for any $t>r>s$. 

Fix $t>r>s$. Take $K>0$ large enough so that $\geo\from{(x,s)}\dir{S}{\tht}{\sig}(r)$ and $\geo\from{(x,s)}\dir{S}{\tht}{\sig}(t)$ are in $(-K,K)\times\{r\}$ and $(-K,K)\times\{t\}$, respectively.
Then the same holds for $\geo\from{(x_n,s_n)}\dir{S'}{\tht}{\sig}(r)$ and $\geo\from{(x_n,s_n)}\dir{S'}{\tht}{\sig}(t)$, once $n$ is large enough.
Since $\{(a,r,b,t)\in\R^4:a,b\in[-K,K]\}$ is a compact subset of $\{(a,s',b,s'')\in\R^4:s'<s''\}$, we get from  Theorem 1.18 in \cite{Bat-Gan-Ham-22} that there exists $\e > 0$ such that for any 
$a_1, a_2, b_1, b_2 \in [-K,K]$, 
any geodesics from $(a_i,r)$ to $(b_i,t)$, $i=1,2$, that remain within distance $\e$ of each other must intersect. 
Applying this to 
$\bigl(\geo\from{(x_n,s_n)}\dir{S'}{\tht}{\sig}(r),
\geo\from{(x_n,s_n)}\dir{S'}{\tht}{\sig}(t)\bigr)$
and
$\bigl(\geo\from{(x,s)}\dir{S}{\tht}{\sig}(r),
\geo\from{(x,s)}\dir{S}{\tht}{\sig}(t)\bigr)$,
we find that if
\begin{align}\label{1156}
\sup_{t' \in [r,t]} 
\bigl|\geo\from{(x_n,s_n)}\dir{S'}{\tht}{\sig}(t')
- \geo\from{(x,s)}\dir{S}{\tht}{\sig}(t')\bigr|
< \e,
\end{align}
then $\geo\from{(x_n,s_n)}\dir{S'}{\tht}{\sig}$ 
intersects $\geo\from{(x,s)}\dir{S}{\tht}{\sig}$ 
at some $t' \in [r,t]$. 
By the coalescence property \eqref{geo:coal1}, 
the two geodesics then coincide on $[t',\infty)$, and hence on $[t,\infty)$.
By the aforementioned uniform convergence, there exists $n_0$ such that 
\eqref{1156} holds for all $n \ge n_0$. 
The theorem is proved.
\end{proof}

\begin{proof}[Proof of Lemma \ref{lm:shocksandwich}]
The ordering claim is trivial for $r=s$ since all interfaces emanate from $(x,s)$. Suppose $r<s$. 
Abbreviate $\gamma^S=\geo\from{\tau(r)}\dir{S}{\tht}{\sig}$, $S\in\{L,R\}$. Since these are both geodesics, \eqref{geomaximizes} says that
\[\sup_{y\in\R}\bigl\{\mathcal L(\tau(r);y,s)+W^{\tht\sig}(y,s;\tau(r))\bigr\}\]
is maximized when $(y,s)\in\{\gamma^L(s),\gamma^R(s)\}$.
By \eqref{shocksplitsgeo} and the continuity of the paths, $\gamma^L(s)\le(x,s)\le\gamma^R(s)$. Thus, the two suprema in the definition \eqref{ddef} of $d_{(x,s)}(\tau(r))$ are equal to the one in the above display and, consequently, $d_{(x,s)}(\tau(r))=0$. The inequalities \eqref{shocksmono} now follow from \eqref{CIs} and the definition \eqref{Upsilondef} of the left and right shock interfaces.

Directedness \eqref{shocksdir} follows from the ordering \eqref{shocksmono} and the directedness \eqref{Upsilondir}.
\end{proof}

\begin{proof}[Proof of Lemma \ref{lm:shockasymorder}]
The inequality \eqref{+<-} is trivially satisfied when $r=s$ because the interfaces start at $(x,s)$ and $(y,s)$ and $x<y$.
Suppose, for a contradiction, that there exists an $r<s$ such that $\tau^+(r)>\tau^-(r)$. Buy the continuity of the paths, the two interfaces must 
intersect at some $(u,t)$ with $t\in(r,s)$.
Take $z$ strictly between $\tau^-(r)$ and $\tau^+(r)$. By the ordering \eqref{geo:mono}, $\geo\from{(z,r)}\dir{L}{\tht}{-}(t)\le\geo\from{(z,r)}\dir{R}{\tht}{+}(t)$.
By \eqref{no-intersection} and the continuity of the paths, $\geo\from{(z,r)}\dir{L}{\tht}{-}(t)\ge\tau^-(t)$ and $\geo\from{(z,r)}\dir{R}{\tht}{+}(t)\le\tau^+(t)$. Now we have 
\[\tau^-(t)=\geo\from{(z,r)}\dir{L}{\tht}{-}(t)=\geo\from{(z,r)}\dir{R}{\tht}{+}(t)=\tau^+(t)=(u,t).\]
This contradicts \eqref{no-intersection}, which says that the relative interiors of geodesics and shock interfaces do not intersect. The inequality \eqref{+<-} is thus proved. The claim \eqref{+<=-} follows from \eqref{+<-} and the fact that, by \eqref{Itree}, $\sigg-$ shock interfaces all coalesce and $\sigg+$ shock interfaces all coalesce.
\end{proof}

\begin{proof}[Proof of Lemma \ref{lm:shocklims}]
\eqref{Icv} and the definition \eqref{Upsilondef} imply the existence of a sequence $z_n$ that strictly increases to $x$ and satisfies
\[\forall r\le t\,:\,\Upsilon_{(x,s)}\dir{L}{\tht}{\sig}(r)=\lim_{n\to\infty}\Upsilon_{(z_n,s)}\dir{R}{\tht}{\sig}(r).
\]
For $z$ strictly between $z_{n-1}$ and $z_{n+1}$ (which includes $z=z_n$), the continuity of the paths and the coalescence due to \eqref{Itree} gives that $\Upsilon_{(z_{n-1},s)}\dir{R}{\tht}{\sig}\preceq \Upsilon_{(z,s)}\dir{S}{\tht}{\sig}\preceq\Upsilon_{(z_{n+1},s)}\dir{R}{\tht}{\sig}$. This and the above convergence imply the first limit in \eqref{shocklims}. The second limit comes analogously.
\end{proof}

\begin{proof}[Proof of Lemma \ref{lm:intcoal}]
    Since the two interfaces coalesce at time $s<t\wedge t'$, they cannot be equal at any time in $(s,t\wedge t')$. By the continuity of the paths, one of the two interfaces has to be strictly to the left of the other, over this time interval. Renaming the interfaces, if necessary, we can assume that $\tau(r)<\tau'(r)$ for all $r\in(s,t\wedge t')$.

   Let $s_n$ be a sequence in $(s,t\wedge t')$ that strictly decreases to $s$. Let $x_n$ be such that for each $n$, $\tau(s_n)<x_n<\tau'(s_n)$. By Theorem \ref{th:convanydir}, there exists a subsequence, denoted again by $n$, such that $\geo\from{(x_n,s_n)}\dir{L}{\tht}{\sig}$ converges, in the overlap topology, to a $W^{\tht\sig}$-geodesic $\gamma$ out of $(x,s)$.
   
   By \eqref{no-intersection}, for each $n$, $\geo\from{(x_n,s_n)}\dir{L}{\tht}{\sig}$ has to remain strictly between $\tau$ and $\tau'$ over the time interval $[s_n,t\wedge t')$. Consequently, the limiting geodesic $\gamma$ has to be weakly between $\tau$ and $\tau'$ on the time interval $[s,t\wedge t')$. Applying \eqref{no-intersection} once again implies that $\gamma$ is in fact strictly between $\tau$ and $\tau'$, on $(s,t\wedge t')$. We have proved the existence of a geodesic $\gamma$ that goes strictly between $\tau$ and $\tau'$. We next prove its uniqueness.

   By \eqref{shocksplitsgeo}, $\geo\from{(x,s)}\dir{R}{\tht}{\sig}$ goes strictly right of $\tau'$ on the interval $(s,t\wedge t')$. Similarly, $\geo\from{(x,s)}\dir{L}{\tht}{\sig}$ goes strictly left of $\tau$ on the interval $(s,t\wedge t')$. Consequently, $\gamma$ is different from both $\geo\from{(x,s)}\dir{S}{\tht}{\sig}$, $S\in\{L,R\}$, and, since $\gamma$ is $W^{\tht\sig}$-geodesic, we are left with the unique remaining option,  $\gamma=\geo\from{(x,s)}\dir{M}{\tht}{\sig}$.

   Lastly, if there were three distinct $\tht\sigg$ shock interfaces $\tau\preceq\tau'\preceq\tau''$ all coalescing exactly at $(x,s)$, then we would have one geodesic out of $(x,s)$ that goes strictly between $\tau$ and $\tau'$ and another that would go strictly between $\tau'$ and $\tau''$, contradicting the uniqueness, we just proved, of the geodsic that goes strictly between $\tau$ and $\tau''$.
\end{proof}

\begin{proof}[Proof of Lemma \ref{lm:geocoal}]
  Since the two geodesics coalesce at time $s>r\vee r'$, they cannot be equal at any time in $(r\vee r',s)$. By the continuity of the paths, one of the two geodesics has to be strictly left of the other over this time interval. Renaming the geodesics, if necessarly, we can assume that $\gamma(t)<\gamma'(t)$ for all $t\in(r\vee r',s)$.

  Let $s_n\in(r\vee r',s)$ be strictly increasing to $s$. 
  Since $\gamma(s_n)<\gamma'(s_n)$, Lemma \ref{lm:geosplitsshocks} and \eqref{no-intersection} imply 
  \begin{align}\label{2276}
  \gamma(t)<\Upsilon\from{\gamma(s_n)}\dir{R}{\tht}{\sig}(t)<\gamma'(t)\quad\forall t\in(r\vee r',s_n).
  \end{align}
  
  Applying \eqref{2276} with $t=s_{n-1}$ we get that for all $n>1$,  $\Upsilon\from{\gamma(s_n)}\dir{R}{\tht}{\sig}(s_{n-1})>\gamma(s_{n-1})$ and then the coalescence \eqref{Itree} of $\Upsilon\from{\gamma(s_n)}\dir{R}{\tht}{\sig}$ and $\Upsilon\from{\gamma(s_{n-1})}\dir{R}{\tht}{\sig}$, together with continuity, implies that we must have $\Upsilon\from{\gamma(s_n)}\dir{R}{\tht}{\sig}\succeq\Upsilon\from{\gamma(s_{n-1})}\dir{R}{\tht}{\sig}$. 
 This and the second inequality in \eqref{2276} say that for any $t\in(r\vee r',s)$, once $n$ is large enough so that $s_n>t$, the sequence $\Upsilon\from{\gamma(s_n)}\dir{R}{\tht}{\sig}(t)$ is nondecreasing and bounded. 
It therefore converges to a finite limit, which we denote by $\tau(t)$.

    By the coalescence \eqref{Itree}, when $s_n>t$, $\Upsilon\from{\gamma(s_n)}\dir{R}{\tht}{\sig}\big|_{(-\infty,t]}=\Upsilon\from{\Upsilon\from{\gamma(s_n)}\dir{R}{\tht}{\sig}(t)}\dir{R}{\tht}{\sig}$. Then \eqref{shocklims} implies that these interfaces converge, pointwise, to $\Upsilon\from{\tau(t)}\dir{L}{\tht}{\sig}$. This implies that $\tau|_{(r\vee r',t]}= \Upsilon\from{\tau(t)}\dir{L}{\tht}{\sig}\big|_{(r\vee r',t]}$, for all $t\in(r\vee r',s)$. In particular, $\tau$ is continuous on $(r\vee r',s)$. Furthermore, it can be extended continuously to $(-\infty,s)$ by setting $\tau(s')=\Upsilon\from{\tau(t)}\dir{L}{\tht}{\sig}(s')$ for $s'\le r\vee r'$. Thus, for any $t\in(r\vee r',s)$, the extension $\tau|_{(-\infty,t]}$ is a $\tht\sigg$ shock interface.

     Taking $n\to\infty$ in \eqref{2276}, we get that 
\begin{align}\label{2285}
\gamma(t)\le\tau(t)\le\gamma'(t)\quad\forall t\in(r\vee r',s).
\end{align}
Since both $\gamma(t)$ and $\gamma'(t)$ converge to $(x,s)$ as $t\nearrow s$, we get that $\tau$ can be further extended continuously to $(-\infty,s]$ by setting $\tau(s)=(x,s)$.  Now, $\tau$ is a $\tht\sigg$ shock interface out of $(x,s)$. Furthermore, \eqref{2285} shows that $\tau$ remains weakly between $\gamma$ and $\gamma'$ on $(r\vee r',s)$ and, since we have established it is a $\tht\sigg$ shock interface, \eqref{no-intersection} implies that $\tau$ in fact remains strictly between the two geodesics.
\end{proof}

\begin{proof}[Proof of Lemma \ref{lem:shocksdense}]
Take $x<y$. By \eqref{geo:coal2}, the geodesics 
$\geo\from{(x,s)}\dir{R}{\tht}{\sig}$ and $\geo\from{(y,s)}\dir{L}{\tht}{\sig}$ 
coalesce at some time $t>s$. Choose $z$ 
such that $(z,\tfrac{s+t}2)$ is 
strictly between 
$\geo\from{(x,s)}\dir{R}{\tht}{\sig}(\tfrac{s+t}2)$ and 
$\geo\from{(y,s)}\dir{L}{\tht}{\sig}(\tfrac{s+t}2)$. 
By \eqref{no-intersection} and the continuity of the paths, the  interface $\Upsilon_{(z,\frac{s+t}{2})}^{L,\tht\sig}$ must remain between these two geodesics and therefore intersects the segment $[x,y]\times\{s\}$. 
The intersection point is a shock belonging to $\NU_1^{\tht\sig}\cap([x,y]\times\{s\})$.
\end{proof}

\bibliographystyle{aop-no-url}
\bibliography{firasbib2010}

\end{document}